\documentclass[reqno, a4paper]{amsart}

\usepackage[T1]{fontenc}
\usepackage[utf8]{inputenc}
\usepackage[british]{babel}
\usepackage{enumitem}
\usepackage{amsfonts, amssymb, amsthm, amsmath}
\usepackage[all]{xy}
\usepackage{mathbbol}
\usepackage{booktabs}
\let\CMcal\mathcal
\usepackage{calrsfs}
\usepackage{graphicx}
\usepackage{xcolor}
\usepackage{tikz-cd}
\usetikzlibrary{calc}
\usepackage{mathtools}
\usepackage[colorlinks, citecolor=blue]{hyperref}

\newenvironment*{eqdiagr}{%
\begin{equation}%
}{%
  \end{equation}\ignorespacesafterend%
}
\usepackage{mathabx}

\mathtoolsset{mathic}

\theoremstyle{plain}
\newtheorem{lemma}[subsection]{Lemma}
\newtheorem{proposition}[subsection]{Proposition}
\newtheorem{theorem}[subsection]{Theorem}
\newtheorem{corollary}[subsection]{Corollary}

\theoremstyle{remark}
\newtheorem{example}[subsection]{Example}
\newtheorem{remark}[subsection]{Remark}

\theoremstyle{definition}
\newtheorem{definition}[subsection]{Definition}

\allowdisplaybreaks

\makeatletter
\newcommand*{\comment}[4][]{%
  \@ifundefined{c@#2}{\newcounter{#2}}{}%
  \begingroup
  \ifblank{#1}{%
    \newcommand*{\header}{\csname the#2\endcsname}%
  }{%
    \newcommand*{\header}{(#1\csname the#2\endcsname)}%
  }%
  \addtocounter{#2}{1}%
  \textnormal{\textcolor{#3}{\header.[#4\@]}}%
  \endgroup
}
\makeatother

\newlist{tfae}{enumerate}{1}
\setlist[tfae]{label=\textnormal{(\roman*)}}

\newcommand*{\CondName}[1]{\textrm{(#1)}}
\newcommand*{\LACC}{\CondName{LACC}}

\newcommand{\defn}{\textbf}
\newcommand{\del}{\partial}
\newcommand{\comp}{\raisebox{0.2mm}{\ensuremath{\scriptstyle{\circ}}}}
\newcommand*{\id}[1]{1_{#1}}
\newcommand*{\from}{\colon}
\newcommand{\join}{\vee}
\newcommand{\meet}{\wedge}
\newcommand*{\biprod}{\oplus}
\newcommand{\tensor}{\otimes}
\newcommand{\subobj}{\leq}
\newcommand{\supobj}{\geq}
\newcommand{\links}{\lgroup}
\newcommand{\rechts}{\rgroup}
\newcommand{\To}{\Rightarrow}

\newcommand{\point}{\rightleftarrows}
\newcommand{\cosmash}{\ensuremath{\diamond}}
\newcommand{\cre}{\ensuremath{\mathrm{cr}}}
\newcommand{\talf}{\text{\reflectbox{\(\flat\)}}}

\renewcommand{\H}{\ensuremath{\mathrm{H}}}
\renewcommand{\Im}{\ensuremath{\mathrm{Im}}}
\DeclareMathOperator{\gl}{\mathfrak{gl}}
\DeclareMathOperator{\Aut}{Aut}
\DeclareMathOperator{\Der}{Der}
\DeclareMathOperator{\End}{End}
\DeclareMathOperator{\Hom}{Hom}
\DeclareMathOperator{\ev}{ev}

\DeclareMathOperator{\Ker}{Ker}
\DeclareMathOperator{\Coker}{Coker}
\DeclareMathOperator{\ab}{Ab}
\DeclareMathOperator{\gp}{Gp}
\DeclareMathOperator{\nIL}{NIL}
\newcommand{\nil}{\mathrm{Nil}}
\newcommand{\im}{\mathrm{im}}
\DeclareMathOperator{\nilunit}{nil}
\DeclareMathOperator{\tw}{tw}
\renewcommand{\S}{\ensuremath{\mathfrak{S}}}
\newcommand{\DefEq}{\coloneq} 
\newcommand{\noproof}{\hfil\qed}
\newcommand{\normal}{\mathrel{\lhd}}

\newcommand{\y}{\ensuremath{\mathbf{y}}}

\newcommand{\ctgry}[1]{\mathcal{#1}}
\newcommand{\A}{\ensuremath{\ctgry{A}}}
\newcommand{\B}{\ensuremath{\ctgry{B}}}
\newcommand{\C}{\ensuremath{\ctgry{C}}}
\newcommand{\D}{\ensuremath{\ctgry{D}}}
\newcommand{\V}{\ensuremath{\ctgry{V}}}
\newcommand{\X}{\ensuremath{\ctgry{X}}}
\newcommand{\Y}{\ensuremath{\ctgry{Y}}}

\renewcommand{\P}{\ensuremath{\CMcal{P}}}
\newcommand{\T}{\ensuremath{\CMcal{T}}}

\newcommand{\Z}{\ensuremath{\mathbb{Z}}}

\newcommand{\Ab}{\ensuremath{\mathsf{Ab}}}
\newcommand{\Alg}{\ensuremath{\mathsf{Alg}}}
\newcommand{\Algg}[1]{\Alg\text{-}#1}
\newcommand{\PAlg}{\ensuremath{\Algg\P}}
\newcommand{\CAlg}{\ensuremath{\mathsf{CAlg}}}
\newcommand{\CRng}{\ensuremath{\mathsf{CRng}}}
\newcommand{\CURng}{\ensuremath{\mathsf{CRing}}}
\newcommand{\CUAlg}{\ensuremath{\mathsf{UCAlg}}}
\newcommand{\CStarAlg}{\ensuremath{\text{\(\mathsf{C}^{*}\)-\(\mathsf{Alg}\)}}}
\newcommand{\Ext}{\ensuremath{\mathsf{Ext}}}
\newcommand{\Gp}{\ensuremath{\mathsf{Gp}}}
\newcommand{\HSLat}{\ensuremath{\mathsf{HSLat}}}
\newcommand{\Lie}{\ensuremath{\mathsf{Lie}}}
\newcommand{\Leib}{\ensuremath{\mathsf{Leib}}}
\newcommand{\Loop}{\ensuremath{\mathsf{Loop}}}
\newcommand{\Mod}{\ensuremath{\mathsf{Mod}}}
\newcommand{\Nil}{\ensuremath{\mathsf{Nil}}}
\newcommand{\NIL}{\ensuremath{\mathsf{NIL}}}
\newcommand{\Pt}{\ensuremath{\mathsf{Pt}}}
\newcommand{\RG}{\ensuremath{\mathsf{RG}}}
\newcommand{\Sh}{\ensuremath{\mathsf{Sh}}}
\newcommand{\XMod}{\ensuremath{\mathsf{XMod}}}
\DeclarePairedDelimiterX{\IndArr}[1]{(}{)}{
  \renewcommand*{\and}{,}
  #1
}
\DeclarePairedDelimiterX{\CoindArr}[1]{\langle}{\rangle}{
  \renewcommand*{\and}{,}
  #1
}

\newcommand*{\DIndArr}[1]{
  \begingroup
  \renewcommand*{\and}{\\}
  \left\links\begin{smallmatrix}
    #1
  \end{smallmatrix}\right\rechts
  \endgroup
}
\newcommand*{\DCoindArr}[1]{
  \begingroup
  \renewcommand*{\and}{&}
  \left\links\begin{smallmatrix}
    #1
  \end{smallmatrix}\right\rechts
  \endgroup
}

\newcommand*{\mono}{\rightarrowtail}
\tikzcdset{
  mono/.style=tail
}
\tikzcdset{
  nmono/.style=Triangle[{open, reversed}]->
}
\tikzcdset{
  repi/.style=-Triangle[open]
}
\tikzcdset{
  induced/.style=dashed
}

\def\pullback{
 \ar@{-}[]+R+<6pt,-1pt>;[]+RD+<6pt,-6pt>%
 \ar@{-}[]+D+<1pt,-6pt>;[]+RD+<6pt,-6pt>}
 
\def\pushout{%
 \ar@{-}[]+L+<-6pt,1pt>;[]+LU+<-6pt,6pt>%
 \ar@{-}[]+U+<-1pt,6pt>;[]+LU+<-6pt,6pt>}

\def\splitpullback{%
 \ar@{-}[]+R+<6pt,-.51ex>;[]+RD+<6pt,-6pt>%
 \ar@{-}[]+D+<.51ex,-6pt>;[]+RD+<6pt,-6pt>}

\def\skewpullback{%
 \ar@{-}[]+LD+<-6pt,-6pt>;[]+LDD+<-6pt,-15.5pt>%
 \ar@{-}[]+D+<-1pt,-6pt>;[]+LDD+<-6pt,-15.5pt>}

\newdir{>>}{{}*!/3.5pt/:(1,-.2)@^{>}*!/3.5pt/:(1,+.2)@_{>}*!/7pt/:(1,-.2)@^{>}*!/7pt/:(1,+.2)@_{>}}
\newdir{ >>}{{}*!/8pt/@{|}*!/3.5pt/:(1,-.2)@^{>}*!/3.5pt/:(1,+.2)@_{>}}
\newdir{ |>}{{}*!/-3.5pt/@{|}*!/-8pt/:(1,-.2)@^{>}*!/-8pt/:(1,+.2)@_{>}}
\newdir{ >}{{}*!/-8pt/@{>}}
\newdir{>}{{}*:(1,-.2)@^{>}*:(1,+.2)@_{>}}
\newdir{<}{{}*:(1,+.2)@^{<}*:(1,-.2)@_{<}}

\tikzcdset{
  mono/.style=tail
}
\tikzcdset{
  nmono/.style=Triangle[{open, reversed}]->
}
\tikzcdset{
  repi/.style=-Triangle[open]
}
\tikzcdset{
  squared/.style={sep={#1em,between origins}},
  squared/.default=4.5
}
\DeclareMathDelimiter\AMSlrcorner{\mathclose}{AMSa}{"79}{AMSa}{"79}

\newcommand*{\pb}[2][.2]{
  \arrow[#2, to path={
    let
    \p0 = ($(\tikztotarget.center)-(\tikztostart.center)$)
    in
    (\tikztostart.center) -- +(-#1*\y0,#1*\y0) node{\Large{\(\AMSlrcorner\)}}
  }, phantom]
}

\begin{document}

\title[Intrinsic tensor products]{Intrinsic tensor products and a Ganea-type extension of the five-term exact sequence}

\dedicatory{Dedicated to Trueman MacHenry}

\author{Bo Shan Deval}
\author{Manfred Hartl}
\author{Tim Van~der Linden}

\email{bo.deval@uclouvain.be}
\email{mhartl@gmx.fr}
\email{tim.vanderlinden@uclouvain.be}

\address[Manfred Hartl]{Northeast Normal University, Changchun, Jilin, China}

\address[Manfred Hartl]{Universit\'e Polytechnique Hauts-de-France, CERAMATH and FR CNRS 2956,   F-59313~Valenciennes, France}

\address[Manfred Hartl]{Mathematics \& Data Science, Vrije Universiteit Brussel, Pleinlaan 2, B--1050 Brussel, Belgium}

\address[Bo Shan Deval, Tim Van der Linden]{Institut de Recherche en Mathématique et Physique, Université catholique de Louvain, chemin du cyclotron 2 bte L7.01.02, B--1348 Louvain-la-Neuve, Belgium}

\address[Manfred Hartl, Tim Van der Linden]{Mathematics \& Data Science, Vrije Universiteit Brussel, Pleinlaan 2, B--1050 Brussel, Belgium}

\address[Tim Van der Linden]{CMUC\@, University of Coimbra, 3001--454 Coimbra, Portugal}

\thanks{The first author's research is supported by a grant of the Fund for Research Training in Industry and Agriculture (FRIA) of the Fonds de la Recherche Scientifique--FNRS\@.}
\thanks{The second author wishes to thank the Centro de Matem\'atica da Universidade de Coimbra, the Institut de Recherche en Math\'ematique et Physique, the Max-Planck-Institut f\"ur Ma\-the\-ma\-tik, the Chern Institute of Mathematics, the CINVESTAV, and the Departments of Mathematics of the Universities of La Rioja and S\~ao Paulo (supported by FAPESP, project 2015/07917-2) and of the North East Normal University for their kind hospitality and ideal working conditions provided during his stay at Coimbra, Louvain-la-Neuve, Bonn, Tianjin, Mexico City, Logro\~no, S\~ao Paulo and his employment at Changchun, respectively\@.}
\thanks{The third author is a Research Director of the Fonds de la Recherche Scientifique--FNRS\@. His research was partially supported by Centro de Matem\'atica da Universidade de Coimbra and by Funda\c c\~ao para a Ci\^encia e a Tecnologia (grant number SFRH/BPD/38797/2007). He began studying some of the preliminary material during a stay at York University in 2007 supported by Walter Tholen's Natural Sciences and Engineering Research Council of Canada Discovery Grant no.\ 501260. He wishes to thank the Max-Planck-Institut f\"ur Ma\-the\-ma\-tik and Universidad de Vigo for their kind hospitality during his stays at Bonn and at Pontevedra.}

\date{\today}



\begin{abstract}
  We define an intrinsic symmetric bi-right-exact (and for varieties, bi-cocontinuous) bilinear product on objects of a semi-abelian category, constructed as the cosmash product in the two-nilpotent reflection. When applied to abelian objects, this recovers classical tensor products in many cases. A recognition theorem states that any symmetric bi-cocontinuous bifunctor on an abelian variety of algebras is realised as the bilinear product in the variety of algebras over a suitable \(2\)-nilpotent symmetric operad in the monoidal category of abelian groups. For abelian groups replaced with any commutative ring, the bilinear product of algebras over such an operad is associative as long as the only unary operations are given by multiplication with scalars, but not in general.

  This relies on a right-exactness theorem for cross-effects of bifunctors, and consequently for cosmash products. We develop basic properties, compare the bilinear product to the Brown--Loday non-abelian tensor product, and prove a categorical version of Ganea's six-term exact homology sequence. We further characterise abelian extensions via internal action cores, obtaining explicit descriptions of bilinear products in categories of representations; in particular, the bilinear product of the associated Beck modules generalises the classical tensor product of representations for groups and Lie algebras.
\end{abstract}

\maketitle

\section*{Introduction}\label{Section Intro}
In a pointed variety of universal algebras \(\V\), let the expression \(t(x_1,\dots, x_m)\) denote a term with distinct variables \(x_1\), \dots, \(x_m\). Let \(0\) denote the unique constant in the theory of \(\V\). In particular, this implies that \(t(0,\dots, 0)=0\). We say that \(t(x_1,\dots, x_m)\) is a \defn{commutator term} if \(t(x_1,\dots, x_m)=0\) whenever \(x_i=0\), for some \(1\leq i\leq m\). For instance, in the theory of groups (where the constant is denoted \(1\)), we may take \(t(x,y)\coloneq [x,y]=xyx^{-1}y^{-1}\) because  \(t(x,1)=1=t(1,y)\). The term \(u(x,y,z)\coloneq [x,[y,z]]\) is another commutator term. It is well known that in the variety of groups, any commutator term can be written as a product of such basic commutator terms; and in fact, reciprocally, many of the ``usual'' commutators naturally occurring in pointed varieties of universal algebras may indeed be characterised by means of abstract commutator terms in the above sense. In the case of commutative associative algebras, for instance, \(v(x,y)\coloneq xy\) is a commutator term, because \(0y=0=0x\). Likewise, so is \(w(x,y,z)\coloneq [x,[y,z]]\) in the case of Lie algebras. In the context of loops, the associator \(a(x,y,z)\coloneq (xy\cdot z)/(x\cdot yz)\) is an example of a somewhat different kind. Commutator terms generate commutator objects: for example, the \defn{commutator subgroup} \([X,X]\coloneq {\langle[x,y]\mid \text{\(x\), \(y\in X\)}\rangle}\) of a group \(X\) vanishes if and only if \(X\) is an abelian group.

What interests us here is the remarkable observation that, in the seemingly quite remote (abelian!) situation of a tensor product \(A\tensor_\Z B\) of two abelian groups, the pure tensors \(a\tensor b\) behave precisely like this. Not only do we have that \(a\tensor 0=0=0\tensor b\) for all \(a\in A\) and \(b\in B\); it turns out that the entire tensor product \(A\tensor_\Z B\) may be recovered as a commutator object \emph{in the variety of nil-\(2\) groups}---as was first proved by MacHenry in~\cite{MacHenry}.

The aim of this article is to explore the idea of viewing tensor products as commutators in a two-nilpotent category from a categorical perspective, using it to prove some fundamental general results in homological algebra involving tensors, which so far had remained out of reach of categorical algebra, relying on tools borrowed from abstract polynomial functor calculus and certain results in this framework we are adding to it here. As we shall see, a general version of the nilpotency condition that occurs in the case of groups is what extends the mere preservation of zeroes in both variables of the product to \emph{bilinearity}. A key tool here is the concept of a \emph{cosmash product}.

\subsection*{Cosmash products}
In the article~\cite{Smash} where \emph{smash products} are studied in a general, categorical setting, Carboni and Janelidze remark that for algebraic objects, the dual concept is in fact far more interesting. Indeed, unlike in topology, smash products in algebraic categories tend to be trivial, while cosmash products turn out to have many uses, the limits of which are still unknown today. It is already well established that the cosmash product of two objects in a semi-abelian category~\cite{Janelidze-Marki-Tholen} may serve as their \emph{formal commutator}~\cite{MM-NC} in order to capture aspects of Higgins's approach to commutator theory~\cite{Higgins} in the context of varieties of \(\Omega\)-groups, as well as a ``non-monadic'' view on internal object actions~\cite{Actions} and internal crossed modules~\cite{Janelidze,HVdL}. Yet, ever since its introduction in~\cite{Smash}, it has also been clear that certain tensor products appear as a cosmash product---for instance, the tensor product of commutative rings does.

This naturally leads to the question of how to capture other tensor products occurring ``in nature'' as a cosmash product. Assuming, for instance, that the classical tensor product of modules over a commutative ring is a cosmash product, then what we must ask ourselves is: \emph{Where?} or, more precisely: \emph{In which category?} Answering this question is one of our two main goals\footnote{Actually, this particular instance of the problem has a simple answer, as explained in~\cite{Smash}---see Subsection~\ref{Commutative algebras}.}. In particular, motivated by the fundamental example of the tensor product of abelian groups, we obtain our first main result: a recognition theorem (Theorem~\ref{Recognition Theorem}) showing that any symmetric bi-cocontinuous product on an abelian variety indeed is a cosmash product in a canonical (actually two-nilpotent) semi-abelian variety containing the given abelian variety as its abelian core. (In fact, it is even possible to capture \emph{all} two-nilpotent semi-abelian varieties solving this problem, but this will be presented elsewhere.)

Our second main goal, to which we devote a large portion of the text, is to answer the converse question: given a (semi-abelian) category \(\X\), does it admit an \emph{intrinsic symmetric bilinear product} playing the multiple roles for  \(\X\) which the tensor product of abelian groups plays in the theory of groups? Indeed, such a bilinear product can be constructed in two different ways, each of which proves to be useful depending on the context, but both of which again rely on the notion of cosmash product---as is made explicit in sections~\ref{Section Revision} and~\ref{Section Two-Nilpotency} below. Either way, it is an object obtained by combining certain limits and colimits.

We shall, moreover, explain how this new categorical-algebraic viewpoint on tensor products leads to results in (co)homology theory---of which we present an instance here, that will be substantially extended in subsequent work~\cite{D-H-VdL26}. The example of such a result developed here is a categorical version of \emph{Ganea's Theorem} in low-dimensional homology. In~\cite{Ganea}, Ganea proved that any central extension of groups
\begin{equation*}
  \xymatrix{0 \ar[r] & K \ar@{{ |>}->}[r] & B \ar@{-{ >>}}[r] & A \ar[r] & 0}
\end{equation*}
induces a six-term exact sequence
\[
  {\xymatrix{K\tensor_{\Z} \frac{B}{[B,B]} \ar[r] & \H_{2}B \ar[r] & \H_{2}A \ar[r] & K \ar[r] & \H_{1}B \ar@{-{ >>}}[r] & \H_{1}A \ar[r] & 0}}
\]
in integral homology. Later, a number of authors have considered versions of this theorem in other categories~\cite{Lue:Ganea, MR897010, Casas:CELA, CP, Pira:Ganea, Casas:Ganea, AriasLadra}, which naturally leads to the question of whether perhaps all those different versions are special cases of one categorical result. Although a categorical-algebraic interpretation of the right hand side of the sequence (starting from \(\H_{2}B\)) has been available for quite some time~\cite{EverVdL1} and this five-term exact sequence may in fact be seen as the tail of a long exact homology sequence~\cite{Tomasthesis, GVdL2}, until recently we could provide no such interpretation for the leftmost term~\(K\tensor_{\Z}({B}/{[B,B]})\). This is where our intrinsic approach to tensor products, defined in categorical-algebraic terms as a certain quotient of a (formal) commutator, helps: it allows us to prove a version of Ganea's Theorem for semi-abelian categories \(\X\) where the key term~\(K\tensor_{\Z}({B}/{[B,B]})\) is generalised to our bilinear product~\(K\tensor B\) --- provided~\(\X\) satisfies a fairly mild additional condition called \emph{algebraic coherence}~\cite{acc}. This includes all known results mentioned above (except the one given in~\cite{Lue:Ganea}) but excludes loops, for example, and algebras with genuine ternary or higher operations (that is, such operations not decomposable into binary operations), in particular \(n\)-Leibniz algebras~\cite{Filippov,CaLoPi} and algebras arising as linearisations of certain structures on manifolds (loops, webs) such as Akivis algebras and Sabinin algebras, or from homotopy theory such as \(A_{\infty}\)- and \(E_{\infty}\)-algebras.

To provide further evidence of the potential of the bilinear product introduced here let us give a quick preview of further developments to come: under the before-mentioned hypothesis of algebraic coherence we construct in subsequent work a categorical generalisation of J.~H.~C.\ Whitehead's classical \(\Gamma\)-functor and a natural morphism \(\sigma\colon \Gamma(K) \to K\otimes K\) such that our six-term extension of the five-term exact sequence further extends by at least two more terms (namely the third homologies of~\(A\) and \(B\)) after replacing our Ganea term \(K\otimes B\) with the cokernel of the composite map \(\Gamma(K) \to K\otimes K \to K\otimes B\), just like in the case of groups, where according to~\cite{EH71} we find expressions such as the \(K\tensor_{\Z}({B}/{[B,B]})\) considered above.

Mimicking the case of groups, the cokernel of \(\sigma\) will be defined to be the \emph{exterior square} of \(K\) which turns out to be naturally isomorphic with \(\H_2(K)\). So again just like in groups, the second homology of abelian objects naturally identifies with their exterior square. Together with our eight-term extension of the five-term exact sequence this implies that a Künneth type formula holds for \(\H_1\) and \(\H_2\) of \(K\times B\), where our bilinear product plays the role of the classical tensor product of abelian groups in the classical Künneth formula for products of groups. This motivates the question of whether a Künneth type formula in general holds for semi-abelian homology, where our bilinear product and its first derived functor play the role of the classical tensor and torsion products of abelian groups in the case of groups, see~\cite{Hilton-Stammbach-2}.

Another application of the tensor product we may mention here is a natural pairing
\[
  \tfrac{Z_k(X)}{Z_{k-1}(X)} \tensor \tfrac{\gamma_n(X)}{\gamma_{n+1}(X)} \to \tfrac{Z_{k-n}(X)}{Z_{k-n-1}(X)}
\]
where \(0=Z_0(X)\subobj Z_1(X)\subobj Z_2(X)\subobj\cdots\) denotes the \emph{(Higgins) upper central series} defined for semi-abelian varieties in~\cite{CCC} and \(X=\gamma_{1}(X)\supobj \gamma_{2}(X)\supobj \cdots\) denotes the \emph{(Higgins) lower central series} (see Subsection \ref{The lower central series}), generalising\footnote{Let us point out from the start that this lower central series, being based on Higgins commutators rather than Huq commutators, is fundamentally different from the concepts considered in \cite{EverVdL1} and \cite{BeBou} as shown, for instance, by means of a counterexample in the category of Moufang loops provided in the latter article. As explained in detail in \cite{SVdL3}, the two types of nilpotency will coincide in the context of an algebraically coherent~\cite{acc} semi-abelian category. Both categorical renditions of this classical concept have their use; for our purposes here, the lower central series defined in terms of higher Higgins commutators is the appropriate one.} the well-known such pairing in nilpotent group theory~\cite{MR409661}. Moreover, thanks to the above pairing and its higher analogues obtained by the multilinearisation of the higher cosmash products, if we collect all of the underlying abelian groups of~the upper central quotients into a (negatively) graded abelian group, the latter becomes a \emph{graded module}
over the (positively) graded abelian group formed by the underlying abelian groups of the lower central quotients, viewed as an \emph{algebra} over the reduced symmetric \emph{operad in abelian groups associated with the variety}. All this is introduced in~\cite{CCC} and generalises the corresponding fact in group theory, where the associated operad is the Lie operad and the (negatively graded) sequence of the quotients of the upper central series is a representation of the graded Lie ring formed by the quotients of the lower central series. However, the analogous fact seems to be new in particular for the variety of loops where the associated operad is the Sabinin operad.

All these facts provide further evidence that our bilinear product in a given semi-abelian category indeed plays the exact same multiple roles which the tensor product plays in the category of groups.

Coming back to our original question
\begin{quote}
  \emph{Given some known tensor product in an abelian category, in which category may it be seen as a cosmash product?}
\end{quote}
let us now briefly sketch how we characterise tensor products as cosmash products in semi-abelian categories, and then lay out the structure of the article.

\subsection*{The tensor product of abelian groups}
First recall that, for two given objects \(X\) and \(Y\) of a semi-abelian category \(\X\), the \defn{cosmash product} \(X\cosmash Y\) is defined as in the short exact sequence
\[
  \xymatrix@=3em{ 0 \ar[r] & X\cosmash Y \ar@{{ |>}->}[r] &
  X+Y \ar@{-{ >>}}[r]^-{\left\links\begin{smallmatrix}1_{X} & 0 \\ 0 & 1_{Y}\end{smallmatrix}\right\rechts} & X\times Y \ar[r] & 0\text{.}}
\]
As an illustration, we consider abelian groups \(A\), \(B\) and their tensor product as \(\Z\)-modules \(A\tensor_{\Z}B\). Our aim is then to find a semi-abelian category \(\X\) in which there exists an isomorphism \(A\cosmash B\cong A\tensor_{\Z}B\).

If \(\X\) is the category \(\Mod_{\Z}\) of modules over the ring of integers \(\Z\) itself, then since the canonical comparison \(A+B\to A\times B\) is an isomorphism, the object \(A\cosmash B\) is zero, so it need not be isomorphic to \(A\tensor_{\Z}B\). That is to say, ``\(\X=\Mod_{\Z}\)'' is not the right answer to the above question. The same argument shows that also no other abelian category \(\X\) containing \(\Mod_{\Z}\) can be.

The idea is now to view \(\Mod_{\Z}\) as contained in a larger \emph{semi}-abelian category~\(\X\) of which it forms the \defn{abelian core} \(\Ab(\X)\), the full reflective subcategory determined by those objects which admit an internal abelian group structure. We can, for instance, choose \(\X=\Gp\), the category of (non-abelian) groups, which is semi-abelian and has abelian groups, so \(\Z\)-modules, for its abelian objects.

It is, however, well known that the cosmash product \(X\cosmash Y\) of two given groups \(X\) and \(Y\) is the subgroup of \(X+Y\) generated (in fact, freely generated) by the commutators \(xyx^{-1}y^{-1}\) of non-trivial elements \(x\), \(y\) of~\(X\) and~\(Y\), respectively---so again \(A\cosmash B\) is different from~\(A\tensor_{\Z}B\). In other words, also the answer ``\(\X=\Gp\)'' is wrong.

A correct answer, first discovered by MacHenry~\cite{MacHenry}, lies in between the categories \(\Mod_{\Z}=\Ab\) and \(\Gp\): it is the category \(\Nil_{2}(\Gp)\) of all groups of nilpotency class at most \(2\) (also called \defn{nil-\(2\) groups} or \defn{\(2\)-step nilpotent} groups). As shown in~\cite{MacHenry,Hartl-Vespa}, the cosmash product of two abelian groups \(A\) and~\(B\), computed in this category, will indeed be precisely \(A\tensor_{\Z}B\).

In fact, this answer is not unique: other semi-abelian categories~\(\X\) exist in which the tensor \(A\tensor_{\Z}B\) may be recovered as a cosmash product; examples will be given below. One characteristic feature those categories have in common is that they are \emph{two-nilpotent}.

\subsection*{Bilinear cosmash products and two-nilpotent categories}
Recall that the tensor product of modules over a commutative ring \(R\) is \emph{bilinear} in the sense that
\[
  (A\oplus B)\tensor_{R}C\cong (A\tensor_{R}C)\oplus (B\tensor_{R}C)
\]
which, by symmetry of \(\tensor_R\), also implies linearity with respect to the second variable, namely \(A \tensor_R (B \biprod C) \cong (A \tensor_R B) \biprod (A \tensor_R C)\).

So if we hope to recover the tensor as a cosmash product in some semi-abelian category \(\X\), the latter should behave similarly. This leads us to ask that the canonical morphism
\[
  (X+Y)\cosmash Z \to (X\cosmash Z)\times(Y\cosmash Z)
\]
is an isomorphism in~\(\X\). (We shall later explain why this choice of \(+\) and \(\times\) is the ``right'' one.) As it thus turns out~\cite{HVdL}, this happens precisely when the kernel in the natural short exact sequence
\[
  \xymatrix@=3em{ 0 \ar[r] & X\cosmash Y\cosmash Z \ar@{{ |>}->}[r] &
  (X+Y)\cosmash Z \ar@{-{ >>}}[r] & (X\cosmash Z)\times(Y\cosmash Z) \ar[r] & 0\text{,}}
\]
which is nothing but the \defn{ternary cosmash product} of \(X\), \(Y\) and \(Z\) in \(\X\), vanishes. (For the original definition, which is symmetric in \(X\), \(Y\) and \(Z\), see Subsection~\ref{Binary and ternary cosmash products}). A semi-abelian category which satisfies this condition is called \defn{two-nilpotent}, for the following reason.

Given three subobjects \(K\), \(L\), \(M\leq X\) represented by monomorphisms \(k\), \(l\) and~\(m\), according to~\cite{Actions,HVdL} the \defn{ternary commutator} \([K,L,M]\leq X\) is defined to be the image of the composite
\[
  \xymatrix@=3em{K\cosmash L\cosmash M \ar@{{ |>}->}[r]^-{\iota_{K,L,M}} & K+L+M \ar[r]^-{\DCoindArr{k \and l \and m}} & X\text{.}}
\]
We shall see that all ternary commutators in \(\X\) vanish if and only if every object~\(X\) of \(\X\) is \defn{two-nilpotent}, which means that \([X,X,X]=0\)---see Subsection~\ref{Subsec Two-nilpotent}.

In many categories (all \emph{Orzech categories of interest}~\cite{Orzech}, for instance, as demonstrated in~\cite{acc}; see also~\cite{SVdL3}) this ternary commutator coincides with the repeated binary commutator \([X,[X,X]]\) one expects it to be, so that in those cases two-nilpotency is a familiar concept. In the category of groups, for instance, \([X,X,X]=0\) if and only if \(X\) is a group of nilpotency class at most two. On the other hand, examples exist of semi-abelian categories in which the two concepts are different---the category of loops, for instance~\cite{Mostovoy1,M-PI-SH,SVdL3}---and in those categories the right one for our purposes is \([X,X,X]\).

In summary, we need cosmash products to be bilinear, and we can only expect this to happen in semi-abelian categories which are two-nilpotent. This at once leads to the following procedure: given an abelian category \(\A\), let \(\X\) be a semi-abelian category such that \(\A=\Ab(\X)\). Then \(\A\subseteq \Nil_{2}(\X)\subseteq \X\) where \(\Nil_{2}(\X)\) is the \defn{two-nilpotent core} of \(\X\), the full reflective subcategory determined by the two-nilpotent objects. Given objects \(M\) and \(N\) in \(\A\), the cosmash product \(M\cosmash N\) in \(\Nil_{2}(\X)\) is now a good candidate for a tensor product of \(M\) and~\(N\).

\subsection*{The bilinear product on a semi-abelian category}
Let
\[
  \nil_{2}\colon \X\to \Nil_{2}(\X)\colon X\mapsto X/[X,X,X]
\]
denote the left adjoint to the inclusion functor. Given objects \(X\) and \(Y\) of \(\X\), we shall write
\[
  X\tensor_{\X}Y\DefEq \nil_{2}(X)\cosmash_2 \nil_{2}(Y)
\]
for the cosmash product in \(\Nil_{2}(\X)\) of the reflections of \(X\) and \(Y\). Since, as we will prove later on, this object is always abelian, we obtain a functor
\[
  \tensor_{\X}\colon\X\times\X\to \Ab(\X)\colon (X,Y)\mapsto X\tensor_{\X} Y
\]
which we call the \defn{bilinear product} on \(\X\). Note that the product \(X\tensor_{\X} Y\) of two given two-nilpotent objects \(X\), \(Y\) only depends on \(\Nil_2(\X)\) and not on the ambient category \(\X\).

\begin{table}
  \resizebox{\textwidth}{!}
  {\begin{tabular}{cccc}
      \toprule
      category \(\X\)   & objects                                & abelian core \(\Ab(\X)\) & bilinear product \(A\tensor_{\X} B\)             \\
      \midrule
      \(\Gp\)           & groups                                 & \(\Ab\)                  & \(A\tensor_{\Z}B\)                               \\
      \(\Loop\)         & loops                                  & \(\Ab\)                  & \(A\tensor_{\Z}B\)                               \\
      \(\Nil_{2}(\Gp)\) & two-nilpotent groups                   & \(\Ab\)                  & \(A\tensor_{\Z}B\)                               \\
      \(\CRng\)         & commutative rings                      & \(\Ab\)                  & \(A\tensor_{\Z}B\)                               \\
      \(\Mod_{R}\)      & \(R\)-modules                          & \(\Mod_{R}\)             & \(0\)                                            \\
      \(\Pt_G(\Gp)\)    & \(G\)-actions                          & \(\Mod_{\Z[G]}\)         & \(A\tensor_{\Z}B\) with appropriate \(G\)-action \\
      \(\CAlg_{R}\)     & commutative associative \(R\)-algebras & \(\Mod_{R}\)             & \(A\tensor_{R}B\)                                \\
      \(\Alg_{R}\)      & associative \(R\)-algebras             & \(\Mod_{R}\)             & \((A\tensor_{R}B)\oplus (B\tensor_{R}A)\)        \\
      \(\Lie_{R}\)      & \(R\)-Lie algebras                     & \(\Mod_{R}\)             & \(A\tensor_{R}B\)                                \\
      \(\Leib_{R}\)     & \(R\)-Leibniz algebras                 & \(\Mod_{R}\)             & \((A\tensor_{R}B)\oplus (B\tensor_{R}A)\)        \\
      \(\CStarAlg\)     & \(C^*\)-algebras                       & \(\{0\}\)                & \(0\)                                            \\
      \(\HSLat\)        & Heyting semilattices                   & \(\{0\}\)                & \(0\)                                            \\
      \bottomrule
    \end{tabular}}
  \bigskip
  \caption{The bilinear product \(A\tensor_{\X} B\) of \(A\), \(B\in \Ab(\X)\) in terms of the surrounding semi-abelian category \(\X\). Here \(R\) is a commutative ring with unit, and an object of \(\Ab(\Pt_G(\Gp))\) is viewed as a \(G\)-action on an abelian group \(A\) or \(B\). Details and additional examples are worked out in Section~\ref{Section Examples} and Section~\ref{Section Tensor of G-actions}.}\label{Overview}
\end{table}

In homological-algebraic applications such as Ganea's Theorem (Theorem~\ref{Theorem-Ganea}), the functor \(\tensor_{\X}\) may play the role of an \emph{intrinsic tensor product} on the category~\(\X\). In contrast, when considered as a functor \(\Ab(\X)\times\Ab(\X)\to \Ab(\X)\), the bilinear product \(\tensor_{\X}\) is not intrinsic but relative, since it depends on the surrounding semi-abelian category~\(\X\). In Table~\ref{Overview}, which gives an overview of some examples, we see for instance that the tensor product of \(R\)-modules may be captured as the bilinear product in either the category of commutative associative algebras or in the category of Lie algebras over \(R\). On the other hand, the category of associative \(R\)-algebras induces a different bilinear product on~\(\Mod_{R}\).

\subsection*{Limitations}
Of course, not all tensor products ``in nature'' are readily captured as cosmash products in some two-nilpotent semi-abelian category \(\X\). An important limitation of our approach is that a bilinear cosmash product of objects in \(\X\) will automatically be an abelian object in \(\X\), which is a requirement that seems hard to satisfy in certain cases. One ``bad'' example is the tensor product of associative algebras, which is supposed to be again an associative algebra---but certainly not an abelian one, which means that it has a trivial multiplication. We do not know in which way this example fits the theory, or whether it fits the theory at all.

In fact, our original motivating example taken from~\cite{Smash}, the category of commutative rings, does not follow the pattern we just sketched: here the cosmash product in the category itself is already the tensor product in \(\CRng\). This example is further developed in~\cite{Illinois}.

Another such example is the Brown--Loday \emph{non-abelian tensor product}~\cite{Brown-Loday} of two groups mutually and compatibly acting on each other, where again it is not obvious which category to choose, or whether such a category can be found at all. However, in this case, one connection between the two tensor products is clear, because the bilinear product of two abelian groups is the non-abelian tensor product of these groups, when viewed as acting trivially on each other. As proved in~\cite[Theorem~7.4]{dMVdL19.3}, this result stays valid for the bilinear product in any \emph{algebraically coherent}~\cite{acc} semi-abelian category: given two abelian objects acting trivially on each other, their non-abelian tensor product as introduced in~\cite{dMVdL19.3} is precisely the bilinear product. In such categories, the bilinear product may be viewed as a special case of the non-abelian tensor product.

Finally, other examples are just complicated: for instance, finding a precise formula for the bilinear product of internal actions in a semi-abelian category is difficult, at least in a category which is not \emph{locally algebraically cartesian closed}---cf.\ Section~\ref{Section Tensor of G-actions}. Nevertheless, first results have recently been obtained for semi-abelian varieties, cf.\ subsection \ref{reps}; however, this topic remains a subject of future work.

\subsection*{Structure of the text}
The first section recalls the basic definitions and results on cosmash products. In Section~\ref{Section Two-Nilpotency}, this is used in the description of abelian and two-nilpotent objects, which are essential in the definition of the bilinear product (Subsection~\ref{Subsection Bilinear Product}). In the ensuing Section~\ref{Section Nilpotentisation} we study nilpotency in general, with the aim of proving that nilpotentisation commutes with Birkhoff reflectors (Theorem~\ref{Nil vs Birkhoff}). The next Section~\ref{Section Further Properties} is devoted to proving more advanced properties of the bilinear product, in particular the following two most fundamental ones: firstly, a description of the bilinear product in terms of~\(\X\) alone (Theorem~\ref{prop:internal descr of tensor}), namely as being the \emph{bilinearisation} of the cosmash product of~\(\X\), which is the key to its various \emph{applications} as it comes with a (co)universal property---in addition to being its origin, see Example \ref{bilinearisation of the cosmash product of groups}---while the description as being the cosmash product of the two-nilpotent core of~\(\X\) is more suitable for its \emph{computation} (as in many basic examples of two-nilpotent categories it is not hard to explicitly construct binary sums), next to the conceptual aspect that it thus appears as a cosmash product itself (and thus inherits the key properties of the latter). A second fundamental property of the bilinear product is the sequential right exactness of the functors~\({X\tensor(-)}\) (Proposition~\ref{Proposition-RightExact-Tensor}).

Section~\ref{Section Examples} treats a first set of examples. To those which appear in Table~\ref{Overview}, we add sheaves of abelian groups and internal (pre)crossed modules.

In Section~\ref{Section Operads} we consider the case of algebras over an operad in full detail, analysing nilpotency in this context, comparing it with the operadic concept of nilpotency, computing the bilinear product and considering the question of its associativity. The next Section~\ref{Section Varieties} collects general results on bilinear products in varieties of algebras. We start with structural properties of the bilinear product such as cocontinuity (Proposition~\ref{Proposition Tensor Cocontinuous}) and left adjointness (Proposition~\ref{Prop pre SMC}), leading to the question when it is closed monoidal. We then work towards one of our main results: the Recognition Theorem~\ref{Recognition Theorem}, where we prove that any symmetric bi-cocontinuous bifunctor on an abelian variety of algebras can be recovered as the bilinear product in the semi-abelian variety of algebras over a certain \(2\)-nilpotent operad.

Section~\ref{Section Ganea} provides an application in homological algebra: Theorem~\ref{Theorem-Ganea}, a categorical version of Ganea's six term exact sequence, of which known instances now become a special case. Section~\ref{Section Cross-Effects} focuses on a key technical result which is already used in Section~\ref{Section Two-Nilpotency}: Theorem~\ref{Theorem-Rightex} on the right exactness of cosmash products. Section~\ref{Section Abelian Extensions} treats the concept of an abelian extension from the cosmash product viewpoint, and in Section~\ref{Section Tensor of G-actions} we use this to calculate the bilinear product in certain categories of internal actions: Example~\ref{Example Lie algebra representations} and Example~\ref{Example group representations}. Here, the bilinear product yields a tensor product on Beck modules in any semi-abelian category, generalising the one of representations of groups and Lie algebras, even if we cannot yet provide an explicit computation in general. We end with Section~\ref{Section Further Questions} listing some additional questions for further investigation.

\tableofcontents

\section{\texorpdfstring{Cosmash products and Higgins commutators:\except{toc}{\linebreak} overview of definitions and basic properties}{Cosmash products and Higgins commutators}}\label{Section Revision}
We are defining cosmash products and the induced tensor products as in~\cite{Smash, Actions, HVdL} and~\cite{Hartl-Vespa}. As recalled in Section~\ref{Section Cross-Effects}, these cosmash products are actually instances of cross-effects of functors.

Throughout the text, unless mentioned otherwise, we shall assume that \(\X\) is a \emph{semi-abelian} category in the sense of~\cite{Janelidze-Marki-Tholen,Borceux-Bourn}.

\subsection{Binary and ternary cosmash products.}\label{Binary and ternary cosmash products}
Given objects \(X\) and \(Y\) of \(\X\), their \defn{cosmash product}, written \(X\cosmash Y\) or \(X\cosmash_{\X}Y\), is defined as in the short exact sequence
\[
  \xymatrix@=3em{ 0 \ar[r] & X\cosmash Y \ar@{{ |>}->}[r]^-{\iota_{X,Y}} &
  X+Y \ar@{-{ >>}}[r]^-{r_{X,Y}} & X\times Y \ar[r] & 0}
\]
where \(r_{X,Y}\coloneq\left\links\begin{smallmatrix}1_{X} & 0\\ 0 & 1_{Y}\end{smallmatrix}\right\rechts\).
(Exactness of this sequence on the right corresponds to~\(\X\) being unital, together with the coincidence of the classes of regular epimorphisms and normal epimorphisms in~\(\X\); these properties follow from \(\X\) being pointed, regular and protomodular~\cite{Borceux-Bourn}.) Note that if \(X\) or \(Y\) is trivial, then so is \(X\cosmash Y\).

Given a third object \(Z\) of \(\X\),
\[
  \iota_{X,Y,Z}\colon X\cosmash Y\cosmash Z \to X+Y+Z
\]
denotes the kernel of
\[
  r_{X,Y,Z}\coloneq\left\links\begin{smallmatrix} i_{X} & i_{Y} & 0     \\
    i_{X} & 0     & i_{Z} \\
    0     & i_{Y} & i_{Z}\end{smallmatrix}\right\rechts\colon X+Y+Z \to (X+Y)\times (X+Z) \times (Y+Z)\text{.}
\]
The object \(X\cosmash Y\cosmash Z\) is the \defn{ternary cosmash product} of \(X\), \(Y\) and~\(Z\). The following alternative construction of this object appears in Remark~2.8 of~\cite{HVdL}: we view it as a \emph{cross-effect} of the identity functor---see Section~\ref{Section Cross-Effects}. It is actually the \emph{third} cross-effect of \(1_\X\)---this corresponds to the definition we gave just above---which happens to coincide with the \emph{second} cross-effect of the functor \((-)\cosmash Z\colon \X\to \X\):

\begin{lemma}\label{Lemma-Bilinear}
  There is a short exact sequence
  \[
    \xymatrix@=3em{ 0 \ar[r] & X\cosmash Y\cosmash Z \ar@{{ |>}->}[r] &
    (X+Y)\cosmash Z \ar@{-{ >>}}[r] & (X\cosmash Z)\times(Y\cosmash Z) \ar[r] & 0}
  \]
  for all \(X\), \(Y\), \(Z\) in \(\X\).
\end{lemma}
\begin{proof}
  A direct, explicit proof of the exactness of this sequence on the left is worked out in~\cite[Proposition~2.12]{Actions}; see also \cite[Proposition~4.1]{SVdL3} for a higher-order generalisation. Since \((-) \cosmash Z\) preserves \(0\), we may apply Lemma~\ref{Lemma-Reduced} to see that the right-hand arrow is a regular epimorphism. Hence it is a cokernel of its kernel, which we already know is the left-hand arrow.
\end{proof}

With this alternative point of view on the ternary commutator, we can apply Proposition~\ref{crossdecomp'} to obtain:

\begin{lemma}\label{Lemma Smash of Sum}
  If \(X\), \(Y\), and \(Z\) are objects in \(\X\), then there is a decomposition
  \[
    (X+Y)\cosmash Z \cong \bigl((X\cosmash Y \cosmash Z)\rtimes (X\cosmash Z)\bigr)\rtimes (Y \cosmash Z)\text{.}
  \]
  More precisely, there exists an object \(V\) and split short exact sequences
  \[
    \xymatrix{0 \ar[r] & V \ar@{{ |>}->}[r] & (X+Y)\cosmash Z \ar@<.5ex>@{-{ >>}}[r] & Y \cosmash Z \ar[r] \ar@<.5ex>[l] & 0}
  \]
  and
  \begin{equation*}
    \xymatrix{0 \ar[r] & X\cosmash Y \cosmash Z \ar@{{ |>}->}[r] & V \ar@<.5ex>@{-{ >>}}[r] & X \cosmash Z \ar[r] \ar@<.5ex>[l] & 0}
  \end{equation*}
  in \(\X\)---see~Figure~\ref{Figure Ternary}.\noproof
\end{lemma}

\begin{figure}
  \hfil\xymatrix{& 0 \ar[d] & 0 \ar[d] & 0\ar[d] \\
    0 \ar[r] & X\cosmash Y \cosmash Z \ar@{{ |>}->}[r] \ar@{{ |>}->}[d] & W \ar@{{ |>}->}[d] \ar@<.5ex>@{-{ >>}}[r] & \ar@<.5ex>[l] Y\cosmash Z \ar@{=}[d] \ar[r]& 0\\
    0 \ar[r] & V \ar@<.5ex>@{-{ >>}}[d] \ar@{{ |>}->}[r] & (X+Y)\cosmash Z \ar@<.5ex>@{-{ >>}}[r] \ar@<.5ex>@{-{ >>}}[d] & Y\cosmash Z \ar@<.5ex>[l] \ar[d] \ar[r]& 0\\
    0\ar[r] & X\cosmash Z \ar@<.5ex>[u] \ar@{=}[r] \ar[d] & X\cosmash Z \ar@<.5ex>[u]\ar[r] & 0\\
    & 0}\hfil
  \caption{Alternate description of \(X\cosmash Y \cosmash Z\)}\label{Figure Ternary}
\end{figure}

\subsection{The Higgins commutator.}\label{Subsection Binary Higgins Commutator}
Given two objects \(K\) and \(L\) of \(\X\), the cosmash product \(K\cosmash L\)
behaves as a kind of ``formal commutator'' of~\(K\) and~\(L\) (see~\cite{Actions} and~\cite{MM-NC}). If now \(k\colon K \to X\) and \(l\colon L \to X\) are subobjects of an object \(X\), their \defn{(Higgins) commutator} \([K,L]\leq X\) is the image of the induced composite morphism
\[
  \xymatrix@=2em{K\cosmash L \ar@{{ |>}->}[r]^-{\iota_{K,L}} & K+L \ar[r]^-{\DCoindArr{k \and l}} & X\text{.}}
\]
We write \(e_{K,L}^{X}\colon {K\cosmash L\to [K,L]}\) for the regular epi--part of the image factorisation, and \([k,l]\colon{[K,L]\to X}\) for the image itself.

If \(K\join L=X\), then the Higgins commutator \([K,L]\) is normal in \(X\)~\cite[Proposition~5.2]{MM-NC} so that it coincides with the Huq commutator considered in~\cite{Bourn-Gran, Borceux-Bourn}. Hence any extension in~\(\X\) such as
\begin{equation*}
  \xymatrix{0 \ar[r] & K \ar@{{ |>}->}[r]^{k} & B \ar@{-{ >>}}[r]^{f} & A \ar[r] & 0}
\end{equation*}
is central in the sense of~\cite{Bourn-Gran} if and only if \([K,B] \normal B\) is trivial.

Another instance of a situation where \(K\join L=X\) occurs when \(X=K+L\) and \(k=i_1\colon K\to K+L\), \(l=i_2\colon L\to K+L\) are the coproduct inclusions. Then the canonical inclusion \(\iota_{K,L}\) of cosmash product \(K\cosmash L\) into \(X\) plays the role of the Higgins/Huq commutator \([k(K),l(L)]\) of \(K\) and \(L\), viewed as subobjects of \(X=K+L\). In this precise sense, the cosmash product is a ``formal commutator''.

\subsection{The ternary commutator.}
The Higgins commutator generally does not preserve joins, but the defect may be measured---it is a ternary commutator which can be computed by means of a ternary cosmash product. Let us extend the definition above: given a third subobject \(m\colon M \to X\) of the object \(X\), the ternary commutator \([K,L,M]\leq X\) is the image of the composite
\[
  \xymatrix@=3em{K\cosmash L\cosmash M \ar@{{ |>}->}[r]^-{\iota_{K,L,M}} & K+L+M \ar[r]^-{\DCoindArr{k \and l \and m}} & X\text{.}}
\]
Like in the binary case, the cosmash product \(K\cosmash L\cosmash M\) may be viewed as the (``formal'') commutator \([i_1(K),i_2(L),i_3(M)]\normal K+L+M\) of \(K\), \(L\) and \(M\), canonically included into their sum.

The following fundamental join decomposition result or \defn{distributive law} (first obtained in~\cite{Higgins} in the context of varieties of \(\Omega\)-groups, then extended to semi-abelian categories in~\cite{Actions, HVdL}) is a consequence of Lemma~\ref{Lemma Smash of Sum}.

\begin{lemma}\cite[Proposition 2.22]{HVdL}\label{Lemma-Join}
  If \(K\), \(L\), \(M\leq X\), then
  \[
    [K,L\join M]=[K,L]\join [K,M]\join [K,L,M]
  \]
  where all joins are computed in \(X\).\noproof
\end{lemma}

By the symmetry of the commutator, which follows from symmetry of the cosmash product, which in turn follows from symmetry of sums and products, a similar result of course also holds for a commutator with a join in its first slot.

\section{Two-nilpotency and the bilinear product}\label{Section Two-Nilpotency}

A \defn{Birkhoff subcategory}~\cite[Section~1.2]{Janelidze-Kelly} of a semi-abelian category is a full reflective subcategory closed under subobjects and regular quotient objects. As is well known, the unit of a Birkhoff reflection is always a regular epimorphism.

In what follows, we consider certain Birkhoff subcategories any semi-abelian category has, each determined by objects satisfying a commutator condition.

\subsection{Abelian objects.}
An object \(A\) of a semi-abelian category \(\X\) is \defn{abelian} when it admits a (necessarily unique) structure of internal abelian group. It is well known that this happens precisely when \([A,A]=0\), or equivalently, when the codiagonal morphism \(\nabla_A=\CoindArr{1_A \and 1_A}\colon A+A \to A\) factors through the regular epimorphism \(\xymatrix@1{r_{A,A}\colon A+A \ar@{- >>}[r]&A\times A}\) (where the resulting map \(A\times A\to A\) is the internal abelian group structure), cf.\ Subsection~\ref{Binary and ternary cosmash products}.

The (abelian) Birkhoff subcategory of~\(\X\) determined by the abelian objects is denoted \(\Ab(\X)\) and called the \defn{abelian core} of \(\X\). The reflector \(\ab\colon{\X\to\Ab(\X)}\) is determined by \(\ab(X)=X/[X,X]\). Note that all objects of \(\Ab(\X)\) are indeed abelian, and that \(\Ab(\X)\) is the largest Birkhoff subcategory with this property.

\subsection{Two-nilpotent objects.}\label{Subsec Two-nilpotent}
Likewise, an object \(U\) is called \defn{two-nilpotent} or \defn{nilpotent of class \(2\)} when the ternary commutator \([U,U,U]\) vanishes. This determines a Birkhoff subcategory of \(\X\) denoted \(\Nil_{2}(\X)\) and called the \defn{two-nilpotent core} of \(\X\). The reflector \(\nil_{2}\colon{\X\to \Nil_{2}(\X)}\) is given by \(\nil_{2}(X)=X/[X,X,X]\).

\subsection{Two-nilpotent categories.}
A semi-abelian category is called \defn{two-nilpotent} when all ternary cosmash products in it vanish. This happens precisely when its objects are two-nilpotent, because for any three given objects \(U\), \(V\) and \(W\), the object \(U\cosmash V\cosmash W\) is the commutator \([i_1(U),i_2(V),i_3(W)]\normal U+V+W\) of \(U\), \(V\) and \(W\), smaller than the commutator \([U+V+W,\;U+V+W,\;U+V+W]\), which is trivial when \(U+V+W\) is two-nilpotent. As a consequence, the two-nilpotent core \(\Nil_{2}(\X)\) of any semi-abelian category \(\X\) is indeed two-nilpotent. The converse is immediate from the definition. Similarly to what happens for \(\Ab(\X)\), since~\(\Nil_{2}(\X)\) includes all two-nilpotent objects of \(\X\), it is the largest two-nilpotent Birkhoff subcategory of \(\X\).

\subsection{Products and coproducts}\label{Notation coproduct}
Since the right adjoint inclusions preserve all limits, the product in \(\Ab(\X)\) or in~\(\Nil_{2}(\X)\) is given by the product in~\(\X\). For coproducts, however, this is not the case. In \(\Ab(\X)\) the coproduct is the product \(A\times B\) which we denote \(A\oplus B\), the biproduct of \(A\) and \(B\). And in order to avoid confusion with the coproduct \(X+Y\) in~\(\X\), we write \(U+_{2} V\) for the sum in \(\Nil_{2}(\X)\). Always \(\nil_{2}(X+Y)=\nil_{2}(X)+_{2} \nil_{2}(Y)\) so that when \(A\) and \(B\) are abelian (or two-nilpotent), \(A+_{2} B=\nil_{2}(A+B)\).

\subsection{Linear functors.}
A functor \(F \colon \X \to \Y\) between pointed categories is \defn{reduced} whenever it preserves zero objects. The intuition from polynomial functor theory that a reduced functor between semi-abelian categories should be called \defn{linear} when it sends binary sums to products is made precise by asking that the canonical comparison regular epimorphism
\[
  r^F_{X,Y}=\DIndArr{F(\DCoindArr{\id{X} \and 0}) \and F(\DCoindArr{0 \and \id{Y}})} \from F(X+Y) \to F(X)\times F(Y)
\]
is an isomorphism---so that, in the terminology of Section~\ref{Section Cross-Effects}, its \emph{second cross-effect} vanishes, which means that this morphism has a trivial kernel. Note that the morphism \(F(X+Y) \to F(X)\times F(Y)\) is indeed a regular epimorphism by Lemma~\ref{Lemma-Reduced}.

For instance, the identity functor \(1_{\X}\colon \X\to \X\) is linear precisely when the category \(\X\) is \defn{linear} in the sense that the canonical comparison morphism \(X+Y\to X\times Y\) is an isomorphism\footnote{This is the terminology used in~\cite{Borceux-Bourn}, not to be confused with the concept of a \defn{\(R\)-linear category}, which is a category enriched in \(R\)-vector spaces, so that in particular its hom-sets carry an \(R\)-vector space structure.}, which is well known to be equivalent with the category being enriched over commutative monoids---a condition that, in the present semi-abelian context, is equivalent to the abelianness of \(\X\), and ensures that the commutative monoid structures of the hom-sets are actually abelian groups.

On the other hand, Lemma~\ref{Lemma-Bilinear} tells us that each functor \((-)\cosmash Z\colon \X\to \X\) is linear if and only if \(\X\) is a two-nilpotent category.

\begin{proposition}\label{linear functors have abelian values}
  Any linear functor \(F \colon \X \to \Y\) between semi-abelian categories takes values in (that is, factors through) the abelian core of \(\Y\).
\end{proposition}
\begin{proof}
  For an object \(X\) of \(\X\), consider the following diagram.
  \[
    \xymatrix{
    F(X) + F(X) \ar[dd]_{\CoindArr{1_{F(X)} \and 1_{F(X)}}} \ar[rd]|{\CoindArr{F(i_1) \and F(i_2) }} \ar@/^/[rrd]^{r_{F(X),F(X)}} & &\\
    & F(X+X) \ar[ld]|{F(\CoindArr{1_X \and 1_X})}\ar[r]^-{\cong}_-{r^F_{X,X}} &F(X)\times F(X) \ar@/^/[dll]^{\hspace{18pt} F(\CoindArr{1_X \and 1_X})\comp (r^F_{X,X})^{-1} }\\
    F(X) & &
    }
  \]
  All internal triangles commute, hence so does the exterior one, which shows that \(F(X)\) is abelian, as desired.
\end{proof}

\subsection{Bilinear bifunctors.}\label{Bilinear bifunctors}
A bifunctor \(F \colon \X \times \Y \to \Z\), where \(\X\), \(\Y\) and \(\Z\) are pointed categories, is \defn{bireduced} whenever \(F(X,Y) \) is trivial as soon as one of \(X\) and \(Y\) is trivial\footnote{In other words, \(F(X,Y) \) behaves as a \emph{commutator term} in the sense of the first paragraph of the Introduction.}. A bireduced bifunctor \(F\colon{\X\times\X\to\X}\) on a semi-abelian category \(\X\) is called \defn{bilinear} when for all \(X\), \(Y\), \(Z\in\X\), the canonical arrows
\[
  \DIndArr{F(\DCoindArr{\id{X} \and 0},\id{Z}) \and F(\DCoindArr{0 \and \id{Y}},\id{Z})} \from F(X+Y,Z) \to F(X,Z)\times F(Y,Z)
\]
and
\[
  \DIndArr{F(\id{X},\DCoindArr{\id{Y} \and 0}) \and F(\id{X},\DCoindArr{0 \and \id{Z}})} \from F(X,Y+Z) \to F(X,Y)\times F(X,Z)
\]
are isomorphisms---cf.~\cite{LACK2012593} where isomorphisms of this kind are studied in general. Note that, as the name suggests, a functor is bilinear whenever it is linear in each of its variables. Let us remark that, if \(F\) is symmetric (as is the case for the cosmash product, for example), then it is enough to check that only one of the two arrows above is an isomorphism, the other one following by symmetry. Lemma~\ref{Lemma-Bilinear} now tells us:

\begin{proposition}\label{Cosmash Bilinear iff Category Two-Nilpotent}
  The bifunctor \(\cosmash\colon \X\times\X\to\X\colon (X,Y)\mapsto X\cosmash Y\) is bilinear if and only if \(\X\) is a two-nilpotent category.\noproof
\end{proposition}

\subsection{The bilinear product on a semi-abelian category.}\label{Subsection Bilinear Product}
On a semi-abelian category \(\X\), the \defn{bilinear product} is the extension of the cosmash product in \(\Nil_{2}(\X)\) to all of \(\X\):
\[
  \tensor_{\X}\colon \X\times\X\to\X\colon (X,Y)\mapsto X\tensor_{\X}Y\DefEq \nil_{2}(X)\cosmash_2 \nil_{2}(Y)\text{.}
\]
This bifunctor \(\tensor_{\X}\) is indeed bilinear, because the reflector \(\Nil_{2}\colon{\X\to \Nil_{2}(\X)}\) preserves binary sums.

When the category \(\X\) is understood, we may sometimes drop the subscript \(\X\) and write \(\tensor\) for the bilinear product.

\subsection{Symmetry, reducedness.}\label{Subsection Symmetry}
It is immediately clear that bilinear products are symmetric---it comes naturally equipped with a \defn{symmetry isomorphism} or \defn{twist map}
\[
  \xymatrix{\tau=\tau_{X,Y}=\tau_{X,Y}^{\X}\colon X\tensor Y \ar[r]^-{\cong} & Y\tensor X}
\]
---and bireduced (\(X\tensor 0= 0=0\tensor Y\)) because cosmash products have the same properties. Remarkably, the isomorphism \(\tau_{X,Y}\) is not necessarily the usual or ``obvious'' one: in Section~\ref{Section Examples}, we will make it explicit for some concrete cases. Let us emphasise here the importance of viewing the natural isomorphism \(\tau\) as a structure the product \(\tensor\) comes equipped with (rather than viewing its existence as a mere property of the product); this will become crucial in Section~\ref{Section Varieties}, in particular in the proof of the Recognition Theorem~\ref{Recognition Theorem}.

We shall further see that \(\tensor\) need not be part of a monoidal structure on \(\X\), or even on \(\Ab(\X)\), since the bilinear product can be non-associative or lack a unit (Example~\ref{non-associative bilinear product} and Subsection~\ref{Associative algebras}, respectively). Before giving examples of such situations, we list some properties which the bilinear product does satisfy. We may already prove that bilinear products are abelian; more advanced properties will be discussed in Section~\ref{Section Further Properties}, which depends on the concept of higher commutators of Section~\ref{Section Nilpotentisation}, of which properties are developed there.

\begin{proposition}\label{Proposition-Abelian}
  For all objects \(X\), \(Y\) in a semi-abelian category \(\X\), we have \(X\tensor_{\X}Y\in\Ab(\X)\). This makes the bilinear product on \(\X\) into a functor
  \[
    \tensor_{\X}\colon \X\times\X\to\Ab(\X)\colon (X,Y)\mapsto X\tensor_{\X}Y\text{.}
  \]
\end{proposition}
\begin{proof}
  We give two proofs, each of which is instructive. The first one consists in applying Proposition \ref{linear functors have abelian values} to the linear functor \(X\tensor -\colon \X\to\X\). The second proof is based on the use of commutator calculus. By definition of the bilinear product (\ref{Subsection Bilinear Product}), we may assume that \(X\) and \(Y\) are two-nilpotent. The commutator
  \[
    [X\cosmash_2Y,X\cosmash_2Y]\leq X+_2Y
  \]
  in \(\Nil_2(\X)\) is then indeed trivial, because \(X\cosmash_2Y=[i_1(X),i_2(Y)]\), so that
  \begin{align*}
    [X\cosmash_2Y,X\cosmash_2Y] & =\bigl[[i_1(X),i_2(Y)],[i_1(X),i_2(Y)]\bigr] \\
                                & \leq [i_1(X),i_2(Y),i_1(X),i_2(Y)]           \\
                                & \leq [i_1(X),i_1(X),i_2(Y)]=0
  \end{align*}
  since removing brackets or duplicate objects makes a commutator larger~\cite[Proposition 2.21]{HVdL} and the category \(\Nil_2(\X)\) is two-nilpotent.
\end{proof}

\section{Higher-order commutators and nilpotency}\label{Section Nilpotentisation}

In what follows, we need commutators of arbitrary length, and the concept of nilpotency that comes with it. Then, the two-nilpotentisation functor considered in Section~\ref{Section Two-Nilpotency} extends to arbitrary nilpotency degrees. Our main aim in this section is to prove that all nilpotentisation functors commute with Birkhoff reflectors: Theorem~\ref{Nil vs Birkhoff} generalises Theorem~3.4 in~\cite{PVdL1} from abelianisation to arbitrary nilpotentisation degrees.

Recall from~\cite{HVdL, Actions} that commutators of an arbitrary length are defined similarly to the cases of length \(2\) and \(3\) by means of cosmash products. These cosmash products may either be seen as instances of cross-effects as recalled in Section~\ref{Section Cross-Effects}, or explicitly as follows.

\subsection{Cosmash products and commutators}
Given objects \(X_1\), \dots, \(X_n\), their \defn{cosmash product} \(X_1\cosmash\cdots\cosmash X_n\) is the kernel \(\iota_{X_1,\dots,X_n}\colon X_1\cosmash\cdots\cosmash X_n\to X_1+\cdots+ X_n\) of the morphism
\[
  r=r_{X_1,\ldots,X_n}\colon X_1 + \dots + X_n \to \prod_{j=1}^n (X_1 + \dots + \widehat{X_j} + \dots + X_n)
\]
which is such that \(\pi_j \comp r \comp i_k = i_{X_k}\) if \(k \neq j\) and \(0\) if \(k=j\), for all \(j\), \(k \in \{1,\dots,n\}\).

Given a natural number \(n\geq 2\), an object \(X\), and \(n\) subobjects \(K_i\leq X\) with respective representing monomorphisms \(k_i\colon K_i\to X\) (for \(1\leq i\leq n\)), the commutator \([K_1,\dots,K_n]\leq X\) is defined as the image of the composite morphism
\[
  \xymatrix@=4em{X_1\cosmash\cdots\cosmash X_n \ar@{{ |>}->}[r]^{\iota_{X_1,\ldots,X_n}} & X_1+\cdots+ X_n \ar[r]^-{\DCoindArr{k_1 \and \cdots \and k_n}} &X\text{.}}
\]

As a consequence of Proposition~\ref{crossdecomp'}, Lemma~\ref{Lemma-Join} now extends to higher commutators as follows~\cite[Proposition 2.22]{HVdL}:

\begin{lemma}\label{Lemma-Long-Join}
  If \(n\geq 1\) and \(K_1\), \dots, \(K_n\), \(L\), \(M\leq X\) then
  \[
    [K_1,\dots, K_n,L\join M]=[K_1,\dots, K_n,L]\join [K_1,\dots, K_n,M]\join [K_1,\dots, K_n,L,M]
  \]
  where all joins are computed in \(X\).\noproof
\end{lemma}

Note that cosmash products, and hence commutators, are symmetric in their variables, so that this join decomposition formula holds in all slots of the commutator, not just the last one.

\subsection{The lower central series}\label{The lower central series}
Given an object \(X\) of \(\X\), we write \(\gamma_{n}(X)=[\underbrace{X,\dots,X}_{\text{\(n\) terms}}]\) to obtain the \defn{lower central series}
\[
  X=\gamma_{1}(X) \geq \gamma_{2}(X)\geq\cdots\geq \gamma_{n}(X)\geq\cdots
\]
of \(X\). Note that, by the above, the object \(\gamma_{n}(X)\) can be obtained explicitly as the image \(i^X_n\colon \gamma_n(X)\to X\) of the composite morphism
\begin{equation*}
  \xymatrix@=4em{X^{\diamond n}\ar@{{ |>}->}[r]^{\iota_{X,\ldots,X}} & X^{+ n}  \ar@{- >>}[rr]^-{\nabla^n_X=(1_X\ldots 1_X)} &&X}
\end{equation*}
whose source is the \(n\)-fold cosmash product \(X^{\diamond n} = X \diamond \cdots\diamond X\). Note that the natural transformation \(i_n\colon \gamma_n\To 1_\X\) represents \(\gamma_n\colon \X\to \X\) as a normal subfunctor of \(1_\X\). (It is the image of the kernel \(\iota\) along the natural transformation \(\nabla^n\).)

\subsection{\texorpdfstring{The \(n\)-nilpotent core}{The n-nilpotent core}}\label{Subsection n-nilpotent core}
Recall~\cite{Actions,HVdL,SVdL3} (but also~\cite{BeBou} which uses slightly different terminology) that for \(n\geq 1\), an object \(U\) is called \defn{\(n\)-nilpotent} or \defn{nilpotent of class \(\le n\)} when the \((n+1)\)-ary commutator \([U,\dots,U]\) vanishes. This determines a Birkhoff subcategory of \(\X\) denoted \(\Nil_{n}(\X)\) and called the \defn{\(n\)-nilpotent core} of~\(\X\). The reflector \(\nil_{n}\colon{\X\to \Nil_{n}(\X)}\) is given by \(\nil_{n}(X)=X/[X,\dots,X]\). We write \(\nilunit_n\colon 1_{\X}\to \nil_{n}\) for the unit of the adjunction. Note that \emph{\(1\)-nilpotent} is the same as \emph{abelian}.

Theorem~\ref{Nil vs Birkhoff} depends on the next lemma.

\begin{lemma}\label{Subvariety Lemma}
  Given a semi-abelian category \(\X\) with a Birkhoff subcategory \(\B\), for any \(n\geq 2\) and all objects \(B_1\), \dots, \(B_n\) of \(\B\) there is a natural quotient morphism
  \[
    \nu_{B_1,\dots,B_n}\colon B_1\cosmash_\X\cdots\cosmash_\X B_n\to B_1\cosmash_\B\cdots\cosmash_\B B_n
  \]
  between the cosmash product of the \(B_i\) in the category \(\X\) and the same cosmash product computed in \(\B\). This morphism is the canonical restriction of the unit of the reflection from \(\X\) to \(\B\) at the coproduct of the \(B_i\), computed in \(\X\).
\end{lemma}
\begin{proof}
  This is essentially a consequence of Lemma~2.8 in~\cite{SVdL3}. In explaining why, we here freely use the concepts and results in that article, leading up to this lemma. We prove that the restriction of the component
  \[
    \eta_{B_1+_\X\cdots+_\X B_n}\colon B_1+_\X\cdots+_\X B_n\to B_1+_\B\cdots+_\B B_n
  \]
  of the adjunction unit \(\eta\) to cosmashes is a regular epimorphism. Indeed, the natural comparison morphism \(E_\X(B_1,\dots,B_n)\to E_\B(B_1,\dots,B_n)\) induced by \(\eta\) is an \((n+1)\)-cubic extension, as a pointwise regular epimorphism between split \(n\)-cubic extensions. The lemma mentioned above implies that the induced morphism \(\nu_{B_1,\dots,B_n}\colon{B_1\cosmash_\X\cdots\cosmash_\X B_n\to B_1\cosmash_\B\cdots\cosmash_\B B_n}\) is a \(1\)-cubic extension, hence a regular epimorphism, in \(\X\).
\end{proof}

\begin{remark}
  Alternatively, this result can be viewed as a consequence of a general theorem of functor calculus stating that the functor mapping a reduced functor to its \(n\)-th cross effect is exact. As in the classical case of functors with values in abelian categories this is a straightforward application of the \(3\times 3\)-Lemma for \(n=2\) and then follows by induction for higher \(n\).
\end{remark}

\begin{theorem}\label{Nil vs Birkhoff}
  Let \(\X\) be a semi-abelian category and \(n\geq 1\). Any Birkhoff reflector \(F\colon{\X\to \B}\) commutes with \(\nil_{n}\), up to isomorphism (square on the left) and under~\(F\) (triangle on the right).
  \[
    \vcenter{\xymatrix{\X \ar@{}[rd]|(.55){\alpha}|(.6){\swarrow} \ar[d]_-{\nil_n^\X} \ar[r]^-{F} & \B \ar[d]^-{\nil_n^\B}\\
    \Nil_n(\X)  \ar[r]_-{F} & \Nil_n(\B)}}
    \qquad\qquad
    \vcenter{\xymatrix{& F(X) \ar[dl]_-{F((\nilunit_n^{\X})_X)} \ar[dr]^-{(\nilunit_n^{\B})_{F(X)}}\\
    F(\nil^\X_n(X)) && \nil^\B_n(F(X)) \ar[ll]^-{\alpha_X}}}
  \]
  More precisely, there exists a natural isomorphism \(\alpha\colon \nil_n^{\B} \comp F \To F\comp \nil_n^{\X}\) such that \(\alpha_X \comp (\nilunit_n^{\B})_{F(X)} = F((\nilunit_n^{\X})_X)\) for each object \(X\) of \(\X\).
\end{theorem}
\begin{proof}
  For an object \(X\) in \(\X\), its reflection \(\nil^\X_n(X)\) into \(\Nil_n(\X)\) is the cokernel \((\nilunit_n^\X)_X\colon X\to \nil^\X_n(X)\) of the composite
  \[
    \xymatrix{X\cosmash_\X\cdots\cosmash_\X X\ar[r]^-{\iota_{X,\dots, X}} & X+_\X\cdots+_\X X \ar[r]^-{\nabla^{n+1}_X} & X\text{,}}
  \]
  where the first morphism is the canonical inclusion and the second morphism is the \((n+1)\)-fold codiagonal. Since \(F\) is cocontinuous, we find \(F(\nil^\X_n(X))\) as the cokernel \(F((\nilunit_n^{\X})_X)\) of the composite
  \[
    \xymatrix@=4em{F(X\cosmash_\X\cdots\cosmash_\X X)\ar[r]^-{F(\iota_{X,\dots, X})} & F(X+_\X\cdots+_\X X) \ar[r]^-{F(\nabla^{n+1}_X)} & F(X)\text{.}}
  \]
  Likewise, the reflection \(\nil^\B_n(F(X))\) of \(F(X)\) into \(\Nil_n(\B)\) is the cokernel
  \[
    (\nilunit_n^{\B})_{F(X)}\colon F(X)\to \nil^\B_n(F(X))
  \]
  of the composite
  \[
    \xymatrix@=4em{{F(X)}\cosmash_\B\cdots\cosmash_\B {F(X)}\ar[r]^-{\iota_{{F(X)},\dots, {F(X)}}} & {F(X)}+_\B\cdots+_\B {F(X)} \ar[r]^-{\nabla^{n+1}_{F(X)}} & {F(X)}\text{.}}
  \]
  Lemma~\ref{Subvariety Lemma} provides us with the dotted regular epimorphism \(\nu_{F(X),\dots,F(X)}\) induced by the commutative diagram of plain arrows in Figure~\ref{Figure Nil Cube} where \(\varphi\) is induced by the universal property of the unit \(\eta_{X+_\X\cdots+_\X X}\).
  \begin{figure}
    \resizebox{\textwidth}{!}{
    \(
    \xymatrix@!0@C=9em@R=4.5em{X\cosmash_\X\cdots\cosmash_\X X \ar@{-{ >>}}[rd]|-{\eta_{X\cosmash_\X\cdots\cosmash_\X X}} \ar@{-{ >>}}[dd]_-{\eta_X\cosmash_\X\cdots\cosmash_\X \eta_X} \ar@{{ |>}->}[rr]^-{\iota_{X,\dots,X}} && X+_\X\cdots+_\X X \ar@{-{ >>}}[dd]_(.25){\eta_X+_\X\cdots+_\X \eta_X} \ar@{-{ >>}}[rd]^-{\eta_{X+_\X\cdots+_\X X}}\\
    & F(X\cosmash_\X\cdots\cosmash_\X X) \ar@{->}[rr]|(.5){\hole}_(.65){F(\iota_{X,\dots,X})} \ar@{--{ >>}}[dd]|(.45){\hole}|(.5){\hole}|(.55){\hole}^(.7){\sigma} && F(X+_\X\cdots+_\X X) \ar[dd]^{\varphi}\\
    F(X)\cosmash_\X\cdots\cosmash_\X F(X) \ar@{.{ >>}}[rd]_-{\nu_{F(X),\dots,F(X)}} \ar@{{ |>}->}[rr]^-{\iota_{F(X),\dots,F(X)}} && F(X)+_\X\cdots+_\X F(X) \ar@{-{ >>}}[rd]|-{\eta_{F(X)+_\X\cdots+_\X F(X)}} \\
    & F(X)\cosmash_\B\cdots\cosmash_\B F(X) \ar@{{ |>}->}[rr]_-{\iota_{F(X),\dots,F(X)}} && F(X)+_\B\cdots+_\B F(X)}
    \)
    }
    \caption{Comparing nilpotent reflections}\label{Figure Nil Cube}
  \end{figure}
  Orthogonality of regular epimorphisms with respect to monomorphisms applied to the commutative square
  \[
    \xymatrix@C=8em{X\cosmash_\X\cdots\cosmash_\X X \ar[r]^-{\nu_{F(X),\dots,F(X)}\circ \eta_{X\cosmash_\X\cdots\cosmash_\X X}} \ar@{-{ >>}}[d]_-{\eta_{X\cosmash_\X\cdots\cosmash_\X X}} & F(X)\cosmash_\B\cdots\cosmash_\B F(X) \ar@{{ |>}->}[d]^-{\iota_{F(X),\dots,F(X)}}\\
    F(X\cosmash_\X\cdots\cosmash_\X X) \ar@{-->}[ru]^-{\sigma} \ar[r]_-{\varphi\circ F(\iota_{X,\dots,X})} & F(X)+_\B\cdots+_\B F(X)}
  \]
  provides us with the dashed lifting \(\sigma\) in the diagram, which is a regular epimorphism as well (since \(\nu_{F(X),\dots,F(X)}\circ \eta_{X\cosmash_\X\cdots\cosmash_\X X}\) is one by composition). It satisfies
  \begin{equation*}
    \nabla^{n+1}_{F(X)} \comp \iota_{F(X),\dots,F(X)} \comp \sigma  = \nabla^{n+1}_{F(X)} \comp \varphi\comp F(\iota_{X,\dots,X}) = F(\nabla^{n+1}_X)\comp F(\iota_{X,\dots, X})\text{,}
  \end{equation*}
  where the last equality holds because \(\eta_{X+_\X\cdots+_\X X}\) is an epimorphism. This proves that the cokernels determining \(F(\nil^\X_n(X))\) and \(\nil^\B_n(F(X))\) mentioned above are isomorphic. Let \(\alpha_X\) denote the unique isomorphism such that \(\alpha_X \comp (\nilunit_n^{\B})_{F(X)} = F((\nilunit_n^{\X})_X) \); it clearly satisfies the final claim in the statement of the theorem.
\end{proof}

Note that we use this result in the examples of loops and internal crossed modules worked out in Subsections~\ref{Loops} and~\ref{crossed modules}, respectively, and also to study nilpotency of low class of algebras over certain operads in Subsection~\ref{Nilpotency for operads}.

\section{Further properties of the bilinear product}\label{Section Further Properties}

\subsection{Description of the bilinear product in terms of \(\X\) alone.}\label{Subsection-Folding}
To use the above definition of bilinear product in practice, some knowledge about the two-nilpotent core of the category at hand is needed. We now explain how this may be avoided: we describe the bilinear product as a quotient of the cosmash product in the category \(\X\) itself. In order to do so, we first need to introduce some notation.

Given objects \(X\) and \(Y\) in~\(\X\), we consider the \defn{folding operations}
\[
  S_{1,2}^{X,Y}\colon {X\cosmash Y\cosmash Y\to X\cosmash Y}\qquad\text{and}\qquad S_{2,1}^{X,Y}\colon {X\cosmash X\cosmash Y\to X\cosmash Y}\text{.}
\]
The one on the left is induced by
\[
  \xymatrix@=3em{X\cosmash Y\cosmash Y \ar[r]^-{\iota_{X,Y,Y}} & X+Y+Y \ar[r]^-{1_{X}+\nabla_{Y}} & X+Y;
  }\]
indeed, \(\Im((1_{X}+\nabla_{Y})\comp \iota_{X,Y,Y}) = [i_1(X),i_2(Y),i_2(Y)] \subobj [i_1(X),i_2(Y)] = X \cosmash Y\). Being a restriction of the direct image of a kernel, the image
\[
  S_{1,2}^{X,Y}(X\cosmash Y\cosmash Y) = [i_1(X),i_2(Y),i_2(Y)]\normal X\cosmash Y
\]
of \(S_{1,2}^{X,Y}\) is normal in \(X\cosmash Y\). Likewise, \(S_{2,1}^{X,Y}(X\cosmash X\cosmash Y)\normal X\cosmash Y\). We denote their join in \(X\cosmash Y\), which of course also is a normal subobject~\cite{Borceux-Semiab,Huq,EverVdLRCT}, by \(N_{X,Y}\). By the above,
\[
  N_{X,Y} = [i_1(X),i_2(Y),i_2(Y)] \join [i_1(X),i_1(X),i_2(Y)]\normal X\cosmash Y
\]
which is useful in the proof of the next result.

\begin{proposition}
  \label{prop:internal descr of tensor}
  If \(X\) and \(Y\) are objects in a semi-abelian category~\(\X\), then their bilinear product \(X\tensor_{\X} Y\) is isomorphic to the quotient \((X\cosmash Y)/N_{X,Y}\) of the cosmash product \(X\cosmash Y\) by the normal subobject~\(N_{X,Y}\). In other words, we have a short exact sequence
  \[
    \xymatrix{0 \ar[r] & N_{X,Y} \ar@{{ |>}->}[r] & X\cosmash Y \ar@{-{ >>}}[r]^{\eta_{X,Y}} & X\tensor_{\X}Y \ar[r] & 0\text{.}}
  \]
\end{proposition}

Note that this is not exactly a particular case of Lemma~\ref{Subvariety Lemma} since we do not ask that \(X\) and \(Y\) are two-nilpotent. However, a straightforward generalisation of that lemma does indeed hold, where a cosmash product of objects in \(\X\) is compared with the cosmash of their reflections in \(\B\). On the other hand, what interests us here is the precise shape of the kernel of the comparison morphism \(\eta_{X,Y}\).

In the next proof, we freely use the calculation rules for Higgins commutators, listed in Proposition~2.21 of~\cite{HVdL}. For instance: commutators are symmetric and preserved under direct images (for a morphism \(f\colon A\to B\) and subobjects \(A_1\), \dots, \(A_n\) of \(A\), we have \(f[A_1,\ldots,A_n] =[f(A_1),\ldots,f(A_n)]\)); removing brackets or duplicate elements enlarges a commutator.

\begin{proof}
  We have to prove that the inclusion of \(\gamma_3(X+Y) = [X+Y,X+Y,X+Y]\normal X+Y\) into the cosmash product \(X\cosmash Y\) in~\(\X\) coincides with the object \(N_{X,Y}\): see Figure~\ref{Figure N}.
  \begin{figure}
    \begin{tikzcd}
      & 0 \arrow[d]& 0 \arrow[d] && 0 \arrow[d] \\
      0 \arrow[r] & N_{X,Y} \arrow[r, nmono] \arrow[d, nmono] & X \cosmash Y \arrow[rr, "\eta_{X,Y}", induced] \arrow[d, "\iota_{X,Y}",nmono] && X \tensor_{\X} Y \arrow[r] \arrow[d, nmono] & 0 \\
      0 \arrow[r] & {\gamma_3(X+Y)} \arrow[r,"i_3^{X+Y}", nmono] \arrow[d, "\alpha", induced]& X+Y \arrow[rr, "\nilunit_2^{X+Y}", repi] \arrow[d, "r_{X,Y}",repi] && \nil_2(X) +_2 \nil_2(Y) \arrow[r] \arrow[d, repi] & 0 \\
      0 \arrow[r] & {\gamma_3(X)} \times {\gamma_3(Y)}  \ar[d]  \arrow[r,"i_3^{X}\times i_3^{Y}"', nmono]   &  X \times Y \arrow[rr, "\nilunit_2^X\times \nilunit_2^Y"',repi] \arrow[d] && \nil_2(X) \times \nil_2(Y) \arrow[d]\arrow[r]&0 \\
      & 0& 0 && 0
    \end{tikzcd}
    \caption{Proving that \(N_{X,Y}=[X+Y,X+Y,X+Y]\cap (X\cosmash Y)\)}\label{Figure N}
  \end{figure}
  First of all we decompose the commutator \([X+Y,X+Y,X+Y]\) by rewriting the sum as a join in order to be able to repeatedly use the distributive law (Lemma~\ref{Lemma-Long-Join}) together with the fact that removing duplicate objects makes a commutator larger, that commutators are symmetric, and that they are preserved under direct images. We find:
  \begin{align*}
    [X+Y,X+Y,X+Y] & = [i_1(X)\join i_2(Y),
    i_1(X)\join i_2(Y),i_1(X)\join i_2(Y)]                                           \\
                  & =[i_1(X), i_2(Y),i_2(Y)] \join [i_1(X), i_1(X),i_2(Y)]           \\
                  & \quad\join [i_1(X), i_1(X),i_1(X)] \join [i_2(Y), i_2(Y),i_2(Y)] \\
                  & = N_{X,Y}\join i_1([X,X,X]) \join i_2([Y,Y,Y])\text{.}
  \end{align*}

  Consider the commutative diagram in Figure~\ref{Figure N} of which the second and third row and second and third column are exact, so that commutativity of the right-hand bottom square induces the existence of the morphisms \(\eta_{X,Y}\) and \(\alpha\). So the \(3\times 3\)-Lemma provides the desired exactness of the top row, once we have shown exactness of the first column. To see this, first note that
  \begin{align*}
    [i_1([X,X,X]) , i_2([Y,Y,Y])] & \subobj
    [i_1(X), i_1(X),i_1(X),i_2(Y),i_2(Y),i_2(Y)]                                    \\
                                  & \subobj [i_1(X), i_1(X),i_2(Y)] \subobj N_{X,Y}
  \end{align*}
  since removing brackets enlarges the commutator. This provides the upper left-hand square in the commutative diagram with exact rows in Figure~\ref{Figure Beta Delta}, which in turn induces the map \(\beta\), while the bottom left-hand square induces \(\delta\).
  \begin{figure}
    \hfil\xymatrix{
    [i_1(\gamma_3(X)),i_2(\gamma_3(Y))] \ar@{{ >}->}[d]  \ar@{{ |>}->}[r] &\gamma_3(X)+\gamma_3(Y) \ar@{- >>}[rr]^-{r_{\gamma_3(X),\gamma_3(Y)}} \ar[d]^-{\DCoindArr{\gamma_3(i_1) \and \gamma_3(i_2)}} & &\gamma_3(X)\times\gamma_3(Y) \ar@{.>}[d]^-{\beta}\\
    N_{X,Y} \ar@{{ >}->}[d] \ar@{{ |>}->}[r] & \gamma_3(X+Y) \ar@{{ >}->}[d]^-{i_3^{X+Y}} \ar@{- >>}[rr]^-{q} & &\gamma_3(X+Y)/N_{X,Y} \ar@{.>}[d]^-{\delta}\\
    [i_1(X),i_2(Y)] \ar@{{ |>}->}[r]_-{\iota_{X,Y}} & X+Y \ar@{- >>}[rr]_-{r_{X,Y}} & &X\times Y
    }\hfil
    \caption{Constructing the morphisms \(\beta\) and \(\delta\)}\label{Figure Beta Delta}
  \end{figure}
  Now, by naturality of \(i_3\),
  \begin{align*}
    \delta\comp \beta\comp r_{\gamma_3(X),\gamma_3(Y)} & = r_{X,Y}\comp
    i_3^{X+Y}\comp \CoindArr{\gamma_3(i_1) \and \gamma_3(i_2)}                               \\
                                                       & =  r_{X,Y} \comp  (i_3^{X} + i_3^Y)
    = (i_3^{X} \times i_3^Y) \comp r_{\gamma_3(X),\gamma_3(Y)}\text{,}
  \end{align*}
  whence \(\delta\comp \beta = i_3^{X} \times i_3^Y\) is a monomorphism, and thus so is \(\beta\). But by the decomposition of \(\gamma_3(X+Y)\) as the join \(N_{X,Y}\join i_1([X,X,X]) \join i_2([Y,Y,Y])\) obtained above, the composite \(\beta\comp r_{\gamma_3(X),\gamma_3(Y)} = q\comp \CoindArr{\gamma_3(i_1) \and \gamma_3(i_2)} \) is a regular epimorphism, hence so is \(\beta\). So \(\beta\) is an isomorphism and thus \(\delta\) is a monomorphism. Now \( (i_3^{X} \times i_3^Y)\comp \alpha = r_{X,Y} \comp  i_3^{X+Y} = \delta \comp q\), whence \(\Ker(\alpha) = \Ker((i_3^{X} \times i_3^Y)\comp \alpha) = \Ker(\delta \comp q) = \Ker(q) = N_{X,Y}\). Thus the first column of the diagram in Figure~\ref{Figure N} is exact in \(\gamma_3(X+Y)\). To show that \(\alpha\) is a regular epimorphism, note that \(\im((i_3^{X} \times i_3^Y)\comp \alpha) = \delta = \im(\delta\comp \beta) = i_3^{X} \times i_3^Y\), so there exists a regular epimorphism \(q'\colon \gamma_3(X+Y) \to\gamma_3(X)\times\gamma_3(Y) \) such that \((i_3^{X} \times i_3^Y)\comp \alpha = (i_3^{X} \times i_3^Y)\comp q'\). Thus \(\alpha =q'\) is a regular epimorphism as desired.
\end{proof}

\begin{remark}
  This may be used in an alternative proof of Proposition~\ref{Proposition-Abelian}. Indeed, the composite of canonical morphisms
  \[
    (X\cosmash Y)\cosmash (X\cosmash Y) \to X\cosmash Y\cosmash X\cosmash Y \to X\cosmash X\cosmash Y\text{,}
  \]
  together with the analogous arrow into \(X \cosmash Y \cosmash Y\), induces an inclusion \([X\cosmash Y,X\cosmash Y]\leq N_{X,Y}\). As a consequence, the regular epimorphism \(\eta_{X,Y}\colon{X\cosmash Y\to X\tensor_{\X}Y}\) factors through \(\ab(X\cosmash Y)\). The result now follows because the class of abelian objects is closed under quotients~\cite[Proposition~1.6.11]{Borceux-Bourn}.
\end{remark}

\begin{remark}\label{tensor = bilinearisation of cosmash}
  From the viewpoint of functor calculus, the bifunctor
  \[
    {(X,Y)\mapsto \frac{X\cosmash Y}{N_{X,Y}}}
  \]
  is in fact the \defn{bilinearisation} of the bifunctor given by the cosmash product---see~\cite{Hartl-Vespa} for the case of \(\Ab\)-valued functors, and~\cite{CCC} for the general case. So Proposition~\ref{prop:internal descr of tensor} states that the bilinear product of \(\X\) can be described as the bilinearisation of its cosmash product.
\end{remark}

\begin{example}\label{bilinearisation of the cosmash product of groups}
  As mentioned above, a classical theorem of combinatorial group theory states that for groups \(G\), \(H\) the kernel of the canonical map from the coproduct \(G+H=G*H\) to \(G\times H\) (in our terms, \(G\cosmash H\)) is free on the commutators \([g,h]\), for \((g,h)\in G^*\times H^*\) where \(K^*=K\backslash \{1\}\). Hence as a bifunctor, \(G\cosmash H\) can be described as the quotient of the free group on \(G\times H\) divided by the normal subgroup generated by \((G\times \{1\}) \cup (\{1\}\times H)\). It was shown in~\cite{Hartl-Vespa} that the bilinearisation of this bifunctor is the tensor product \(\ab(G)\tensor_\Z \ab(H)\), an observation which suggested to us that the bilinearisation of the cosmash product might be a suitable generalisation of the classical tensor product of groups to arbitrary semi-abelian categories, and thus led to this paper.
\end{example}

\subsection{From cosmash product to bilinear product.}
The collection of quotients \(\eta_{X,Y}\colon{X\cosmash Y\to X\tensor_{\X}Y}\) from the cosmash product to the bilinear product clearly form a natural transformation \({\cosmash\To\tensor}\). The passage is well-behaved with respect to regular epimorphisms, in the sense of Proposition~\ref{Proposition-cosmash-to-Tensor} below.

To see this, first recall from~\cite{Bourn2003} that a commutative square such as \(f\comp b=a\comp f'\) below is called a \defn{double extension} or a \defn{regular pushout square} when all its arrows and the comparison \((b,f') \from B'\to B\times_{A}A'\) to the induced pullback are regular epimorphisms. Note that a regular pushout square is always a pushout, as follows from the next result.

\begin{lemma}\label{Lemma-LeftRight}\cite{Bourn2001, EGVdL}
  Consider, in a homological category, a commutative diagram with exact rows
  \begin{equation*}\label{Short-Five-Lemma}
    \vcenter{\xymatrix{0 \ar[r] & K' \ar@{{ |>}->}[r] \ar[d]_-k & B' \ar@{-{ >>}}[r]^-{f'} \ar[d]_-b & A' \ar[d]^-a \ar[r] & 0\\
    0 \ar[r] & K \ar@{{ |>}->}[r] & B \ar@{-{ >>}}[r]_-f & A \ar[r] & 0\text{.}}}
  \end{equation*}
  \begin{enumerate}
    \item \label{pb iff k iso} The right hand square \(f\comp b=a\comp f'\) is a pullback if and only if \(k\) is an isomorphism.
    \item \label{reg po iff k repi} Suppose that \(b\) and \(a\) are regular epimorphisms. Then the square \(f\comp b=a\comp f'\) is a regular pushout if and only if \(k\) is a regular epimorphism.
  \end{enumerate}
\end{lemma}
\begin{proof}
  \eqref{pb iff k iso} is Lemma~1 and Proposition~7 of~\cite{Bourn2001} combined. \eqref{reg po iff k repi} follows from Proposition~8 in~\cite{Bourn2001}.
\end{proof}

\begin{proposition}\label{Proposition-cosmash-to-Tensor}
  Suppose \(f\colon{B\to A}\) and \(g\colon {D\to C}\) are regular epimorphisms. Then the commutative square
  \[
    \xymatrix{B\cosmash D \ar@{-{ >>}}[d]_-{\eta_{B,D}} \ar@{-{ >>}}[r]^-{f\cosmash g} & A\cosmash C \ar@{-{ >>}}[d]^-{\eta_{A,C}}\\
    B\tensor D \ar@{-{ >>}}[r]_-{f\tensor g} & A\tensor C}
  \]
  is a regular pushout.
\end{proposition}
\begin{proof}
  Cosmash products preserve regular epimorphisms by~\cite[Proposition~2.9]{HVdL}. Clearly, also joins do. By Lemma~\ref{Lemma-LeftRight} it now suffices to note that the induced morphism \({N_{B,D}\to N_{A,C}}\) is a regular epimorphism, because it is a join of two regular epimorphisms.
\end{proof}

\begin{proposition}\label{Proposition Nilpotent}
  Let \(Z\) be a two-nilpotent object and \(x\colon {X\to Z}\), \(y\colon{Y\to Z}\) two morphisms with codomain \(Z\). Then the induced morphism \(\CoindArr{x \and y}\comp\iota_{X,Y}\colon {X\cosmash Y \to Z}\) factors through \(\eta_{X,Y}\) to a morphism \({X\tensor Y \to Z}\).
\end{proposition}
\begin{proof}
  As the diagram
  \[
    \begin{tikzcd}
      X \cosmash Y \arrow[r, "\iota_{X,Y}", nmono] \arrow[d, "\eta_{X,Y}"', repi] & X+Y \arrow[rr, "\CoindArr{x \and y}"] \arrow[d, repi] && Z \arrow[d, equal] \\
      X \tensor Y \arrow[r, nmono] & \nil_2(X+Y) \arrow[rr, "\nil_2(\CoindArr{x \and y})"'] && \nil_2(Z) = Z
    \end{tikzcd}
  \]
  indicates, the needed morphism is the restriction of \(\nil_2(\CoindArr{x \and y})\) to \(X\tensor Y=\nil_{2}(X)\cosmash_2\nil_{2}(Y)\).
\end{proof}

\subsection{Sequential right exactness.}

We now show that the functor \(X\tensor (-)\) is \defn{sequentially right exact}, which means that it preserves right exact sequences in the following sense. A sequence of the form
\[
  \begin{tikzcd}
    K \arrow[r, "k"] & B \arrow[r, "f", repi] & A \arrow[r] & 0
  \end{tikzcd}
\]
is \defn{right exact} whenever \(k\) is a \defn{proper arrow}, i.e.\ an arrow whose image is a normal subobject of its codomain, and \(f\) is a cokernel of \(k\). Note that, when an arrow \(k \colon K \to B\) is proper, the image \(i \colon I \to B\) of \(k\) is given by the kernel of the cokernel of \(k\).

The sequential right exactness of \(X \cosmash (-)\) is a consequence (Proposition~\ref{Proposition-RightExact-Tensor}) of the following theorem, of which the proof uses cross-effects and is deferred to Section~\ref{Section Cross-Effects}, page \pageref{Proof of Theorem-RightExact-HVdL}. Note that all these results crucially depend on Barr exactness of the category, as explained in Subsection~\ref{Subsection Basic Principle}. This is the main reason we chose the context of semi-abelian categories as our work environment\footnote{The definition of the bilinear product itself, however, makes sense in all homological categories.}.

\begin{theorem}\label{Theorem-RightExact-HVdL}
  Suppose that \(\X\) is a semi-abelian category. Consider an object~\(X\) and a cokernel
  \[
    \xymatrix{K \ar[r]^-{k} & B \ar@{-{ >>}}[r]^-{f} & A \ar[r] & 0}
  \]
  in~\(\X\). Then we obtain a right exact sequence
  \[
    \resizebox{\textwidth}{!}
    {\xymatrix@=4em{(X\cosmash K\cosmash B) \rtimes (X\cosmash K) \ar[rr]^-{\DCoindArr{S_{1,2}^{X,B}\circ (1_{X}\cosmash k\cosmash 1_{B}) \and 1_{X}\cosmash k}} && X\cosmash B \ar@{-{ >>}}[r]^-{1_{X}\cosmash f} & X\cosmash A \ar[r] & 0}}
  \]
  in~\(\X\).
\end{theorem}

Note that the morphism \(S_{1,2}^{X,B}\circ (1_{X}\cosmash k\cosmash 1_{B})\colon X\cosmash K\cosmash B\to X\cosmash B\) factors through \(X\cosmash B\cosmash B\), so that it vanishes when we pass to the quotient~\({X\tensor B}\). Hence:

\begin{proposition}\label{Proposition-RightExact-Tensor}
  Suppose that \(\X\) is a semi-abelian category. Consider an object~\(X\) and a cokernel
  \[
    \xymatrix{K \ar[r]^-{k} & B \ar@{-{ >>}}[r]^-{f} & A \ar[r] & 0}
  \]
  in~\(\X\). Then we obtain the right exact sequence of abelian objects
  \[
    {\xymatrix@=5em{X\tensor K \ar[r]^-{1_{X}\tensor k} & X\tensor B \ar@{-{ >>}}[r]^-{1_{X}\tensor f} & X\tensor A \ar[r] & 0\text{.}}}
  \]
\end{proposition}
\begin{proof}
  This is a consequence of Theorem~\ref{Theorem-RightExact-HVdL}.
  By Proposition~\ref{Proposition-cosmash-to-Tensor}, the regular epimorphism \(1_{X}\tensor f\) is the pushout of \(1_X\cosmash f\) along \(\eta_{X,B}\), and hence the cokernel of
  \begin{equation}\label{semidir}
    \eta_{X,B}\comp\DCoindArr{S_{1,2}^{X,B}\circ (1_{X}\cosmash k\cosmash 1_{B}) \and 1_{X}\cosmash k}\colon (X\cosmash K\cosmash B) \rtimes (X\cosmash K) \to X\tensor B\text{.}
  \end{equation}
  But on \(X\cosmash K\cosmash B\) this morphism vanishes, so this term may be removed without changing the cokernel. Now it suffices to use that \(\eta_{X,B}\comp(1_{X}\cosmash k)=(1_{X}\tensor k)\comp \eta_{X,K}\) and \(1_{X}\tensor k\) have the same cokernel (since \(\eta_{X,K}\) is an epimorphism). Furthermore, the morphism \(1_{X}\tensor k\) is proper, since such is~\eqref{semidir}, so that its image is the kernel of~\(1_{X}\tensor f\).
\end{proof}

Note that this result is stronger than sequential right exactness since we do not ask \(k\) to be proper.

\begin{corollary}\label{Corollary-Products-Tensor}
  Suppose that \(\X\) is a semi-abelian category. Then for any object~\(X\) in~\(\X\), the functor \(X\tensor (-)\colon \X\to \Ab(\X)\) preserves finite products.
\end{corollary}
\begin{proof}
  The functor \(X\tensor (-)\) being reduced, it preserves the terminal object \(1=0\). Binary products are preserved because they may be presented as split short exact sequences.
\end{proof}

The bilinear product commutes with abelianisation:

\begin{corollary}\label{Tensor commutes with Ab}
  For all objects \(X\), \(Y\) of \(\X\) there exist natural isomorphisms
  \[
    X\tensor Y\cong X\tensor \ab(Y)\cong \ab(X)\tensor \ab(Y)\text{.}
  \]
  As a consequence, the triangle of functors on the left
  \[
    \xymatrix@!0@=4em{\X\times \X \ar[rr]^{\tensor_{\X}} \ar[rd]_{\ab\times\ab} && \Ab(\X)\\
    & \Ab(\X)\times \Ab(\X) \ar[ru]_-{\tensor_{\X}}}
    \qquad\qquad
    \xymatrix@!0@=4em{\X \ar[rr]^{X\tensor_{\X}(-)} \ar[rd]_{\ab} && \Ab(\X)\\
    & \Ab(\X) \ar[ru]_-{X\tensor_{\X}(-)}}
  \]
  commutes up to isomorphism, and for each \(X\) of \(\X\), so does the one on the right.
\end{corollary}
\begin{proof}
  By Proposition~\ref{prop:internal descr of tensor}, since \(X\cosmash (Y\cosmash Y)\leq X\cosmash Y\cosmash Y\) is divided out in \(X\tensor Y\), so that the map \(X \tensor [Y,Y] \to X \tensor Y\) is trivial, this follows from Proposition~\ref{Proposition-RightExact-Tensor} and the fact that \(\ab(Y)=Y/[Y,Y]\).
\end{proof}

This sheds some new light on the discussion in~\ref{Bilinear bifunctors}: since the left adjoint \(\ab\) sends both the binary sum + in \(\X\) and the binary sum \(+_{2}\) in \(\Nil_{2}(\X)\) to the binary sum (=~binary biproduct) \(\oplus\) in \(\Ab(\X)\), which is the binary product \(\times\) in both \(\X\) and \(\Nil_{2}(\X)\) and that, by Corollary~\ref{Corollary-Products-Tensor}, \(X \tensor (-) \from \Ab(\X) \to \Ab(\X)\) preserves binary products, the functor \(X\tensor(-)\) sends binary sums to binary products, independently of which category (\(\X\), \(\Nil_{2}(\X)\) or \(\Ab(\X)\)) we took for domain or codomain. Since the \emph{second cross-effect} of Section~\ref{Section Cross-Effects} does nothing but measuring how far a functor is from sending binary sums to binary products, we are interested in asking that it vanishes---which in the case of \(X\cosmash (-)\) leads to ternary commutators, two-nilpotent objects and the bilinear product. On the other hand, Corollary~\ref{Corollary-Products-Tensor} may now be rephrased as follows.

\begin{corollary}\label{Corollary Additive}
  The bifunctor \(\tensor_\X\colon \Ab(\X)\times \Ab(\X) \to \Ab(\X)\) is biadditive\footnote{In the classical sense, valid for functors between additive categories; for instance, addition of morphisms is preserved in each variable.}.\noproof
\end{corollary}

\section{Examples of bilinear products}
\label{Section Examples}

\subsection{Groups and the tensor product of \(\Z\)-modules.}\label{Groups}
In the category of groups,
\[
  X\cosmash Y=\langle [x,y]\mid \text{\(x\in X\), \(y\in Y\)}\rangle
\]
where \([x,y]\coloneq xyx^{-1}y^{-1}\), and \(X\cosmash Y\cosmash Z\) contains all
\[
  [x,[y,z]]=xyzy^{-1}z^{-1}x^{-1}zyz^{-1}y^{-1}
\]
for \(x\in X\), \(y\in Y\) and \(z\in Z\). A~group is two-nilpotent when all commutator words of the latter type vanish in it. It is shown in~\cite{MacHenry} and~\cite[Proposition~2.3]{Hartl-Vespa} that for any two groups \(X\) and~\(Y\), the bilinear product \(X\tensor Y\) coincides with \(\ab(X)\tensor_{\Z} \ab(Y)\), the tensor product as \(\Z\)-modules of their abelianisations. Here a generator \(x[X,X]\tensor y[Y,Y]\) of \(\ab(X)\tensor_{\Z} \ab(Y)\) corresponds to the equivalence class of \([x,y]\) in the quotient \(\nil_2(X) +_2 \nil_2(Y)=\nil_2(X+Y)\) of \(X+Y\). As explained in~\cite{dMVdL19.3}, this result also follows from its being a special case of the Brown--Loday non-abelian tensor product~\cite{Brown-Loday}.

Note that the bilinear product in \(\Gp\) is associative. However, since non-abelian groups exist, the bilinear product does in general not have a unit. Of course, when~\(A\) is abelian, we \emph{do} have that \(\Z\tensor A\cong A\). For us this is an argument in favour of the idea that the bilinear product should be considered as an operation on the abelian core. It is also worth noticing that the functor \(\Z\tensor_\Gp (-)\colon{\Gp\to \Ab}\) is precisely the abelianisation functor, which follows from the fact that in \(\Ab(\Gp)=\Ab\), the bilinear product \(\tensor_\Gp\) has \(\Z\) as a unit. As mentioned before, this remains true in other cases where a unit exists for the bilinear product viewed as a product on \(\Ab(\X)\), like in some of the examples that follow below. On the other hand, the example of internal crossed modules shows that tensoring with the free object on a single generator need not agree with abelianisation---see Subsection~\ref{crossed modules}.

For \(A\) and \(B\) abelian groups, we now examine the isomorphism
\[
  \tau^{\Gp}_{A,B}\colon A\tensor B\to B\tensor A\text{,}
\]
induced by the symmetry of sum and product, under the identification of these bilinear products with the corresponding tensor products of abelian groups, explicitly given by the isomorphism
\[
  \xymatrix{
  \theta_{A,B}\colon A\tensor_{\Z} B \ar[r]^-{\cong} & A\tensor B \subobj A+_2B
  }
\]
which sends \(a\tensor b\) to the commutator \([i_1(a), i_2(b)] = i_1(a)i_2(b)i_1(a)^{-1}i_2(b)^{-1}\), as mentioned above. What we want to compute is the dotted map \(\tilde{\tau}_{A,B}^{\Gp}\) rendering the diagram
\[
  \xymatrix{
  A\tensor_{\Z} B \ar[r]^{\theta_{A,B}} \ar@{.>}[d]_{\tilde{\tau}_{A,B}^{\Gp}}& A\tensor B \ar@{-->}[d]_-{\tau_{A,B}^{\Gp}} \ar@{{ |>}->}[r]^-{\iota_{A,B}} & A+_2B \ar[d]^-{\tw}\\
  B\tensor_{\Z} A \ar[r]_{\theta_{B,A}} &   B\tensor A \ar@{{ |>}->}[r]_-{\iota_{B,A}} & B+_2A}
\]
commutative where \(\iota_{A,B}\) and \(\iota_{B,A}\) are the canonical kernel inclusions and
\[
  \tw\coloneq \CoindArr{i_2,i_1}\colon A+_2B\to B+_2A
\]
is the  twist map. Now
\begin{align*}
  \tw\comp\, \theta_{A,B}(a\tensor b) & = \tw(i_1(a)i_2(b)i_1(a)^{-1}i_2(b)^{-1}) =i_2(a)i_1(b)i_2(a)^{-1}i_1(b)^{-1}          \\
                                      & = \bigl(i_1(b)i_2(a)i_1(b)^{-1}i_2(a)^{-1}\bigr)^{-1} = \theta_{B,A}(b \tensor a)^{-1} \\
                                      & = \theta_{B,A}(- b \tensor a)\text{.}
\end{align*}
Hence the symmetry isomorphism \(\tilde{\tau}^{\Gp}_{A,B}\colon A\tensor_{\Z} B \to B\tensor_{\Z} A\) is given by
\begin{align*}
  \tilde{\tau}_{A,B}(a\tensor b) & = (\theta_{B,A})^{-1}\comp \tw\comp \,\theta_{A,B}(a\tensor b)
    = {}- b \tensor a\text{,}
\end{align*}
that is, \(\tilde{\tau}^{\Gp} = -\tau\) where \(\tau\) is the canonical twist of the tensor product.

\subsection{Loops}\label{Loops}
In the category of loops essentially the same result holds: for loops \(X\) and \(Y\) we have \(X\tensor Y\cong \ab(X)\tensor_{\Z} \ab(Y)\). This is a consequence of Theorem~\ref{Nil vs Birkhoff}, which in the situation at hand amounts to the fact that the reflector \(\gp\colon \Loop \to \Gp\) commutes with the nilpotentisation functors \(\nil^\Gp_{2}\colon \Gp\to \Nil_{2}(\Gp)\) and \({\nil^\Loop_{2}\colon \Loop\to \Nil_{2}(\Loop)}\) in the sense that \(\nil^\Gp_2\comp \gp=\nil^\Loop_2\). This makes sense because \(\nil^\Loop_2\) ``makes associative'' by quotienting out the associator object \(\ldbrack X,X,X\rdbrack\), which---see, for instance, the explanation in~\cite{EverVdL4}---is contained in the ternary commutator \([X,X,X]_\Loop\) (in \(\Loop\)) for any loop \(X\). Thus we see that \([X,X,X]_\Loop\cong [X,X,X]_\Gp\) for any group \(X\). As a consequence, \(\Nil_{2}(\Loop)\simeq \Nil_{2}(\Gp)\) and the resulting bilinear products coincide. Moreover, the analysis of the symmetry isomorphism we carried out for groups above remains valid for loops as well.

\subsection{Commutative associative algebras: the tensor product of \(R\)-modules.}\label{Commutative algebras}
Given a unitary commutative ring \(R\), the category \(\CUAlg_{R}\) of unitary commutative associative algebras over \(R\) is the coslice category \((R\downarrow \CURng)\) where \(\CURng\) is the category of unitary commutative rings. Instead of the non-pointed category \(\CUAlg_{R}\) we consider its semi-abelian slice \(\CAlg_{R}=(\CUAlg_{R}\downarrow 0)\)---where \(0\) denotes the initial object of \(\CUAlg_R\), i.e.\ the ring \(R\)---of non-unitary commutative associative algebras over \(R\). (As usual, we remove the unit by adding an augmentation.) As explained in~\cite{Smash}, given objects \(X\) and \(Y\) of \(\CAlg_{R}\), the cosmash product \(X\cosmash Y\) is the ordinary tensor product \(X\tensor_{R}Y\) of \(X\) and \(Y\) over \(R\), because \(X+Y={X\oplus Y\oplus (X\tensor_{R}Y)}\) with the obvious multiplication. So the classical tensor product of commutative associative algebras is captured as a cosmash product, which implies that the cosmash product is associative. Associativity is actually rare for a cosmash product: see~\cite{Illinois} where this is explained.

Then what is the bilinear product in \(\CAlg_{R}\)? We have to divide out the images of the ternary cosmash products \(X\cosmash X\cosmash Y=X\tensor_{R} X\tensor_{R} Y\) and \(X\cosmash Y\cosmash Y=X\tensor_{R} Y\tensor_{R} Y\), so that \(X\tensor Y\cong \ab(X)\tensor_{R} \ab(Y)\). (In this category, abelianisation kills the multiplication.)

Thus, when restricting \(\tensor\) to \(\Mod_{R}\simeq \Ab(\CAlg_{R})\), we regain the tensor product of \(R\)-modules. For instance, the bilinear product of non-unitary commutative rings \(X\) and \(Y\) is \(\ab(X)\tensor_{\Z} \ab(Y)\), so that the \(\Z\)-tensor product on \(\Ab\) is regained as ``intrinsic tensor product'' of the category \(\CRng = \CAlg_\Z\).

Here, the symmetry isomorphism \(A\tensor B\to B\tensor A\) is just the usual twist map \(\tau\colon A\tensor_R B\to B\tensor_R A\). So, even though when \(R=\Z\) we find the same tensor product \(A\tensor_\Z B\) on \(\Ab\cong\Mod_\Z\) as in Subsection~\ref{Groups}, the induced symmetry isomorphisms are different, which proves that the shape of the symmetry isomorphism depends on the surrounding semi-abelian category in which the bilinear product is computed.

\subsection{Non-commutative algebras.}\label{Associative algebras}
The case of non-commutative algebras is different, and provides an example of a situation where we do \emph{not} find the ordinary tensor product of \(R\)-modules out of a construction in a bigger category of which \(\Mod_{R}\) forms the abelian core. This illustrates how the bilinear product depends on the surrounding category.

In the category \(\Alg_{R}\) of non-unitary non-commutative associative algebras over a commutative unitary ring \(R\), sums are computed as follows~\cite{Berstein}. Given algebras \(A_{1}\) and \(A_{2}\) we write \((A_{1},A_{2})_{n}\) for the alternating tensor product over \(R\) of length \(n\): so \((A_{1},A_{2})_{0}=R\), \((A_{1},A_{2})_{1}=A_{1}\), \((A_{1},A_{2})_{2}=A_{1}\tensor_{R}A_{2}\), \((A_{1},A_{2})_{3}={A_{1}\tensor_{R}A_{2}\tensor_{R} A_{1}}\), etc. Then the coproduct of \(A_{1}\) and \(A_{2}\) is
\[
  A_{1}+A_{2}=\bigoplus_{n\geq 1}\bigl((A_{1},A_{2})_{n}\oplus (A_{2},A_{1})_{n}\bigr)
\]
with multiplication \((A_{i_{1}},A_{i_{2}})_{k}\tensor (A_{j_{1}},A_{j_{2}})_{l}\to (A_{i_{1}},A_{i_{2}})_{m}\) induced either by the multiplication on \(A_{j_{1}}\) (if the last term in the tensor product \((A_{i_{1}},A_{i_{2}})_{k}\) is \(A_{j_{1}}\); then we take \(m=k+l-1\)) or by identities (if not; then we take \(m=k+l\)).

It follows that for three given algebras \(A_{1}\), \(A_{2}\) and \(A_{3}\), their cosmash product \(A_{1}\cosmash A_{2}\cosmash A_{3}\) is a direct sum
\[
  \bigoplus_{n\geq 3}(A_{1},A_{2},A_{3})_{n}
\]
where \((A_{1},A_{2},A_{3})_{n}\) sums all possible tensor products of length \(n\) in \(A_{1}\), \(A_{2}\) and \(A_{3}\) such that any two adjacent factors are different. So, for instance, the products \({A_{1}\tensor_{R}A_{2}\tensor_{R}A_{3}\tensor_{R} A_{2}}\) and \(A_{3}\tensor_{R}A_{2}\tensor_{R}A_{1}\tensor_{R} A_{2}\) are terms in \((A_{1},A_{2},A_{3})_{4}\), but \(A_{1}\tensor_{R}A_{2}\tensor_{R}A_{3}\tensor_{R} A_{3}\) and \(A_{1}\tensor_{R}A_{2}\tensor_{R}A_{1}\tensor_{R} A_{2}\) are not.

When computing the bilinear product \(X\tensor Y\), we now have to divide the images of \(X\cosmash X\cosmash Y\) and \(X\cosmash Y\cosmash Y\) out of \(X\cosmash Y=\bigoplus_{n\geq 2}\bigl((X,Y)_{n}\oplus (Y,X)_{n}\bigr)\). It is clear that this quotient simply makes the terms of length three and higher vanish, so
\begin{align*}\label{algebras}
  X\tensor Y & \cong \bigl(\ab(X)\tensor_{R}\ab(Y)\bigr)\oplus \bigl(\ab(Y)\tensor_{R}\ab(X)\bigr) \\
             & \cong 2\cdot\bigl(\ab(X)\tensor_{R}\ab(Y)\bigr)\text{.}
\end{align*}

In this example, the symmetry isomorphism \(A\tensor B\to B\tensor A\) sends \((a \tensor b) \oplus (b' \tensor a')\) to \((b' \tensor a') \oplus (a \tensor b)\).

When \(R\) is a field, we find an example showing that a bilinear product need not have a unit. Indeed, if \(I\) is a unit, then
\[
  R=I\tensor R=(I\tensor_{R}R)\oplus (R\tensor_{R}I)=I\oplus I\text{,}
\]
which gives a contradiction with the fact that \(R\) is one-dimensional as a vector space over itself. On the other hand, via a direct calculation it is easy to see that associativity does hold for this bilinear product. An explicit example where associativity fails will be constructed at the end of Section~\ref{Section Operads}, in Example~\ref{non-associative bilinear product}.

\subsection{Lie algebras and the tensor product of \(R\)-modules.}\label{Lie algebras}
Lie algebras over~\(R\) behave quite similarly to commutative associative algebras---here also we regain the tensor product of \(R\)-modules when restricting the tensor to abelian objects. This may be explained by the fact that Lie brackets are antisymmetric. The two-nilpotent core of \(\Lie_R\) consists of algebras \((X, [-,-])\) where \([x,[y,z]]=0\) and \([x,y]=-[y,x]\) for all \(x\), \(y\), \(z\in X\). Hence the coproduct \(X+_2Y\) in \(\Nil_2(\Lie_R)\) of two abelian Lie algebras \(A\) and \(B\) is \(A\oplus B\oplus (A\tensor_RB)\), so that \(A\cosmash B\) is \(A\tensor_{R}B\), and \(X\tensor Y\cong \ab(X)\tensor_{R} \ab(Y)\) for all \(X\) and \(Y\) in \(\Lie_R\). Here the symmetry isomorphisms are as explained for groups in Subsection~\ref{Groups}. We adapt the argument given there to the case of Lie algebras, recycling the notations used there without further mention. In the case of \(R\)-Lie algebras, the isomorphism
\[
  \xymatrix{
  \theta_{A,B}\colon A\tensor_{R} B \ar[r]^-{\cong} & A\tensor B \subobj A+_2B
  }
\]
sends \(a\tensor b\) to the Lie bracket \([i_1(a),i_2(b)]\) calculated in \(A+_2B\). The morphism \(\tw=\CoindArr{i_2 \and i_1}\colon A+_2B\to B+_2A\) sends this to \([i_2(a),i_1(b)]\in B+_2A\), which is equal to \(-[i_1(b),i_2(a)]=-\theta_{B,A}(b\tensor a)\). It follows that \(\tilde{\tau}^{\Lie_R} = -\tau\) where \(\tau\) is the canonical twist of the tensor product.

\subsection{Leibniz algebras.}
Similarly to the difference between commutative and non-commutative algebras, we may now compare the Lie algebra with the Leibniz algebra tensor product. Leibniz algebras being ``non-antisymmetric Lie algebras'', we may expect a similar result. Recall from~\cite{Loday-Leibniz} that for a field \(R\), an object \(X\) in \(\Leib_{R}\) is an \(R\)-vector space with a bilinear bracket \([-,-]\colon X\times X\to X\) that satisfies the identity
\[
  [x,[y,z]]=[[x,y],z]-[[x,z],y]
\]
for \(x\), \(y\), \(z\in X\). And indeed, using that the coproduct \(A+_2B\) in \(\Nil_2(\Leib_R)\) of two abelian Leibniz algebras \(A\) and \(B\) is \(A\oplus B\oplus (A\tensor_RB)\oplus (B\tensor_RA)\), in accordance with~\cite{CP}, we obtain the isomorphism
\[
  X\tensor Y \cong \bigl(\ab(X)\tensor_{R}\ab(Y)\bigr)\oplus \bigl(\ab(Y)\tensor_{R}\ab(X)\bigr)\text{.}
\]
The symmetry isomorphisms are as in Subsection~\ref{Associative algebras}.

\subsection{Heyting semilattices.}
The semi-abelian category \(\HSLat\) of Heyting semilattices is arithmetical~\cite{Pedicchio, Borceux-Bourn, Jo}, which implies that the only abelian object in \(\HSLat\) is the zero semilattice. As a consequence, all bilinear products in \(\HSLat\) are trivial. For the same reason, bilinear products vanish in the dual of the category of pointed objects in any topos~\cite{Bourn:Dual-topos}, in the categories of \emph{boolean} and \emph{von Neumann regular} rings~\cite{Borceux-Bourn} and in the category of \(C^{*}\)-algebras~\cite{Gran-Rosicky:Monadic}.

\subsection{Abelian categories.}\label{Trivial Tensor Abelian}
On the other end of the spectrum we have the context of abelian categories. Here also all bilinear products vanish, because all cosmash products are trivial. To obtain the bilinear product of modules over a commutative ring \(R\), those modules may be considered as commutative algebras over~\(R\) (with a trivial multiplication \(x\cdot y=0\)) as in~\ref{Commutative algebras}.

\subsection{Sheaves of abelian groups.}
The usual tensor product of sheaves of abelian groups (for a certain topology \(\T\)) coincides with the bilinear product in the semi-abelian category \(\Sh_{\T}(\Nil_{2}(\Gp))\) of sheaves of two-nilpotent groups. Indeed, the tensor product of two sheaves is the sheafification of their tensor product as pre\-sheaves, which is taken pointwise. Following~\ref{Groups}, these pointwise tensor products may be computed in \(\Nil_{2}(\Gp)\). Now it suffices to note that sheafification is an exact functor. Note that here the symmetry isomorphisms are pointwise as in Subsection~\ref{Groups}.

\subsection{Beck modules}
In Section~\ref{Section Tensor of G-actions}, we study the bilinear product in categories of internal actions, which leads to a tensor product of Beck modules~\cite{Beck,Barr-Beck}. We recover the known tensor products of representations of groups and Lie algebras as special cases.

\subsection{Internal crossed modules.}\label{crossed modules}
We show that the tensor product of crossed modules of groups introduced in~\cite{Pira:Ganea} is intrinsic, and prove a general formula for bilinear products of internal crossed modules~\cite{Janelidze,HVdL} in a given semi-abelian category.

In a semi-abelian category \(\X\), let \((G,A,\del\colon{A\to G})\) and \((H,B,\delta\colon{B\to H})\) be two abelian crossed modules. As explained in~\cite{Bourn-Gran}, this means that their domain and codomain are abelian objects in \(\X\), and their action is trivial, so we omit it from the notation; in essence, they are just morphisms in \(\Ab(\X)\). We shall first of all prove that in the category \(\XMod(\X)\), the bilinear product \((G,A,\del)\tensor(H,B,\delta)\) coincides with \({(G\tensor H,\Coker(\alpha),\epsilon)}\), where the cokernel of
\[
  \alpha= \IndArr{\del\tensor 1_{B} \and -1_{A}\tensor \delta} \colon{A\tensor B\to (G\tensor B)\oplus (A\tensor H)}
\]
is taken in \(\Ab(\X)\), and the morphism \(\epsilon\colon \Coker(\alpha)\to G\tensor H\) is induced by
\[
  \CoindArr{1_{G}\tensor \delta \and \del\tensor 1_{H}}\colon (G\tensor B)\oplus (A\tensor H)\to G\tensor H\text{.}
\]
We compute the bilinear product of two abelian crossed modules in \(\X\) as the internal groupoid induced by the cosmash product of the respective associated reflexive graphs in~\(\Nil_{2}(\X)\). By Theorem~\ref{Nil vs Birkhoff}, the associated groupoid construction does indeed commute with the functor
\[
  \nil_{2}\colon{\RG(\X)\to \Nil_{2}(\RG(\X))= \RG(\Nil_{2}(\X))}\text{.}
\]
In order to obtain the internal groupoid universally induced by the cosmash product in the latter category, which is the reflexive graph (in fact, groupoid)
\begin{equation}\label{pcm}
  \xymatrix@C=4em{(G\oplus A)\tensor (H\oplus B) \ar@<1ex>[r]^-{\pi_{G}\tensor \pi_{H}} \ar@<-1ex>[r]_-{\DCoindArr{1_{G} \and \del}\tensor \DCoindArr{1_{H} \and \delta}} & G\tensor H \ar[l]}
\end{equation}
included in
\begin{equation*}
  \xymatrix@C=4em{(G\oplus A)+_2 (H\oplus B) \ar@<1ex>[r]^-{\pi_{G}+_2 \pi_{H}} \ar@<-1ex>[r]_-{\DCoindArr{1_{G} \and \del}+_2 \DCoindArr{1_{H} \and \delta}} & G+_2 H\text{,} \ar[l]}
\end{equation*}
by~\cite[Theorem~5.2]{HVdL} (combined with~\cite[Example~4.7]{HVdL}, to see that the ternary commutator mentioned in the theorem vanishes) we must divide out the commutator
\begin{align*}
   & \bigl[\Ker(\pi_{G}+_2 \pi_{H}),\Ker\bigl(\CoindArr{1_{G} \and \del}+_2 \CoindArr{1_{H} \and \delta}\bigr)\bigr] \\ & \normal (G\oplus A)\tensor (H\oplus B)                                    \\
   & = (G\tensor H) \oplus (G\tensor B) \oplus (A\tensor H)\oplus (A\tensor B)
\end{align*}
in \(\Nil_{2}(\X)\). Note that this commutator, which is originally obtained as a subobject of \({(G\oplus A)+_2 (H\oplus B)}\), is indeed contained in the abelian object \({(G\oplus A)\tensor (H\oplus B)}\), because the reflexive graph
\begin{equation*}
  \xymatrix@C=4em{(G\oplus A)\times_2 (H\oplus B) \ar@<1ex>[r]^-{\pi_{G}\times_2 \pi_{H}} \ar@<-1ex>[r]_-{\DCoindArr{1_{G} \and \del}\times_2 \DCoindArr{1_{H} \and \delta}} & G\times_2 H \ar[l]}
\end{equation*}
is a groupoid as well. Now \(\Ker(\pi_{G}+_2 \pi_{H})\) is the normal closure of
\[
  i_1\Ker(\pi_{G})\join i_2 \Ker(\pi_{H})
\]
in \({(G\oplus A)+_2 (H\oplus B)}\), so by~\cite[Proposition~4.14]{Actions} the join
\[
  i_1\Ker(\pi_{G})\join i_2 \Ker(\pi_{H})\join \bigl[i_1\Ker(\pi_{G})\join i_2 \Ker(\pi_{H}),(G\oplus A)+_2 (H\oplus B)\bigr]\text{,}
\]
while \(\Ker\bigl(\CoindArr{1_{G} \and \del}+_2 \CoindArr{1_{H} \and \delta}\bigr)\) is the normal closure of
\[
  i_1\Ker(\CoindArr{1_{G} \and \del})\join i_2 \Ker(\CoindArr{1_{H} \and \delta})
\]
in \({(G\oplus A)+_2 (H\oplus B)}\), so the join
\begin{multline*}
  i_1\Ker(\CoindArr{1_{G} \and \del})\join i_2 \Ker(\CoindArr{1_{H} \and \delta})\\
  \join \bigl[i_1\Ker(\CoindArr{1_{G} \and \del}) \join i_2 \Ker(\CoindArr{1_{H} \and \delta}),(G\oplus A)+_2 (H\oplus B)\bigr]\text{.}
\end{multline*}
Using Lemma~\ref{Lemma-Join}, we thus find that the commutator
\[
  \bigl[\Ker(\pi_{G}+_2 \pi_{H}),\Ker\bigl(\CoindArr{1_{G} \and \del}+_2 \CoindArr{1_{H} \and \delta}\bigr)\bigr]
\]
decomposes as
\begin{multline*}
  [i_1\Ker(\pi_{G}),i_1\Ker\CoindArr{1_{G} \and \del}]
  \join
  [i_1\Ker\CoindArr{1_{G} \and \del},i_2\Ker(\pi_{H})]\\*
  \join
  [i_1\Ker(\pi_{G}),i_2\Ker\CoindArr{1_{H} \and \delta}]
  \join
  [i_2\Ker(\pi_{H}),i_2\Ker\CoindArr{1_{H} \and \delta}]
\end{multline*}
since all higher-order commutators vanish by two-nilpotency of \(\Nil_{2}(\X)\). Here, only the two mixed middle terms are non-trivial, because \(G\oplus A\) and \(H\oplus B\) are abelian objects. Their join may be obtained through the image of the morphism
\[
  \CoindArr[\big]{\DIndArr{-\del \and 1_{A}} \tensor \DIndArr{0 \and 1_{B}} \and \DIndArr{0 \and 1_{A}} \tensor \DIndArr{-\delta \and 1_{B}}}
  \colon
  (A\tensor B)\oplus (A\tensor B) \to (G\oplus A)\tensor (H\oplus B)\text{,}
\]
which factors over the kernel of \(\pi_{G}\tensor \pi_{H}\) to the morphism
\[
  \resizebox{\textwidth}{!}
  {\mbox{\(\left\links\begin{matrix}
        (-\del)\tensor 1_{B} & 0                      \\
        0                    & 1_{A}\tensor (-\delta) \\
        1_{A\tensor B}       & 1_{A\tensor B}
      \end{matrix}\right\rechts
      \colon
      (A\tensor B)\oplus (A\tensor B) \to (G\tensor B) \oplus (A\tensor H)\oplus (A\tensor B)\text{.}\)}}
\]
It is clear that its cokernel coincides with \(\Coker(\alpha)\), which finishes the proof for abelian crossed modules. The general case now follows easily once we take into account that for a crossed module \((G,A,\mu,\del)\) in \(\X\), its abelianisation is
\[
  (G/[G,G], A/[A,G],\overline{\del})
\]
with \(\overline{\del}\) the induced morphism: simply substitute this in the above construction. Here we use that \([A,G] \normal A\) as recalled in Subsection~\ref{Subsection Binary Higgins Commutator}.

Considered as an operation on abelian crossed modules, this bilinear product has a unit as soon as the underlying bilinear product \(\tensor\) in~\(\Ab(\X)\) has one: it is easily seen to be \((I,0,0)\), where \(I\) is the unit in \(\Ab(\X)\). Note that this unit does not necessarily coincide with the abelianisation of the free object of rank one (which exists in case \(\X\) is a variety).

In the case of groups, for instance, this free object is the normal closure of the co\-product inclusion \(i_1\colon\Z\to \Z+ \Z\), so the crossed module \((\Z+\Z, \Z\flat\Z,c^{\Z\flat\Z,\Z+\Z},\kappa_{\Z,\Z})\) where \(\kappa_{\Z,\Z}\colon \Z\flat \Z\to \Z+\Z\) is the kernel of \(\CoindArr{1_\Z \and 0}\colon\Z+\Z\to \Z\) and \(c^{\Z\flat\Z,\Z+\Z}\) the induced conjugation action---which has \((\Z\oplus\Z, \Z,i_{\Z})\) for its abelianisation. Then we have, for example,
\[
  (\Z\oplus\Z, \Z,i_{\Z})\tensor(\Z\oplus\Z, \Z,i_{\Z})=(\Z^{4},\Z^{3},i_{\Z^{3}})\text{.}
\]

For associativity, the same result holds: if \(\tensor\) is associative on~\(\Ab(\X)\) then so is the cosmash product of reflexive graphs in~\(\Nil_{2}(\X)\), because its construction is degree-wise; hence \(\Ab(\XMod(\X))\) has an associative bilinear product. In case~\(\X\) is a variety, by Proposition~\ref{Closed SMC}, the product \(\tensor\) determines a (symmetric) monoidal closed structure on \(\Ab(\XMod(\X))\) as soon as \((\Ab(\X),\tensor,I)\) is a (symmetric) monoidal category. In either case, the symmetry isomorphism is induced by the respective symmetry isomorphisms in the underlying category.

\subsection{Internal precrossed modules}\label{precrossed modules}
The example of internal precrossed modules in a semi-abelian category is already implicit in the above: now we do not need to reflect to internal groupoids, so the bilinear product of two abelian precrossed modules \((G,A,\del\colon{A\to G})\) and \((H,B,\delta\colon{B\to H})\) simply is the normalisation of the reflexive graph~\eqref{pcm}, which is the abelian (pre)crossed module
\[
  (G\tensor H,\; (G\tensor B) \oplus (A\tensor H)\oplus (A\tensor B),\; \CoindArr{1_{G}\tensor \delta \and \del\tensor 1_{H}\and \del\tensor \delta})\text{.}
\]
This agrees with the definition proposed in~\cite{AriasLadra} in the case of crossed modules of groups.

The difference between~\ref{precrossed modules} and~\ref{crossed modules} provides another example showing that the bilinear product of objects in an abelian category depends on the surrounding semi-abelian category in which it is computed\footnote{This should, of course, also be compared with Subsection~\ref{Trivial Tensor Abelian}, which shows that the bilinear product \emph{obviously} depends on the surrounding semi-abelian category \(\X\) of an abelian category~\(\A\), simply because we may always take \(\X\DefEq\A\), which gives us zero bilinear products.}---note that an abelian precrossed module is automatically an abelian crossed module, so that the two abelian cores coincide.

Also in this case, the symmetry isomorphism is induced from the base category.

\section{\texorpdfstring{The bilinear product of algebras\except{toc}{\linebreak} over a reduced symmetric operad}{The bilinear product}}\label{Section Operads}

\subsection{Recollections on (right) symmetric operads and their algebras}
We briefly recall the notion of symmetric operad in the category of modules over a given commutative ring (which is, for instance, nicely explained in the \emph{Wikipedia} article on operads~\cite{Wikipedia:Operads}); however, we need operads to act on the right on their algebras and thus have to adapt the definitions; for convenience of the reader we write them out, as follows.

Let \(R\) be a commutative ring. Recall that a \emph{(constant-free)\footnote{\label{CF Footnote}Sometimes such operads, where \(\P(0)=0\), are called \defn{reduced} operads.} right symmetric operad in the standard monoidal category of \(R\)-modules}, or \defn{right symmetric operad\footnote{This should be read as ``right (symmetric operad)'', not ``(right symmetric) operad''.} in \(R\)-modules} \(\P\) for short, is a sequence of \(R\)-modules \(\P(n)\) with \(n\geq 1\) together with \(R\)-linear left actions of the symmetric group \(\S_n\) on \(\P(n)\) denoted by~\(\star\), an element \(1\in \P(1)\) called the \defn{identity} of \(\P\), and \(R\)-linear maps
\[
  \gamma_{k_1,\ldots,k_n;n} = \gamma_{k_1,\ldots,k_n;n}^{\P} \colon \P(k_1) \tensor_{R} \cdots \tensor_{R} \P(k_n) \tensor_{R} \P(n) \to \P(k_1+\cdots+ k_n)
\]
for \(n,k_1,\ldots,k_n\geq 1\) called \defn{composition operations} of \(\P\) satisfying the \defn{coherence axioms} below, where \(\theta \in \P(n)\), \(\theta_i \in \P(k_i)\).
\begin{itemize}
  \item \defn{Identity axiom}: \(\gamma_{1,\cdots ,1; n}(1\tensor\cdots\tensor 1\tensor\theta) = \theta = \gamma_{n;1}(\theta\tensor 1)\).
  \item \defn{Associativity axiom}: Given in addition \(\theta_{i,j} \in \P(l_{i,j})\) for \(i=1,\ldots,n\) and \(j=1,\ldots, k_i\), then for \(s_i=l_{i,1} + \cdots+l_{i,k_i}\), \(\underline{l}_i=l_{i,1}, \ldots , l_{i,k_i}\) and \(\underline{l}=l_{1,1},\ldots, l_{1,k_1}, \ldots,l_{n,1},\ldots, l_{n,k_n}\) we have
        \begin{multline*}
          \gamma_{s_1,\ldots,s_n;n}\bigl(\gamma_{\underline{l}_1;k_1}(\theta_{1,1}\tensor \cdots\tensor \theta_{1,k_1}\tensor \theta_1)\tensor \cdots\tensor \gamma_{\underline{l}_n;k_n}(\theta_{n,1}\tensor \cdots\tensor \theta_{n,k_n}\tensor \theta_n) \tensor \theta\bigr)\\
          =\gamma_{\underline{l};k_1+\ldots+k_n}\bigl(\theta_{1,1}\tensor\cdots\tensor \theta_{1,k_1}\tensor \cdots\tensor\theta_{n,1}\tensor\cdots\tensor \theta_{n,k_n} \tensor \gamma_{k_1,\ldots,k_n;n}(\theta_1\tensor \cdots \tensor\theta_n \tensor \theta)\bigr)\text{.}
        \end{multline*}
  \item \defn{Equivariance axioms}: \begin{enumerate}
          \item Given a permutation \(t \in \S_n\),
                \[
                  \gamma_{k_1,\ldots,k_n;n}(\theta _{1}\tensor\cdots \tensor\theta _{n}\tensor(t\star\theta))=
                  t'\star\gamma_{k_{t(1)},\ldots,k_{t(n)};n}
                  (\theta_{t(1)}\tensor\cdots \tensor\theta_{t(n)}\tensor \theta)
                \]
                where \(t'\) on the right hand side refers to the element of \(\S_{k_{1}+\cdots +k_{n}}\) that acts on the set \( \{1,2,\dots ,k_{1}+\dots +k_{n}\}\) by breaking it into \(n\) blocks, the first of size \(k_{1}\), the second of size \(k_2\), through the \(n\)-th block of size \(k_n\), and then permutes these \(n\) blocks by \(t^{-1}\), keeping each block intact.
          \item Given \(n\) permutations \(s_i \in \S_{k_i}\),
                \[
                  \gamma_{k_1,\ldots,k_n;n} ((s_1\star \theta_{1})\tensor\cdots \tensor (s_n\star\theta_n)\tensor \theta) = (s_{1},\ldots ,s_{n})\star \gamma_{k_1,\ldots,k_n;n} (\theta _{1}\tensor\cdots \tensor\theta _{n}\tensor \theta)
                \]
                where \( (s_{1},\ldots ,s_{n})\) denotes the element of \(\S_{k_{1}+\dots +k_{n}}\) that permutes the first of these blocks by \(s_{1}\), the second by \(s_{2}\), etc., and keeps their overall order intact.
        \end{enumerate}
\end{itemize}

In order to make the structure of \(\P(1)\) and \(\P(2)\) more explicit we recall the notion of (binary) \emph{wreath product} of associative \(R\)-algebras. For any \(R\)-algebra \(A\), the \defn{wreath product} \((A \otimes_{R} A) \wr \S_2\) is given by
\[(A \otimes_{R} A) \wr \S_2 \coloneq (A \otimes_{R} A) \oplus (A \otimes_{R} A).t\]
with multiplication defined by
\begin{multline*}
  (a_1\otimes a_2+(b_1\otimes b_2).t)(a_1'\otimes a_2'+(b_1'\otimes b_2').t)\coloneq \\*
  (a_1a_1' \otimes a_2a_2' + b_1b_2' \otimes b_2b_1' ) + (a_1b_1' \otimes a_2b_2' + b_1a_2' \otimes b_2a_1' ).t
\end{multline*}
for \(a_i\), \(a_i'\), \(b_i\), \(b_i'\in A\) and where \(t\) denotes the generator of the symmetric group \( \S_2\) of order 2.

\begin{remark}\label{structure of P1 and P2}
  Note that \(\P(1)\) endowed with the multiplication
  \[
    \gamma_{1;1}\colon \P(1)\tensor_{R} \P(1) \to \P(1)
  \]
  is an \(R\)-algebra whose unit is the identity \(1\) of the operad. Moreover, \(\P(2)\) is a \((\P(1)\tensor_{R} \P(1))\wr \S_2\)-\(\P(1)\)-bimodule whose left \((\P(1)\tensor_{R} \P(1))\wr \S_2\)-module structure is given by \(\gamma_{1,1:2}\) as for the action of \(\P(1)\tensor_{R} \P(1)\), and its right \(\P(1)\)-module structure is given by \(\gamma_{2;1}\); they commute with each other by the associativity axiom.
\end{remark}

An \defn{algebra over \(\P\)} or simply \defn{\(\P\)-algebra} is an \(R\)-module \(A\) endowed with operations
\[
  \mu_n = \mu_n^A \colon A^{\tensor_{R} n} \tensor_{R} \P(n) \to A
\]
for all \(n\geq 1\) such that \(\mu_1(a\tensor 1) = a\),
\begin{multline*}
  \mu_{k_{1}+\dots +k_{n}}(a_1\tensor\cdots\tensor a_{k_{1}+\dots +k_{n}}\tensor
  \gamma_{k_1,\ldots,k_n;n}(\theta _{1}\tensor\cdots \tensor\theta _{n}\tensor \theta))\\
  = \mu_{{n}}\big(\mu_{k_1}(a_1\tensor\cdots \tensor a_{k_1} \tensor\theta_1)\tensor \cdots \tensor \mu_{k_n}(a_{k_{1}+\dots +k_{n-1}+1}\tensor\cdots \tensor a_{k_{1}+\dots +k_n} \tensor\theta_n) \tensor \theta\big)
\end{multline*}
and
\[
  \mu_n(a_1\tensor\cdots\tensor a_n \tensor (t\star \theta))
  = \mu_n(a_{t(1)}\tensor\cdots\tensor a_{t(n)} \tensor \theta)
\]
for \(a\), \(a_i\in A\) and \(\theta\in \P(n)\), \(\theta_i\in \P(k_i)\) and \(t\in \S_n\).\medskip

\begin{remark}\label{mu1 and mu2}
  Note that a \(\P\)-algebra \(A\) has a natural structure of \emph{right \(\P(1)\)-module} via the map \(\mu_1\), and that \(\mu_2\) induces a homomorphism of right \(\P(1)\)-modules
  \[\mu_2'\colon (A\tensor_{R} A) \tensor_{(\P(1)\tensor_{R} \P(1)) \wr \S_2} \P(2) \to A
  \]
  where the generator \(t\) of \(\S_2\) acts on \({A}\tensor_{R} {A}\) by the twist \(\tau\colon a\tensor b\mapsto b\tensor a\).
\end{remark}

Now let \(A\) and \(B\) be \(\P\)-algebras. A \defn{morphism of \(\P\)-algebras} from \(A\) to \(B\) is an \(R\)-linear map \(f\colon A\to B\) commuting with the operations \(\mu_n\), that is,
\[
  f(\mu_n^A(a_1\tensor\cdots\tensor a_n \tensor \theta)) = \mu_n^B(f(a_1)\tensor\cdots\tensor f(a_n) \tensor \theta)
\]
for all \(n\geq 1\), \(a_i\in A\) and \(\theta\in \P(n)\). The category of \(\P\)-algebras \(\PAlg\) clearly is a variety of universal algebras. It is, moreover, semi-abelian, because it is a \emph{variety of \(\Omega\)-groups} in the sense of~\cite{Higgins}, as a group structure is part of the underlying module structure of the objects; pointedness is guaranteed because the operads we consider are constant-free---recall Footnote~\ref{CF Footnote} on page \pageref{CF Footnote}.

\subsection{Nilpotency of operads and of their algebras}\label{Nilpotency for operads}

As expected, for algebras over operads in \(R\)-modules, the notions of nilpotency of class \(\leq m\) in the operadic sense and in our categorical sense are equivalent; we here show this only for \(m\in \{1,2\}\) as this suffices for our purposes, and as the proof of the general case requires additional techniques of functor calculus for functors between semi-abelian categories, in particular the notion and properties of polynomial functors~\cite{CCC}---which exceeds the scope of the present paper.

We start by recalling the usual operadic definitions of nilpotency:

\begin{definition}\label{operadic nilpotency}
  Let \(m\geq 1\). An operad \(\P\) as above is said to be \defn{\(m\)-step nilpotent} if \(\P(n)=0\) for all \(n>m\). A \(\P\)-algebra \(A\) is said to be \defn{\(m\)-step nilpotent} if \(\mu_n^A=0\) for all \(n>m\). We denote by \(\NIL_m(\PAlg)\) the full subcategory of \(\PAlg\) whose objects are the \(m\)-step nilpotent \(\P\)-algebras.
\end{definition}

Note that if \(\P\) is \(m\)-step nilpotent then so are all its algebras.

\begin{remark}\label{1 and 2-step nilpotent P-algebras}
  Note that by Remark~\ref{structure of P1 and P2}, a \(2\)-step nilpotent right symmetric operad in \(R\)-modules is nothing but an \(R\)-algebra with unit \(\P(1)\), together with a \(\left((\P(1)\tensor_{R}\P(1))\wr \S_2\right)\)-\(\P(1)\)-bimodule \(\P(2)\).

  Moreover, by Remark~\ref{mu1 and mu2}, a \(1\)-step nilpotent \(\P\)-algebra is nothing but a right \(\P(1)\)-module, so \(\NIL_1(\PAlg) = \Mod_{\P(1)}\). Still following~\cite{CCC}, a \(2\)-step nilpotent \(\P\)-algebra \(A\) can be described as follows: \(A\) is a right \(\P(1)\)-module and has a sub-\(\P(1)\)-module \(D=D(A)\) (called the submodule of \defn{decomposables} of \(A\)) together with a surjective homomorphism of right \(\P(1)\)-modules
  \[
    \xymatrix{
    \overline{\mu_2} = \overline{\mu_2^A}\colon (\overline{A}\tensor_{R} \overline{A})\tensor_{(\P(1)\tensor_{R}\P(1))\wr \S_2} \P(2) \ar@{-{ >>}}[r] &D(A)
    }\]
  where \(\overline{A}=A/D(A)\).
\end{remark}

Next we review \emph{nilpotentisation} of operads and their algebras.

\begin{proposition}
  Consider \(m\geq 1\), a right symmetric operad in \(R\)-modules \(\P\), and a \(\P\)-algebra \(A\).
  \begin{enumerate}
    \item A right symmetric operad in \(R\)-modules \(\nil_m(\P)\) is derived from \(\P\) by \emph{truncation} at degree \(m\), that is, \(\nil_m(\P)(n) = \P(n)\) if \(n\leq m\) and \(\nil_m(\P)(n) = 0\) if \(n>m\), and where \(\gamma_{k_1,\ldots,k_n;n}^{\nil_m(\P)}\) equals \(\gamma_{k_1,\ldots,k_n;n}^{\P}\)  if \(k_1+\cdots+k_n\leq m\) and equals \(0\) otherwise.
    \item An \(m\)-step nilpotent \(\P\)-algebra \(\nIL_m(A)\) is defined by \(\nIL_m(A)= A/J_{m+1}\) with \(J_{n} =J_{n}(A) =\sum_{k\geq n} \Im(\mu_k^A)\) and \(\mu_n^{\NIL_m(A)}\) being induced by \(\mu_n^A\).
  \end{enumerate}
\end{proposition}
\begin{proof}
  For (1) we just need to check that the subsequence
  \[
    I=(0,\ldots,0,\P(m+1),\P(m+2),\ldots)
  \]
  of the sequence of \(R\)-modules \((\P(1),\P(2),\ldots)\) is an ideal of \(\P\); then \(\nil_m(\P)=\P/I\). Indeed, consider an operation \(\gamma=\gamma_{k_1,\ldots,k_n;n}\) of \(\P\). If \(k_i>m\) for some \(i\), then \(k_1+\cdots+k_n>m\) and hence the target module \(\P(k_1+\cdots+k_n)\) of \(\gamma\) equals \(I_{k_1+\cdots+k_n}\). If \(n>m\) then again \(k_1+\cdots+k_n>m\) since \(k_i\geq 1\) for all \(i\), noting that we assume \(\P\) to be constant-free.

  For (2), we verify that \(J_{m+1}\) is an ideal of \(A\). Let \(p\geq 1\), \(a_1\), \dots, \(a_p\in A\) and \(\theta\in\P(p)\). Suppose that for some \(i\) between \(1\) and \(p\), we have \(a_i=\mu_q(a_1'\tensor \cdots\tensor a_q'\tensor \theta')\) for \(q>m\), \(a_1'\), \dots, \(a_q'\in A\) and \(\theta'\in\P(q)\). If \(i=1\), then
  \begin{align*}
     & \mu_p(a_1\tensor\cdots\tensor a_p\tensor \theta)                                                                                                                            \\
     & = \mu_p( \mu_q( a_1'\tensor\cdots\tensor a_q'\tensor \theta')\tensor a_2\tensor\cdots\tensor a_p\tensor \theta)                                                             \\
     & = \mu_{p+q-1}(a_1'\tensor\cdots\tensor a_q'\tensor a_2\tensor \cdots\tensor a_p\tensor \gamma_{q,1,\ldots,1;p}(\theta'\tensor 1\tensor\cdots\tensor1\tensor\theta))\text{.}
  \end{align*}
  But \(p+q-1>m\), so \(\mu_p(a_1\tensor\cdots\tensor a_p\tensor \theta)\in J\). For \(i=2,\ldots,p\) the argument is analogous.
\end{proof}

Note that \(\nIL_m(A)\) is \(m\)-step nilpotent and \(\nIL_m\) canonically defines a functor \(\nIL_m\colon \PAlg \to \nIL_m(\PAlg)\) which is a Birkhoff reflector whose unit is the canonical quotient map \(\eta_m^A\colon A \to \nIL_m(A)\). Note that \(\NIL_m(\PAlg)\) is semi-abelian as a subvariety of a semi-abelian category; in fact, \(\NIL_m(\PAlg) = \Algg{\nil_m(\P)}\) (which also shows that \(\NIL_m(\PAlg)\) is a semi-abelian variety). \(\nil_m(\P)\) and \(\NIL_m(A)\) are called the \defn{\(m\)-nilpotentisation} of \(\P\) and of \(A\), respectively.

We are now ready to compare the operadic notion of nilpotency with the categorical one, recalled in Section~\ref{Section Nilpotentisation}; this amounts to comparing the respective notions of lower central series.

\begin{lemma}\label{J_n in gamma_n}
  For an operad \(\P\) as above and a \(\P\)-algebra \(A\)  we have \(J_n(A) \subobj \gamma_n(A)\) for all \(n\geq 1\).
\end{lemma}
\begin{proof}
  Let \(k\geq n \geq 1\), \(a_1\), \dots, \(a_k\in A\) and \(x\in \P(k)\). Then
  \begin{align*}
    \mu_k^A(a_1\tensor\cdots\tensor a_k\tensor x) & = \nabla^k_A \comp \mu_k^{A^{+ k}}(i_1(a_1)\tensor\cdots\tensor i_k(a_k)\tensor x) \\
                                                  & \in \gamma_k(A) \subobj \gamma_n(A)
  \end{align*}
  because \(\mu_k^{A^{+ k}}(i_1(a_1)\tensor\cdots\tensor i_k(a_k)\tensor x) \in A^{\diamond k} \)
  since
  \begin{align*}
     & \hat{r}_i\mu_k^{A^{+ k}}(i_1(a_1)\tensor\cdots\tensor i_k(a_k)\tensor x) \\
     & = \mu_k^{A^{+ (k-1)}}(i_1(a_1)\tensor\cdots\tensor
    i_{j-1}(a_{j-1})\tensor 0 \tensor i_{j+1}(a_{j+1})\tensor\cdots\tensor i_k(a_k)\tensor x) = 0
  \end{align*}
  for \(j=1,\ldots,k\).
\end{proof}

Note that this implies that \(\Nil_m(\PAlg)\) is a full subcategory of \(\NIL_m(\PAlg)\).

\begin{proposition}\label{abelian core of Alg-P}
  For any right symmetric operad in \(R\)-modules we have
  \[
    \Ab(\PAlg) = \Nil_1(\PAlg) =\NIL_1(\PAlg) = \Mod_{\P(1)}\text{.}
  \]
\end{proposition}
\begin{proof}
  We already know that
  \[
    \Ab(\PAlg) = \Nil_1(\PAlg) \subseteq\NIL_1(\PAlg) = \Mod_{\P(1)}
  \]
  so only need to show that every \(1\)-step nilpotent \(\P\)-algebra \(A\) is an abelian object of \(\PAlg\). Indeed, the homomorphism \(s\colon A\times A\to A\colon (a,b)\mapsto s(a,b)= a+b\) is \(\P(1)\)-linear and hence commutes with \(\mu_1\). It further commutes with all operations \(\mu_n\) for \(n\geq 2\), since these are trivial on both sides.
\end{proof}

\begin{proposition}\label{coproduct in Nil_2(Alg-P)}
  Let \(\P\) be a \(2\)-step nilpotent right symmetric operad in \(R\)-modules and let \(A\), \(B\) be \(\P\)-algebras. Using the description of \(\P\)-algebras in Remark~\ref{1 and 2-step nilpotent P-algebras}, their coproduct \(A+B\) in \(\PAlg\) is given as follows.
  \[
    A+B = A \oplus B\oplus (\overline{A}\tensor_{R} \overline{B})\tensor_{\P(1)\tensor_{R}\P(1)}\P(2)
  \]
  as right \(\P(1)\)-modules;
  \[
    D(A+B) = D(A) \oplus D(B)\oplus (\overline{A}\tensor_{R} \overline{B})\tensor_{\P(1)\tensor_{R}\P(1)}\P(2)
  \]
  whence \(\overline{A+B} = \overline{A} \oplus \overline{B}\). Then the map \(\overline{\mu_2^{A+B}}\) is induced by
  the \(\P(1)\)-linear map \(\tilde{\mu}_2\colon \overline{A+B}\tensor_{R} \overline{A+B}\tensor_{R} \P(2) \to D(A+B)\) given by the matrix
  \[
    \left\links\begin{smallmatrix}
      \overline{\mu_2^A} & 0                  & 0  & 0              \\
      0                  & \overline{\mu_2^B} & 0  & 0              \\
      0                  & 0                  & 1 & \tau\otimes t'
    \end{smallmatrix}\right\rechts
  \]
  with respect to the decompositions
  \[
    \resizebox{\textwidth}{!}{
      \begin{tikzcd}[squared]
        \left(\overline{A}\tensor_{R} \overline{A}\tensor_{R} \P(2)\right)
        \oplus
        \left(\overline{B}\tensor_{R} \overline{B}\tensor_{R} \P(2)\right)
        \oplus
        \left(\overline{A}\tensor_{R} \overline{B}\tensor_{R} \P(2)\right)
        \oplus
        \left(\overline{B}\tensor_{R} \overline{A}\tensor_{R} \P(2)\right)
        \arrow[d, "\tilde{\mu}_2"]
        \\
        D(A) \oplus D(B)\oplus \left((\overline{A}\tensor_{R} \overline{B})\tensor_{\P(1)\tensor_{R}\P(1)}\P(2)\right)
      \end{tikzcd}}
  \]
  Here \(\tau\) denotes the twist of the tensor product and \(t'\) the map given by the action of \(t\).
\end{proposition}
\begin{proof}
  It is straightforward to check that the map \(\tilde{\mu}_2\) factors through
  \[
    (\overline{A+B}\tensor_{R} \overline{A+B})\tensor_{(\P(1)\tensor_{R}\P(1))\wr \S_2} \P(2)\text{,}
  \]
  hence \(A+B\) is a \(\P\)-algebra. Moreover, the coproduct inclusions \(i_1\), \(i_2\) of \(A\), \(B\) into \(A+B\), respectively, are maps of \(\P\)-algebras. If \(C\) is another \(\P\)-algebra and \(f\colon A\to C\), \(g\colon B\to C\) are maps of \(\P\)-algebras then the map \(\CoindArr{ f \and g}\colon A+B \to C\) defined such that
  \[
    \CoindArr{ f \and g}(a,b,a'\tensor b'\tensor x) = f(a)+g(b)+ \overline{\mu_2^C}(f(a')\tensor g(b')\tensor x)
  \]
  for \(a\), \(a'\in A\), \(b\), \(b'\in B\) and \(x\in P(2)\) is a map of \(\P\)-algebras such that \(\CoindArr{ f \and g}i_1= f\) and \(\CoindArr{ f \and g}i_2= g\), and is unique with respect to this property.
\end{proof}

\begin{corollary}\label{cosmash product of Nil_2(Alg-P)}
  Let \(\P\) be a \(2\)-step nilpotent right symmetric operad in \(R\)-modules. Then there is a natural isomorphism
  \[
    A \diamond B \hspace{2pt}\cong\hspace{2pt}(\overline{A}\tensor_{R} \overline{B})\tensor_{\P(1)\tensor_{R}\P(1)}\P(2)
  \]
  of right \(\P(1)\)-modules---that is, of abelian \(\P\)-algebras, by Proposition~\ref{abelian core of Alg-P}---for all \(\P\)-algebras \(A\), \(B\).
  \noproof
\end{corollary}

\begin{theorem}\label{gamma_n(A)}
  For an operad \(\P\) as above and a \(\P\)-algebra \(A\) we have \(\gamma_n(A) = J_n(A)\) for all \(n\geq 1\).
\end{theorem}

We here prove only the cases \(n=1\), \(2\), \(3\) of interest in this paper. The general case is an immediate application of a characterisation in terms of functor calculus of the lower central series among all natural filtrations~\cite{CCC}.

\begin{proof}
  The case \(n=1\) is trivial, since \(\gamma_1(A) = A = J_1\). The case \(n=2\) follows from Proposition~\ref{abelian core of Alg-P}. For \(n=3\) we use Theorem~\ref{Nil vs Birkhoff} on the functor
  \[
    F=\nIL_2\colon \PAlg \to \NIL_2(\PAlg) = \Algg{\nil_2(\P)}
  \]
  to obtain a natural isomorphism
  \[
    \alpha\colon \nil_2^{\NIL_2(\PAlg)} \comp \nIL_2\To \nIL_2 \comp \nil_2^{\PAlg}
  \]
  such that
  \[
    \alpha_A \comp (\nilunit_2^{\NIL_2(\PAlg)})_{\nIL_2(A)} \comp \eta_2^A = \nIL_2((\nilunit_2^{\PAlg})_A) \comp \eta_2^A
  \]
  for any \(\P\)-algebra \(A\). Now by Corollary~\ref{cosmash product of Nil_2(Alg-P)}, the binary cosmash product of the category \(\NIL_2(\PAlg) = \Algg{\nil_2(\P)}\) is bilinear, so by Proposition~\ref{Cosmash Bilinear iff Category Two-Nilpotent} the category \(\NIL_2(\PAlg)\) is two-nilpotent. Hence \(\gamma_3(A/J_3(A)) = 0\). On the other hand, Lemma~\ref{J_n in gamma_n} implies that \(J_3(A/\gamma_3(A))\) is trivial, so we obtain the next commutative diagram.
  \[
    \xymatrix{
    & A/J_3(A)
    \ar@{=}[dd]  \ar@{-{ >>}}[rrr]^-{\nIL_2((\nilunit_2^{\PAlg})_A)} & & &(A/\gamma_3(A))/J_3(A/\gamma_3(A)) \ar[r]^-{\cong} & A/\gamma_3(A)
    \\
    A \ar[ur]^{\eta_2^A} \ar[dr]_{\eta_2^A} &&&&&\\
    &A/J_3(A)\ar@{-{ >>}}[rrr]_-{(\nilunit_2^{\NIL_2(\PAlg)})_{\nIL_2(A)}} & &&A/J_3(A)/\gamma_3(A/J_3(A)) \ar[uu]^{\alpha_A}_{\cong} \ar[r]_-{\cong} & A/J_3(A)
    }
  \]
  It follows that
  \[
    \gamma_3(A) = \Ker(\nIL_2((\nilunit_2^{\PAlg})_A) \comp \eta^2_A)
    \cong \Ker((\nilunit_2^{\NIL_2(\PAlg)})_{\nIL_2(A)} \comp \eta^2_A )
    =J_3(A),
  \]
  which completes the proof.
\end{proof}

Together with Corollary~\ref{cosmash product of Nil_2(Alg-P)}, this implies the following computation of the bilinear product of algebras over an algebraic operad, which unifies several of the results in Section~\ref{Section Examples}.

\begin{corollary}\label{bilinear product of Alg-P}
  For any right symmetric operad in \(R\)-modules \(\P\) we have
  \[
    \mathsf{Nil}_2(\PAlg)=\NIL_2(\PAlg) = \Algg{\nil_2(\P)}
    \quad\text{and}\quad
    \nil_2^{\PAlg} = \nIL_2^{\PAlg}\text{.}
  \]
  Furthermore,
  \[
    A \tensor B = \bigl(A/J_2(A)\tensor_{R} B/J_2(B)\bigr)\tensor_{\P(1)\tensor_{R}\P(1)}\P(2)
  \]
  is the bilinear product of \(\PAlg\), and
  the isomorphism
  \[\xymatrix{
      (A/J_2(A)\tensor_{R} B/J_2(B))\tensor_{\P(1)\tensor_{R}\P(1)}\P(2) \ar[d]_{\tilde{\tau}^{\Algg{\P}}_{A,B}} \\(B/J_2(B)\tensor_{R} A/J_2(A))\tensor_{\P(1)\tensor_{R}\P(1)}\P(2)
    }\]
  corresponding to the symmetry isomorphism
  \[
    {\tau}^{\Algg{\P}}_{A,B}\colon A\cosmash B \to B\cosmash A
  \]
  is given by \(\tilde{\tau}^{\Algg{\P}}_{A,B}(a\tensor b \tensor x)=b\tensor a \tensor (t\star x)\).\noproof
\end{corollary}

We thus recover the calculation of the bilinear product of commutative, associative and Lie algebras (the latter for \(R\) being a field of characteristic different from~\(2\)) in Section~\ref{Section Examples} by taking \(\P(1)=R\) and \(\P(2)\) to be the one-dimensional trivial, regular and signature representation of \(\S_2\) with coefficients in \(R\), respectively.

Corollary~\ref{bilinear product of Alg-P} also allows us to determine a large class of cases where the bilinear product of \(\PAlg\) is \emph{naturally associative}:

\begin{proposition}\label{associativity of bilinear product}
  Let \(\P\) be a right symmetric operad in \(R\)-modules with \(\P(1)=R\). Then for all \(\P\)-algebras \(A\), \(B\), \(C\) there is a natural isomorphism \((A\tensor B)\tensor C\cong A\tensor(B\tensor C)\).
\end{proposition}
\begin{proof}
  First note that the two commuting left \(\P(1)\)-module structures of \(\P(2)\) given by \(\gamma_{1,1;2}\) as well as its right \(\P(1)\)-module structure given by \(\gamma_{2;1}\) coincide with its given \(R\)-module structure. Thus Corollary~\ref{bilinear product of Alg-P} gives \(A\tensor B=\overline{A}\tensor_{R}\overline{B}\tensor_{R}\P(2)\). Now
  \(\overline{A\tensor B}=A\tensor B\) as \(A\tensor B\) is an abelian \(\P\)-algebra, so applying this formula a second time and using associativity of \(\tensor_{R}\) we get:
  \begin{align*}
    (A\tensor B)\tensor C
     & \;\cong\;
    \overline{A}\tensor_{R}\overline{B}\tensor_{R}\P(2)\tensor_{R}\overline{C}\tensor_{R}\P(2)\text{,} \\
    A\tensor(B\tensor C)
     & \;\cong\;
    \overline{A}\tensor_{R}\overline{B}\tensor_{R}\overline{C}\tensor_{R}\P(2)\tensor_{R}\P(2)\text{.}
  \end{align*}
  Thus the map \(1_{\overline{A}}\tensor 1_{\overline{B}}\tensor T\tensor 1_{P(2)}\), where the flip \(T\colon \P(2)\tensor_{R}\overline{C}\xrightarrow{\;\cong\;}\overline{C}\tensor_{R}\P(2)\) is given by the natural symmetry of \(\tensor_{R}\), is a natural \(R\)-linear isomorphism between them.
\end{proof}

Note that the assumption \(\P(1)=R\) holds for all the examples in Section~\ref{Section Examples}. For arbitrary semi-abelian varieties \(\V\) conditions ensuring that \(\tensor_{\V}\) is associative or even a monoidal structure (though possibly without unit) will be studied elsewhere. An example in the operad framework where \emph{associativity fails} is given as follows:

\begin{example}\label{non-associative bilinear product}
  Let \(k\) be a field. Define a \(2\)-step nilpotent right symmetric operad \(\P\) over \(k\) with
  \[
    \P(1)=\P(2)=k^2\text{,}
  \]
  where the ring structure of \(\P(1)\) is the direct product of the ring \(k\) with itself and  the left \(\bigl((\P(1)\tensor_{k}\P(1))\wr\S_2\bigr)\)-module and right \(\P(1)\)-module structures on \(\P(2)\) are given, for \(P=(a,b)\) and \(Q=(a',b')\) in \(\P(1)\), by
  \begin{align*}
    \gamma_{1,1;2}(P\tensor Q\tensor(n_1,n_2)) & =(ab'n_1,\;a'bn_2)\text{,} \\
    t\cdot(n_1,n_2)                            & =(n_2,n_1)\text{,}
    \\
    \gamma_{2;1}((n_1,n_2)\tensor P)           & =(an_1,\;an_2)\text{.}
  \end{align*}
  The wreath product equivariance axiom \(\gamma_{1,1;2}(P\tensor Q\tensor(t\cdot n))=t\cdot\gamma_{1,1;2}(Q\tensor P\tensor n)\) holds: both sides equal \((ab'n_2,\;a'bn_1)\). Moreover, the left action of the wreath product \({(\P(1)\tensor_{R} \P(1))\wr \S_2}\) on \(\P(2)\) commutes with the right action of \(\P(1)\).

  Let \(S_1\) and \(S_2\), \(S_i\cong k\), denote the two simple abelian \(\P\)-algebras (that is, right \(\P(1)\)-modules with \(\mu_2=0\)) with right actions \(\lambda\cdot P=a\lambda\) and \(\mu\cdot P=b\mu\) respectively. By Corollary~\ref{bilinear product of Alg-P}:
  \begin{itemize}
    \item \( S_2\tensor S_2 = (S_2\tensor_{k}S_2)\tensor_{\P(1)\tensor_{k}\P(1)}\P(2)=0\): taking \(P=Q=(0,1)\) in the defining relation of the tensor product gives
          \begin{align*}
            (\mu\tensor\nu)\tensor(n_1,n_2) & =(\mu\tensor\nu)(P\tensor Q)\tensor(n_1,n_2) \\
                                            & =(\mu\tensor\nu)\tensor(P\tensor Q)(n_1,n_2) \\
                                            & =(\mu\tensor\nu)\tensor(0,0)=0.
          \end{align*}
    \item \( S_1\tensor S_2=(S_1\tensor_{k}S_2)\tensor_{\P(1)\tensor_{k}\P(1)}\P(2)\cong S_1\neq 0\): let
          \[
            \alpha\colon (S_1\tensor_{k}S_2)\tensor_{\P(1)\tensor_{k}\P(1)}\P(2) \to S_1
          \]
          be the group homomorphism given by $\alpha(\lambda\tensor \mu \tensor (n_1,n_2)) =
            \lambda\mu n_1$. It is well defined as
          \begin{align*}
            \alpha(
            (\lambda\tensor\mu)(P\tensor Q)\tensor(n_1,n_2)) & =
            \alpha((a\lambda\tensor b'\mu)\tensor(n_1,n_2))                                                               \\
                                                             & = a\lambda b'\mu n_1 = \alpha(
            \lambda\tensor \mu \tensor (ab'n_1,a'bn_2))                                                                   \\
                                                             & =\alpha((\lambda\tensor \mu)\tensor(P\tensor Q)(n_1,n_2)).
          \end{align*}
          Moreover, \(\alpha\) is right \(\P(1)\)-linear, as
          \begin{align*}
            \alpha(\lambda\tensor \mu \tensor (n_1,n_2)P) & =
            \alpha(\lambda\tensor \mu \tensor (an_1,an_2))    \\
                                                          & =
            \lambda\mu an_1                                   \\
                                                          & =
            \alpha(\lambda\tensor \mu \tensor (n_1,n_2))P.
          \end{align*}
          Now let \(\beta\colon k\to (S_1\tensor_{k}S_2)\tensor_{\P(1)\tensor_{k}\P(1)}\P(2)\) be given by \(\beta(\nu)= 1\tensor 1 \tensor (\nu,0)\). Then \(\alpha \comp \beta = 1_k\) and
          \begin{align*}
            (\beta \comp \alpha) ((\lambda\tensor \mu) \tensor (n_1,n_2)) & =
            1\tensor 1 \tensor (\lambda \mu n_1,0)                                                                                          \\
                                                                          & = (1\tensor 1) \tensor (((\lambda,0)\tensor (0,\mu))(n_1,n_2))  \\
                                                                          & =  (1\tensor 1)  (((\lambda,0)\tensor (0,\mu))\tensor(n_1,n_2)) \\
                                                                          & = (\lambda\tensor \mu) \tensor (n_1,n_2).
          \end{align*}
          Thus \(\alpha\) is an isomorphism of right \(\P(1)\)-modules, as desired.
  \end{itemize}
  It follows that
  \[
    (S_1\tensor S_2)\tensor S_2 \cong S_1\tensor S_2 \cong S_1 \neq 0
  \]
  while
  \[
    S_1\tensor (S_2 \tensor S_2) \cong S_1\tensor 0 = 0.
  \]
  Hence the bilinear product of \(\P\)-algebras is non-associative.
\end{example}

\section{Varieties of algebras, in general}\label{Section Varieties}
In this section we prove some general results on bilinear products in semi-abelian varieties of algebras. We work towards the Recognition Theorem~\ref{Recognition Theorem} which characterises any symmetric bi-cocontinuous bifunctor on an abelian variety as the bilinear product of a variety of algebras over an operad induced by the given bifunctor.

\subsection{Properties of tensoring with a fixed object of a variety}\label{Tensoring with a fixed object of a variety}

We resume the discussion at the end of Section~\ref{Section Further Properties} in the case of a variety of algebras \(\V\). Then the functors \(X\tensor(-)\colon\V\to \Ab(\V)\) do not just preserve all finite sums, but are actually cocontinuous:

\begin{proposition}\label{Proposition Tensor Cocontinuous}
  If \(\V\) is a semi-abelian variety of algebras and \(X\in \V\) then the functor \(X\tensor(-)\colon{\V\to \Ab(\V)}\) preserves all colimits.
\end{proposition}
\begin{proof}
  In any variety of algebras, filtered colimits commute with finite limits (see, for instance, \cite{Borceux:Cats2}) so \(X\tensor(-)\) preserves filtered colimits---because so do the abelianisation functor and the functors \(\ab(X)+(-)\) and \(\ab(X)\times(-)\) in the category \(\Nil_2(\V)\). As a consequence, the functor \(X\tensor(-)\) preserves arbitrary sums: indeed, such a sum is a filtered colimit of finite sums, and we already know that both finite sums and filtered colimits are preserved. It also preserves all coequalisers, since \(X\tensor(-)\) factors through the abelianisation functor \(\ab\colon {\V\to \Ab(\V)}\), coequalisers in \(\Ab(\V)\) are calculated as the cokernel of a difference, and, by Corollary~\ref{Corollary Additive} and Proposition~\ref{Proposition-RightExact-Tensor}, the  functor \({X\tensor(-)}\colon {\Ab(\V)\to\Ab(\V)}\) is additive and right exact.
\end{proof}

Thus in view of the symmetry of the tensor product (Subsection~\ref{Subsection Symmetry}) it is bi-cocontinuous. It now follows that the functors \(X\tensor(-)\colon {\V\to \Ab(\V)}\) are always left adjoints. This is a consequence of the Special Adjoint Functor Theorem combined with~\cite[Theorem~1.58]{Adamek-Rosicky} which explains that any variety of algebras is well-co-powered. Note that as a generating set, we may use the free object on a single generator. Thus we proved:

\begin{proposition}\label{Prop pre SMC}
  For any object \(X\) of a semi-abelian variety \(\V\), the functor
  \[
    X\tensor(-)\colon {\V\to \Ab(\V)}
  \]
  is a left adjoint.\noproof
\end{proposition}

In case the bilinear product \(\tensor\) admits a unit \(I\) when restricted to \(\Ab(\V)\), the functor \(I\tensor(-)\colon \V\to \Ab(\V)\) is naturally isomorphic to abelianisation. Note that such a unit need not exist in general: see, for instance, the example of associative non-commutative algebras in Subsection~\ref{Associative algebras}. When a unit \(I\) does exist, we may ask under which conditions the triple \((\Ab(\V),\tensor,I)\) forms a monoidal category. The answer to this question is currently not known. However:

\begin{proposition}\label{Closed SMC}
  If \(\V\) is a semi-abelian variety such that \((\Ab(\V),\tensor,I)\) is a (symmetric) monoidal category, then \((\Ab(\V),\tensor,I)\) is a \emph{closed} (symmetric) monoidal category.\noproof
\end{proposition}

\subsection{Realisability of bilinear bifunctors on abelian varieties as cosmash products}\label{tensors from operads}

We are now ready to give a necessary and sufficient condition for a bilinear---and thus, biadditive: cf.\ Corollary~\ref{Corollary Additive}---bifunctor on an abelian variety~\(\A\) (in fact, a variety of modules over a unitary ring) to be realisable as the bilinear product in a semi-abelian variety \(\V\) containing \(\A\) (up to equivalence) as its abelian core \(\Ab(\V)\).

\begin{theorem}[Recognition Theorem]
  \label{Recognition Theorem}
  Let \(\A\) be an abelian variety of universal algebras and \(B\colon {\A\times \A \to \A}\) a bifunctor on~\(\A\). Then there exists a semi-abelian variety~\(\V\) whose abelian core is equivalent to~\(\A\) and whose cosmash product restricts to the functor \(B\) under this equivalence if and only if \(B\) is symmetric and bi-cocontinuous. Indeed, if these conditions are satisfied, then \(\V\) can be taken to be the category of algebras over a canonical \(2\)-nilpotent reduced (right) symmetric operad in abelian groups \(\P\).
\end{theorem}
\begin{proof}
  The condition is necessary by Proposition~\ref{Proposition Tensor Cocontinuous} and the explanation in~\ref{Subsection Symmetry}. To show that it also is sufficient, suppose that \(B\) is bi-cocontinuous (preserves all colimits in both slots) and symmetric (comes with a natural automorphism \(\tau_B\) such that for morphisms \(f_i\colon A_i\to A_i'\), \(i=1,2\) in \(\A\), the square
  \[
    \xymatrix{B(A_1,A_2) \ar[d]_-{(\tau_B)_{A_1,A_2}} \ar[r]^-{B(f_1,f_2)} & B(A'_1,A'_2) \ar[d]^-{(\tau_B)_{A_1',A_2'}}\\
    B(A_2,A_1) \ar[r]_-{B(f_2,f_1)} & B(A'_2,A'_1)}
  \]
  commutes). We first recall a classical setup and a number of related facts.
  Let \(E\) be a free object of rank \(1\) of \(\A\) and let \({\Lambda}\) be its endomorphism ring. Then the representable functor \(\A(E,-)\) takes values in the category \(\Mod_{\Lambda}\) of right \({\Lambda}\)-modules via precomposition with endomorphisms of \(E\), and actually is an equivalence by the Gabriel--Mitchell theorem~\cite{Freyd}, since \(E\) is a small projective generator of \(\A\) and \(\A\) has infinite sums. So without any loss of generality, we may suppose that \(\A= \Mod_{\Lambda}\).

  Denote by \({\Lambda}_{\Lambda}\) the ring \({\Lambda}\) viewed as a right \({\Lambda}\)-module. For a right \({\Lambda}\)-module \(M\) and \(m\in M\), denote by \(f_m\colon {\Lambda}_{\Lambda}\to M\) the \({\Lambda}\)-linear map that sends \(1\) to \(m\). Then the map \(\ev_M\colon \Hom_{\Lambda}({\Lambda}_{\Lambda},M)\to M\) sending \(f\) to \(f(1)\) is an isomorphism of abelian groups, whose inverse sends \(m\in M\) to \(f_m\). In particular, the map \(\ev_{{\Lambda}_{\Lambda}}\) is a ring isomorphism, and \(\ev_M\) is right \({\Lambda}\)-linear in the sense that for \(g\in \Hom_{\Lambda}({\Lambda}_{\Lambda},M)\) and \(r\in \Lambda\) we have \(\ev_M(g \comp \ev_{\Lambda_\Lambda}^{-1}(r)) = \ev_M(g)r\).

  Now the right \({\Lambda}\)-module \(B({\Lambda}_{\Lambda},{\Lambda}_{\Lambda})\) is an \(\End_{\Lambda}({\Lambda}_{\Lambda}) \tensor \End_{\Lambda}({\Lambda}_{\Lambda})\)-\({\Lambda}\)-bimodule and thus a \({\Lambda}\tensor  {\Lambda}\)-\({\Lambda}\)-bimodule via the isomorphism \(\ev_{\Lambda}\). For right \({\Lambda}\)-modules \(M\),~\(N\) let
  \[
    \Phi_{M,N} \colon (M\tensor_{\Z} N) \tensor_{{\Lambda}\tensor_{\Z}{\Lambda}} B({\Lambda}_{\Lambda},{\Lambda}_{\Lambda}) \to B(M,N)
  \]
  be given by \(\Phi_{M,N}(m\tensor n \tensor x)\coloneq B(f_m,f_n)(x)\). \(\Phi\) is easily checked to be a natural transformation of symmetric bifunctors on \(\Mod_{\Lambda}\) where the symmetry isomorphism on the left, say \(\tau_{\tensor}\), is given by \(\tau_{\tensor}(m\tensor n \tensor x)= n\tensor m\tensor \tau_B(x)\). Moreover, \(\Phi_{{\Lambda},{\Lambda}}\) is an isomorphism, whence so is \(\Phi\), since its source and target bifunctor are bi-cocontinuous and each module can be constructed from \(\Lambda\) using colimits.

  Now we are ready to define the needed \(2\)-step nilpotent right symmetric operad \(\P\) in \(\Z\)-modules. Thanks to Remark~\ref{1 and 2-step nilpotent P-algebras} it is enough to put \(\P(1)={\Lambda}\) as a ring and \(\P(2) = B({\Lambda}_{\Lambda},{\Lambda}_{\Lambda})\) equipped with the action of \(\mathfrak{S}_2\) given by \((\tau_B)_{({\Lambda}_{\Lambda},{\Lambda}_{\Lambda})}\) as a \(({\Lambda}\tensor {\Lambda})\wr \mathfrak{S}_2\)-\({\Lambda}\)-bimodule. Then \(\Ab(\Algg{\P}) = \Mod_{\Lambda}\) by Proposition \ref{abelian core of Alg-P}, and we have natural isomorphisms
  \[
    \xymatrix{
    M \diamond N \ar[r]^-{\cong} & (M\tensor_{\Z} N)\tensor_{\Lambda\tensor_{\Z}\Lambda}B({\Lambda}_{\Lambda},{\Lambda}_{\Lambda}) \ar[r]_-{\Phi_{M,N}}^-{\cong} & B(M,N)
    }
  \]
  for right \(\Lambda\)-modules \(M\), \(N\) by Corollary \ref{cosmash product of Nil_2(Alg-P)}, as desired.
\end{proof}

\begin{remark}
  Note that the semi-abelian variety \(\V\) constructed out of \((\A,B)\) depends on the choice of symmetry isomorphism \(\tau_B\) on \(B\)---as it should, because we already saw in Section~\ref{Section Examples} that a different surrounding semi-abelian variety \(\V\), whether constructed as above or not, may give rise to a different symmetry isomorphisms of the same \((\A,B)\). Indeed, comparing the examples of groups and commutative rings which both codify \((\Ab,\tensor_\Z)\), a different choice of category~\(\V\)---either \(\Gp\) or \(\CRng\)---such that \(\tensor_\V=\tensor_\Z\) on the category \(\Ab\) gives rise to two different symmetry isomorphisms on \(\tensor_\Z\).
\end{remark}

We point out that in forthcoming work we give a description of \emph{all} \(2\)-step nilpotent varieties, up to equivalence, containing the variety \(\A\) as its abelian core and whose cosmash product is isomorphic with the given bifunctor \(B\), provided it is symmetric and bi-cocontinuous. This becomes possible by first proving that just like abelian varieties are known to be categories of modules over unitary rings, \(2\)-step nilpotent varieties are equivalent with categories of modules over square rings~\cite{BHP}, as had been suggested to us by T.\ Pirashvili in private communication. We also provide an example where two non-equivalent two-nilpotent semi-abelian categories give rise to the same symmetry isomorphism.

\section{The Ganea term in semi-abelian homology}\label{Section Ganea}

By Theorem~5.9 in~\cite{EverVdL1}, any short exact sequence
\begin{equation}\label{Extension}
  \xymatrix{0 \ar[r] & K \ar@{{ |>}->}[r]^{k} & B \ar@{-{ >>}}[r]^{f} & A \ar[r] & 0}
\end{equation}
induces a five-term exact sequence
\[
  \xymatrix{\H_{2}B \ar[r] & \H_{2}A \ar[r] & \frac{K}{[K,B]} \ar[r] & \H_{1}B \ar@{-{ >>}}[r] & \H_{1}A \ar[r] & 0\text{.}}
\]
When \(f\) is central, the sequence simplifies, since then \([K,B]=0\). We prove that, using the tensor product defined above, this resulting sequence may be prolonged with an additional term \(K\tensor B\) on the left called the \defn{Ganea term} (Theorem~\ref{Theorem-Ganea}). This result was first obtained in the context of groups by Ganea~\cite{Ganea, Eckmann-Hilton-Stammbach} and later considered in several other situations~\cite{Lue:Ganea, MR897010, Casas:CELA, CP, Pira:Ganea, Casas:Ganea, AriasLadra}. (Note, however, that this five-term exact sequence is also the tail end of a long exact homology sequence~\cite{Tomasthesis, GVdL2}.) On the other hand, the result in~\cite{Lue:Ganea} will \emph{not} be a consequence of the following theorem, since the central extensions considered in that paper are more general.

Here we have to restrict ourselves to the context of an \defn{algebraically coherent} semi-abelian category~\cite{acc}. Algebraic coherence is a relatively mild condition satisfied by many semi-abelian categories, including the categories of groups, crossed modules, and Lie and Leibniz algebras---in fact, all \emph{Orzech categories of interest}~\cite{Orzech} are examples---but excluding for instance the category of loops as well as certain categories of algebras, cf.\ the introduction. Formally, it holds when the forgetful functor \(\Pt_G(\X)\to \X\) is coherent (i.e., it not only preserves finite limits, but also jointly extremal-epimorphic pairs of arrows; further details are given in Section~\ref{Section Tensor of G-actions}). Here we use the fact that in such a category, ternary commutators may be conveniently decomposed into a join of binary commutators: indeed, the \emph{Three Subobjects Lemma for normal subobjects} (Theorem 7.1 in~\cite{acc}) states that whenever \(K\), \(L\), \(M\) are normal subobjects of an object~\(X\) in an algebraically coherent semi-abelian category,
\begin{equation}\label{Equation Decompo}
  [K,L,M]=[[K,L],M]\join [[M,K],L]
\end{equation}
as subobjects of \(X\). A higher-order version of this result was obtained in the article~\cite{SVdL3}.

\begin{theorem}\label{Theorem-Ganea}
  Any central extension~\eqref{Extension} in a semi-abelian algebraically coherent category with enough projectives induces a six-term exact sequence
  \[
    {\xymatrix{K\tensor B \ar[r] & \H_{2}B \ar[r] & \H_{2}A \ar[r] & K \ar[r] & \H_{1}B \ar@{-{ >>}}[r] & \H_{1}A \ar[r] & 0\text{.}}}
  \]
\end{theorem}
\begin{proof}
  We adapt the proof of Theorem~3.2 in~\cite{Eckmann-Hilton-Stammbach} to the present context. Let \(p\colon {P\to B}\) be a projective presentation of \(B\), so a regular epimorphism with a projective domain. Let \(R\) denote the kernel of~\(p\) and write~\(L\) for the kernel of \(f\comp p\).
  \begin{eqdiagr}
    \label{Diag Presentations}
    \begin{tikzcd}[squared]
      & R \arrow[d, nmono] \arrow[ld, nmono, dotted]                       \\
      L \pb{rd} \arrow[r, "l", nmono] \arrow[d, "\bar{p}"', repi] & P \arrow[r, "f \circ p"] \arrow[d, "p", repi] & A \arrow[d, equal] \\
      K \arrow[r, nmono]                                          & B \arrow[r, "f"', repi]                       & A
    \end{tikzcd}
  \end{eqdiagr}
  Note that \(f\comp p\) is a presentation of \(A\) since regular epimorphisms compose, so that \(L\to P\to A\) is a short exact sequence. Furthermore, \(\bar{p}\) is a regular epimorphism since those are pullback-stable, so that \(R\to L\to K\) is a short exact sequence as well. The Hopf formula for \(\H_{2}\) (\cite{EverVdL2}; see~\cite[Theorem~7.2]{EGVdL} for an efficient proof, and~\cite{RVdL3} for a more explicit account of the notations used here) tells us that \(\H_{2}f\) is the morphism
  \[
    \frac{R\meet[P,P]}{[R,P]}\to \frac{L\meet[P,P]}{[L,P]}
  \]
  canonically induced by the inclusion \(R \mono L\). Since
  \[
    \frac{R\meet[P,P]}{[L,P]}\to \frac{L\meet[P,P]}{[L,P]}
  \]
  is a monomorphism (as \(R \subobj L\)), the kernel of \(\H_{2}f\) coincides with the kernel of the induced corestriction
  \[
    \frac{R\meet[P,P]}{[R,P]}\to \frac{R\meet[P,P]}{[L,P]}\text{.}
  \]
  By the \(3\times 3\)-Lemma or the Noether Isomorphism Theorem, this kernel is the quotient \([L,P]/[R,P]\) (see diagram below).
  \[
    \begin{tikzcd}[squared=5.5]
      {[R,P]} \arrow[r, nmono] \arrow[d, equal] & {[L,P]} \arrow[r, repi] \arrow[d, nmono]        & \frac{[L,P]}{[R,P]} \arrow[d, nmono]      \\
      {[R,P]} \arrow[r, nmono] \arrow[d, repi]  & R \meet {[P,P]} \arrow[r, repi] \arrow[d, repi] & \frac{R\meet[P,P]}{[R,P]} \arrow[d, repi] \\
      0 \arrow[r, nmono]                        & \frac{R\meet[P,P]}{[L,P]} \arrow[r, equal]      & \frac{R\meet[P,P]}{[L,P]}
    \end{tikzcd}
  \]
  Hence it suffices to construct a regular epimorphism
  \[
    K\tensor B \cong \bigl(\tfrac{L}{R}\tensor \tfrac{P}{R}\bigr)\to\frac{[L,P]}{[R,P]}\text{.}
  \]
  (Diagram~\eqref{Diag Presentations} explains where the isomorphisms come from since \(L \to K\) and \({P \to B}\) are regular epimorphisms hence cokernels of their kernels.) We only need to check that the canonical morphism \(e_{L,P}^{P}\colon {L\cosmash P\to [L,P]}\), which is the image of \(\CoindArr{l \and 1_P} \comp \iota_{L,P}\), is compatible with the respective quotients (namely, the two quotients by \(R\), the one defining the bilinear product, and the quotient by \([R,P]\)).

  First of all, by algebraic coherence, we may find the dotted arrow completing the next commutative square.
  \[
    \xymatrix{L\cosmash P \ar@{-{ >>}}[d]_-{\eta_{L,P}} \ar@{-{ >>}}[r]^-{e^P_{L,P}} & \ar@{-{ >>}}[d]^-q [L,P]\\
    L\tensor P \ar@{.{ >>}}[r]_-g & \frac{[L,P]}{[R,P]}}
  \]
  Indeed, by \eqref{Equation Decompo} we have that \([L,P,P]=[[L,P],P]\). However, \([L,P]\leq R\) by centrality of~\(f\) (since \(p([L,P]) = [p(L),p(P)] = [K,B]\), the latter being trivial by definition of centrality). Therefore, we have \([L,L,P] \subobj [L,P,P] \subobj [R,P]\) so that the composite \(L \cosmash P \to [L,P] \to [L,P]/[R,P]\) is trivial on \(N_{L,P} \subobj L \cosmash P\) and thus, by Proposition~\ref{prop:internal descr of tensor}, induces the dotted arrow \(g\). Note that this \(g\) is a regular epimorphism.

  The rest of the proof is explained by means of the diagram
  \[
    \xymatrix{R\tensor R \ar[d] \ar[r] & R\tensor P \ar[d] \ar@{-{ >>}}[r] & R\tensor B \ar[d] \ar[r] & 0\\
    L\tensor R \ar@{-{ >>}}[d] \ar[r] & L\tensor P \ar@{-{ >>}}[r] \ar@{-{ >>}}[d] & L\tensor B \ar[r] \ar@{-{ >>}}[d] & 0\\
    K\tensor R \ar[d] \ar[r] & K\tensor P \ar@{-{ >>}}[r] \ar[d] & K\tensor B \pushout \ar[r] \ar[d] & 0\\
    0 & 0 & 0}
  \]
  whose marked square is a pushout since all rows and columns are exact (on the right) by Proposition~\ref{Proposition-RightExact-Tensor}. On the one hand, \(g\) vanishes on~\(R\tensor P\), since \(R \subobj L\) so \(e_{L,P}^{P}\colon {L\cosmash P\to [L,P]}\) restricts to \(e_{R,P}^{P}\colon {R\cosmash P\to [R,P]}\), which is then killed by the quotient. On the other hand, \(g(\Im(L\tensor R\to L\tensor P))=q([L,R])\leq q([P,R])\) is null in the quotient \([L,P]/[R,P]\). Therefore, our morphism \(g \from L \tensor P \to [L,P]/[R,P]\) factors through the cokernels \(L \tensor B\) and \(K \tensor P\), these two arrows then inducing an arrow from \(K \tensor B\) to \([L,P]/[R,P]\), as wanted, by the universal property of the pushout. Finally, this arrow is indeed a regular epimorphism, since so is \(g\).
\end{proof}

\subsection{Application: The lower central series.}
If \(X\) is \(n\)-nilpotent for some \(n\geq 2\) (as in Subsection~\ref{Subsection n-nilpotent core}, meaning that \(\gamma_{n+1}(X)=0\)), then the induced short exact sequence
\[
  \xymatrix{0 \ar[r] & \gamma_{n}(X) \ar@{{ |>}->}[r] & X \ar@{-{ >>}}[r] & \frac{X}{\gamma_{n}(X)} \ar[r] & 0}
\]
is a central extension, because \([\gamma_{n}(X),X]\leq \gamma_{n+1}(X)\)~\cite[Proposition~2.21]{HVdL}. As a consequence, we find that the sequence
\[
  \xymatrix{\gamma_{n}(X)\tensor X \ar[r] & \H_{2}(X) \ar[r] & \H_{2}\bigl(\tfrac{X}{\gamma_{n}(X)}\bigr) \ar[r] & \gamma_{n}(X)}
\]
is exact. This generalises a result in~\cite{Ganea} from groups to algebraically coherent semi-abelian categories.

\section{Right exactness of cross-effects}\label{Section Cross-Effects}
We now focus our attention on the proof of the fact that bilinear products are sequentially right exact. To do so, we will actually prove a more subtle right-exactness result for all cross-effects of arbitrary functors.

\subsection{Preliminaries.}\label{Prelim Cross}
We recall how the definition of cross-effects of functors given in~\cite{Baues-Pirashvili} in the case of groups extends to a general categorical framework~\cite{Actions, Hartl-Vespa}.

Let \(F\colon {\C} \to {\D}\) be a functor from a pointed category with finite sums \({\C}\) to a pointed finitely complete category \({\D}\). For \(n\geq 1\), the \defn{\(n\)-th cross-effect of~\(F\)} is the multi-functor
\[
  \cre_n(F)\colon \C^{n} \to \D\text{,}
\]
defined by \(\cre_n(F)(X_1,\ldots,X_n)=\Ker(r)\), where
\[
  r\colon F(X_1 + \dots + X_n) \to \prod_{k=1}^n F(X_1 + \dots + \widehat{X_k} + \dots + X_n)
\]
is such that \(\pi_k\comp r=F(r_k)\) for all \(k\in\{1,\dots,n\}\) (for \(n=1\), \(\widehat{X_1}=0\)). Here \(r_k\) is the arrow induced by the inclusions \(i_{X_j} \from X_j \to X_1 + \dots + \widehat{X_k} + \dots + X_n\) for \(j \neq k\) and the trivial arrow \(0 \from X_k \to X_1 + \dots + \widehat{X_k} + \dots + X_n\). When \(n>1\), we usually write
\[
  F(X_1|\cdots|X_n)=\cre_n(F)(X_1,\ldots,X_n)
\]
and
\[
  \iota_{X_1,\ldots,X_n}=
  \iota_{X_1,\ldots,X_n}^F = \ker(r) \colon F(X_1|\cdots|X_n) \to F(X_1 + \cdots + X_n)\text{.}
\]
The functor \(\cre_n(F)\) acts on morphisms in the obvious way that makes~\(\iota_{X_1,\ldots,X_n}\) natural. When \(F\) is the identity functor \(1_{\X}\) of \(\X\) we regain the cosmash product
\[
  X_1\cosmash \cdots\cosmash X_n = 1_{\X}(X_1|\cdots|X_n)\text{.}
\]

\begin{figure}
  \begin{tikzcd}[squared=6]
    F(X|Y) \arrow[d, nmono] \arrow[rd, nmono]                                                                                                                          \\
    F(X|Y) \rtimes F(X) \arrow[r, nmono] \arrow[d, shift left] & F(X+Y) \arrow[r, shift left] \arrow[d, shift left] & F(Y) \arrow[d, shift left] \arrow[l, shift left] \\
    F(X) \arrow[r, equal] \arrow[u, shift left]                & F(X) \arrow[r, shift left] \arrow[u, shift left]   & 0 \arrow[l, shift left] \arrow[u, shift left]
  \end{tikzcd}
  \caption{Computing \(F(X|Y)\)}\label{Figure F(X|Y)}
\end{figure}

Note that \(\cre_1(F)(X) = \Ker(F(0)\colon F(X)\to F(0))\), so that \(\cre_1(F) \cong F\) when the functor \(F\) is reduced (=~preserves zero). This allows us to use the notation \(F(X_1|\cdots|X_n)\) without ambiguity even for \(n=1\). In what follows, we shall restrict ourselves to reduced functors. Further note that if the category \(\D\) is semi-abelian, then the converse holds as well: since \(0\colon X\to 0\) is a split epimorphism, its image \(F(0)\) through \(F\) is a normal epimorphism, so that it is the cokernel of its kernel; as a consequence, \(\cre_1(F) \cong F\) implies \(F(0)\cong 0\).

We generalise the folding operations of Subsection~\ref{Subsection-Folding} to an arbitrary reduced functor \(F\), as follows: given objects \(X\) and \(Y\) in~\(\C\), we consider
\[
  S_{1,2}^{X,Y}\colon {F(X|Y|Y)\to F(X|Y)}\qquad\text{and}\qquad S_{2,1}^{X,Y}\colon {F(X|X|Y)\to F(X|Y)}\text{,}
\]
canonically induced by the respective morphisms
\[
  F(1_X+\nabla_Y)\colon F(X+Y+Y)\to F(X+Y)
\]
and \(F(\nabla_X+1_Y)\colon F(X+X+Y)\to F(X+Y)\). Likewise, we define \((S_2^F)_X\) as the composite
\[
  \xymatrix{F(X|X)=\cre_2(F)(X,X) \ar[r]^-{\iota^F_{X,X}} & F(X+X) \ar[r]^-{F(\nabla_X)} & F(X)\text{.}}
\]

For the next proposition, we may adapt the proof of~\cite[Lemma~2.12]{HVdL} which treats the special case announced in Lemma~\ref{Lemma Smash of Sum}. (Note that, in that article, the symbol \(\tensor\) is used for the cosmash instead of \(\cosmash\).) It suffices to observe that for \(F=(-)\cosmash Z\), the diagram in Figure~\ref{Figure F(X|Y)} will give rise to Figure~\ref{Figure Ternary}. (This is part of the phenomenon that cross-effects can be constructed recursively: for instance, \(X_1\cosmash\cdots\cosmash X_n\cong (X_1\cosmash(-))(X_2|\cdots|X_n)\) as made explicit in~\cite[Proposition~4.1]{SVdL3}. Note that~\cite[Lemma~2.22]{HVdL-arXiv} discusses a slightly different type of recursion.)

\begin{proposition}\label{crossdecomp'}
  Suppose that \(\X\) is finitely cocomplete homological, \(\C\)~is pointed with binary sums and \(F\colon{\C\to \X}\) preserves zero. Then we have a decomposition
  \[
    F(X+Y)=\bigl(F(X|Y)\rtimes F(X)\bigr)\rtimes F(Y)
  \]
  for any \(X\), \(Y\) in \(\C\).\noproof
\end{proposition}

\subsection{Exact forks.}\label{Subsection Basic Principle}
Our first aim is to prove that the cross-effects of a functor which preserves coequalisers of reflexive graphs still preserve coequalisers of reflexive graphs. Our proof shall be based on the following \defn{basic principle} concerning those coequalisers, valid in semi-abelian categories. (But not in merely homological ones!) Let
\begin{equation}\label{Refl-Coeq}
  \vcenter{\xymatrix@!0@R=5em@C=4em{R \ar@<1ex>[rr]^-{d} \ar@<-1ex>[rr]_-{c} \ar[rd]_-{r=\IndArr{d \and c}} && B \ar[ll]|-{e} \ar[dl] \ar[r]^-{f} & A\\
  & B\times_A B \ar@<1ex>[ru]^(.4){f_{1}} \ar@<-1ex>[ru]_(.4){f_{2}}}}
\end{equation}
be a reflexive graph with its coequaliser, the induced kernel pair \((B\times_A B,f_{1},f_{2})\) and the comparison morphism \(r\). Certainly both \(f\) and \(r\) are regular epimorphisms: the morphism \(f\) by definition, and \(r\) since it is the regular epi--part of the image factorisation of~\(\IndArr{d \and c} \from R \to B \times B\). (This image is a reflexive relation since so is \(R\), thus it is the kernel pair of its coequaliser, since we are in an exact Mal'tsev context, the conclusion following by uniqueness of the kernel pair.) But in fact, the converse also holds: given regular epimorphisms \(r\) and \(f\) as in the diagram, and such that \(f\) coequalises \(d\) and \(c\), the morphism \(f\) it is their coequaliser. (Here we only use that \(r\) is epic to check the universal property.) Hence any regular epimorphism \(f\) which coequalises \(d\) and \(c\) is their coequaliser if and only if the given morphism \(r\) is a regular epimorphism.

\begin{figure}
  \resizebox{\textwidth}{!}
  {\mbox{\(
  \xymatrix@!0@C=14em@R=5em{0 \ar[d] & 0 \ar@{.>}[d] & 0 \ar[d] & 0 \ar[d] \\
  F(X|R) \ar[r]^-{r'''} \ar@{{ |>}->}[d] & F(X|B)\times_{F(X|A)} F(X|B) \ar[d] \ar@<1ex>[r] \ar@<-1ex>[r] & F(X|B) \ar@{{ |>}->}[d] \ar[l] \ar[r]^-{F(1_{X}|f)} & F(X|A) \ar@{{ |>}->}[d]\\
  F(X+R) \ar@{}[rd]|-{\texttt{(i)}} \ar[r]^-{r''} \ar@{-{ >>}}[d] & F(X+B)\times_{F(X+A)} F(X+B) \ar[d] \ar@<1ex>[r] \ar@<-1ex>[r] & F(X+B) \ar@{}[rd]|-{\texttt{(ii)}} \ar@{-{ >>}}[d] \ar[l] \ar[r]^-{F(1_{X}+f)} & F(X+A) \ar@{-{ >>}}[d]\\
  F(X)\times F(R) \ar[d] \ar[r]_-{r'} & (F(X) \times F(B)) \times_{F(X) \times F(A)} (F(X) \times F(B)) \ar@{.>}[d] \ar@<1ex>[r] \ar@<-1ex>[r] & F(X)\times F(B) \ar[d] \ar[l] \ar[r]_-{F(1_{X})\times F(f)} & F(X)\times F(A) \ar[d]\\
  0 & 0 & 0 & 0}
  \)}}
  \caption{The functor \(F(X|-)\) applied to a reflexive graph}\label{RG}
\end{figure}

\begin{lemma}\label{Lemma-Reduced}
  Suppose that \(\X\) is a homological category, \(\C\) is pointed category with binary sums and \(F\colon{\C\to \X}\) preserves zero. Then for any \(X\), \(Y\in \C\) the morphism
  \[
    \IndArr{F(r_{X}) \and F(r_{Y})} \colon{F(X+Y)\to F(X)\times F(Y)}
  \]
  is a regular epimorphism. Hence also the comparison natural transformation
  \[
    F(X+(-))\To F(X)\times F(-)
  \]
  is regular-epic.
\end{lemma}
\begin{proof}
  Since the functor \(F\) preserves zero, the triangle
  \[
    \xymatrix@!0@R=5em@C=6em{F(X)+F(Y) \ar[rr]^-{\DCoindArr{F(i_{X}) \and F(i_{Y})}} \ar@{-{ >>}}[rd]_(.4){r= \left\links\begin{smallmatrix}
          1_{F(X)} & 0        \\
          0        & 1_{F(Y)}
        \end{smallmatrix}\right\rechts} && F(X+Y) \ar[ld]^(.4){\DIndArr{F(r_{X}) \and F(r_{Y})}} \\
    & F(X)\times F(Y) }
  \]
  commutes. The result follows, since \(r\) is a regular epimorphism (because we are in a unital context). For the natural transformation, the claim follows since regular epimorphisms are pointwise.
\end{proof}

\begin{theorem}\label{Theorem-Preservation-Coequalisers}
  Suppose that \(\X\) is semi-abelian and \(\C\) is pointed with binary sums.
  Let \(F\colon{\C\to \X}\) be a functor which preserves zero and coequalisers of reflexive graphs. For any object~\(X\) of \(\C\), the induced functor \(F(X|-)\colon{\C\to \X}\) also preserves coequalisers of reflexive graphs. Hence, by induction (see~\cite[Section~2.19]{HVdL-arXiv}), so do all resulting functors \(F(X_1|\cdots|X_{k-1}|-|X_{k+1}|\cdots|X_n)\).
\end{theorem}
\begin{proof}
  Consider in \(\C\) a reflexive graph with its coequaliser~\eqref{Refl-Coeq} and the induced diagram in~\(\X\)---Figure~\ref{RG}---that shows how the functor \(F(X|-)\) works on this reflexive graph. By the ``basic principle'' it suffices to prove that both \(F(1_{X}|f)\) and~\(r'''\) are regular epimorphisms.

  In the bottom row, the morphisms \(r'\) and \(F(1_{X})\times F(f)\) are regular-epic by the assumption that \(F\) preserves coequalisers of reflexive graphs and the fact that also the product functor \(F(X)\times(-)\) does. Indeed, products preserve regular epimorphisms and kernel pairs, while \((F(X) \times F(B)) \times_{F(X) \times F(A)} (F(X) \times F(B)) \cong F(X)\times (F(B)\times_{F(A)} F(B))\), since kernel pairs commute with products. In particular, \(F(1_{X})\times F(f)\) is the coequaliser of the morphisms \(F(1_{X})\times F(d)\) and \(F(1_{X})\times F(c)\).

  In the middle row, the morphisms \(r''\) and \(F(1_{X}+f)\) are regular-epic because the sum functor \(X+(-)\) and the functor \(F\) preserve coequalisers of reflexive graphs. In particular, \(F(1_{X}+f)\) is the coequaliser of \(F(1_{X}+d)\) and \(F(1_{X}+c)\).

  The four lower vertical arrows in the diagram are regular epimorphisms by Lemma~\ref{Lemma-Reduced} and by the fact that \(r'\) is a regular epimorphism.

  In the category of extensions \(\Ext(\X)\), coequalisers are computed degree-wise by~\cite[Corollary~3.10]{EGVdL}. In fact, this result says that the commutative square \texttt{(ii)} in Figure~\ref{RG} may be considered as a double extension in~\(\X\).

  Also the square \texttt{(i)} is a double extension in~\(\X\). To see this, consider the following diagram with exact rows, in which \(r'=1_{F(X)}\times \overline{r}\).
  \[
    \resizebox{\textwidth}{!}{
    \xymatrix{0 \ar[r] & \Ker(r'') \ar@{{ |>}->}[r] \ar@{.>}[d] & F(X+R) \ar@{}[rd]|-{\texttt{(i)}} \ar@{-{ >>}}[r]^-{r''} \ar@{-{ >>}}[d] & F(X+B)\times_{F(X+A)} F(X+B) \ar[r] \ar@{-{ >>}}[d] & 0\\
    0 \ar[r] & \Ker(r') \ar@{=}[d] \ar@{{ |>}->}[r] & F(X)\times F(R) \pullback \ar@{-{ >>}}[r]^-{r'} \ar[d]_-{\pi_{F(R)}} & F(X)\times (F(B)\times_{F(A)} F(B)) \ar[r] \ar[d]^-{\pi_{F(B)\times_{F(A)} F(B)}} & 0\\
    0 \ar[r] & \Ker(r') \ar@{{ |>}->}[r] & F(R) \ar@{-{ >>}}[r]_-{\overline{r}} & F(B)\times_{F(A)} F(B) \ar[r] & 0}}
  \]
  The right and middle composed vertical arrows in it are compatibly split epimorphisms: the first one is split by \(F(i_R)\) and the second one by \(F(i_B) \times_{F(i_A)} F(i_B)\) as in the diagram
  \[
    \begin{tikzcd}[column sep=tiny, cramped]
      F(B) \times_{F(A)} F(B) \arrow[rr, dashed] \arrow[rd] \arrow[dd] &                                                          & F(X+B) \times_{F(X+A)} F(X+B) \arrow[rd] \arrow[dd]                     \\
      & F(B) \arrow[rr, "F(i_B)" near start, crossing over]      &                                                     & F(X+B) \arrow[dd] \\
      F(B) \arrow[rr, "F(i_B)" near end] \arrow[rd]                    &                                                          & F(X+B) \arrow[rd]                                                       \\
      & F(A) \arrow[from=uu, crossing over] \arrow[rr, "F(i_A)"] &                                                     & F(X+A)
    \end{tikzcd}
  \]
  Hence the left hand side dotted arrow is a split, hence a regular, epimorphism. Lemma~\ref{Lemma-LeftRight} now implies that the square \texttt{(i)} is a double extension (=~regular pushout square).

  Since kernels commute with kernel pairs the second column is exact, so that applying Lemma~\ref{Lemma-LeftRight}(2) on Figure~\ref{RG} implies that~\(F(1_{X}|f)\) and~\(r'''\) are regular-epic, and the result follows by the ``basic principle''.
\end{proof}

\begin{corollary}\label{Corollary-Coequalisers}
  For any object \(X\) in a semi-abelian category \(\X\), the induced functor \(X\cosmash(-)\colon{\X\to \X}\) preserves coequalisers of reflexive graphs.\noproof
\end{corollary}

\begin{lemma}[Proposition~3.9 in~\cite{EverVdL2}]
  \label{lem:coeq and coker}
  In a pointed protomodular category, let \begin{tikzcd}[cramped]
    R \arrow[r, "d", shift left=1.75] \arrow[r, "c"', shift right=1.75] & B \arrow[r, "f"] \arrow[l, "e" description] & A
  \end{tikzcd} be a \defn{reflexive fork}: \(c\comp e=1_B=d\comp e\) and \(f\comp c=f\comp d\). Then
  \[
    \begin{tikzcd}
      \Ker(d) \arrow[rr, "c \circ \ker(d)"] &  & B \arrow[r, "f"] & A \arrow[r] & 0
    \end{tikzcd}
  \]
  is a cokernel if and only if \(f\) is a coequaliser of \(c\) and \(d\).\noproof
\end{lemma}

Note that protomodularity is essential here, even when \(f\) is known to be a normal epimorphism. In the category of pointed sets, for instance, take \(B=\{*,a\}\) and \(R=\{*,a,x\}\), and let \(d\), \(c\colon{R\to B}\) restrict to the identity on \(B\), while \(d(x)=a\) and \(c(x)=*\); the inclusion of \(B\) into \(R\) makes this a reflexive graph. The coequaliser of \(c\) and \(d\) is the normal epimorphism \(f\colon{B\to 0}\), a cokernel of \(1_{B}\). On the other hand, \(\Ker(d)=\{*\}\), so that \(c\comp \ker(d)=0\), whose cokernel is \(1_{B}\neq f\).

\begin{proposition}\label{Proposition-Rightex}
  Suppose that \(\C\) is pointed with binary coproducts, \(\X\) is semi-abelian and \(F\colon{\C\to \X}\) preserves zero. If \(F\) preserves coequalisers of reflexive graphs, then every cokernel
  \[
    \xymatrix{K \ar[r]^-{k } & B \ar@{-{ >>}}[r]^-{f} & A \ar[r] & 0}
  \]
  in \(\C\) gives rise to a right exact sequence
  \begin{equation}\label{Sequence-Rightex}
    \xymatrix@=4em{
    F(K|B) \rtimes F(K) \ar[rr]^-{\DCoindArr{(S^{F}_{2})_{B}\comp F(k |1_{B}) \and F(k )}} && F(B) \ar@{-{ >>}}[r]^-{F(f)} & F(A) \ar[r] & 0}
  \end{equation}
  in~\(\X\)---see Subsection~\ref{Prelim Cross} for the notation \(S^{F}_{2}\). The converse implication holds as soon as \(\C\) is, in addition, protomodular: for instance, when \(\C\) is homological or semi-abelian.
\end{proposition}
\begin{proof}
  Suppose that \(F\) preserves coequalisers of reflexive graphs. For a cokernel as above, the diagram
  \[
    \xymatrix@=4em{K+B \ar@<1ex>[r]^-{\DCoindArr{k \and 1_{B}}} \ar@<-1ex>[r]_-{\DCoindArr{0 \and 1_{B}}} & B \ar[l]|-{i_{B}} \ar@{-{ >>}}[r]^-{f} & A}
  \]
  is a reflexive graph with its coequaliser, hence so is its image
  \begin{equation}\label{RG-Coeq}
    \xymatrix@=4em{F(K+B) \ar@<1.2ex>[r]^-{F\DCoindArr{k \and 1_{B}}} \ar@<-1.2ex>[r]_-{F\DCoindArr{0 \and 1_{B}}} & F(B) \ar[l]|-{F(i_{B})} \ar@{-{ >>}}[r]^-{F(f)} & F(A)}
  \end{equation}
  through \(F\). Since the kernel of \(F\CoindArr{0 \and 1_{B}}\) is \(F(K|B)\rtimes F(K)\) by Proposition~\ref{crossdecomp'}, the sequence
  \[
    \xymatrix@=5em{
    F(K|B) \rtimes F(K) \ar[r]^-{F\DCoindArr{k \and 1_{B}}\circ j} & F(B) \ar@{-{ >>}}[r]^-{F(f)} & F(A) \ar[r] & 0}
  \]
  is a cokernel by Lemma~\ref{lem:coeq and coker}, where \(j\colon{F(K|B) \rtimes F(K) \to F(K+B)}\) is the canonical inclusion, a normal monomorphism. Hence already the morphism \(F\CoindArr{k \and 1_{B}}\comp j\), as any normalisation of a reflexive graph, is proper: it is a composite of a split epimorphism with a kernel, so an image of a kernel along a regular epimorphism, which in a semi-abelian category is always a normal monomorphism. Furthermore, this morphism decomposes on the semi-direct product as claimed: first of all, \(F\CoindArr{k \and 1_{B}}\comp F(i_{K})=F(k)\); secondly,
  \begin{equation*}\label{Fiota=S2F}
    \begin{aligned}
      F\CoindArr{k \and 1_{B}}\comp \iota_{K,B} & = F\bigl(\nabla_{B}\bigr)\comp F(k +1_{B})\comp\iota_{K,B}                                            \\
                                                & = F\bigl(\nabla_{B}\bigr)\comp\iota_{B,B}\comp F(k |1_{B}) = (S^{F}_{2})_{B}\comp F(k |1_{B})\text{.}
    \end{aligned}
  \end{equation*}

  Conversely, suppose that \(\C\) is moreover protomodular, and let
  \[
    \xymatrix@C=4em{R \ar@<1ex>[r]^-{d} \ar@<-1ex>[r]_-{c} & B \ar[l]|-{e} \ar@{-{ >>}}[r]^-{f} & A}
  \]
  be a reflexive graph with its coequaliser. Then its normalisation
  \[
    \xymatrix@C=4em{\Ker(d) \ar[r]^-{c\circ\ker (d)} & B \ar@{-{ >>}}[r]^-{f} & A \ar[r] & 0}
  \]
  is a cokernel by Lemma~\ref{lem:coeq and coker}, now applied in the pointed protomodular category \(\C\), hence for \(k =c\comp\ker (d)\colon{K=\Ker(d)\to B}\) we obtain the right exact sequence~\eqref{Sequence-Rightex}. By the same computations as for the first implication, we find that \(F\CoindArr{k \and 1_{B}}\comp j\) is the normalisation of the reflexive graph in~\eqref{RG-Coeq}. By the converse implication in Lemma~\ref{lem:coeq and coker}---this time applied in the semi-abelian category \(\X\)---we see that \(F(f)\) is its coequaliser.

  It remains to transfer this property to the image through \(F\) of the given reflexive graph. The morphism \(\CoindArr{\ker(d) \and e}\colon{K+B\to R}\) satisfies \(d\comp \CoindArr{\ker(d) \and e}=\CoindArr{0 \and 1_{B}}\) and \(c\comp \CoindArr{\ker(d) \and e}=\CoindArr{k \and 1_{B}}\). Hence any morphism \(g\colon{F(B)\to Y}\) such that \(g\comp F(c)=g\comp F(d)\) also satisfies \(g\comp F\CoindArr{k \and 1_{B}}=g\comp F\CoindArr{0 \and 1_{B}}\), so factors uniquely through the coequaliser \(F(f)\) of the reflexive graph in~\eqref{RG-Coeq}. Since \(F(f)\comp F(c)=F(f)\comp F(d)\), this proves that \(F(f)\) is the coequaliser of \(F(c)\) and \(F(d)\).
\end{proof}

Combined with Theorem~\ref{Theorem-Preservation-Coequalisers}, this gives us:

\begin{theorem}\label{Theorem-Rightex}
  Suppose that \(\C\) is pointed with binary coproducts, \(\X\) is semi-abelian and \(F\colon{\C\to \X}\) preserves zero and coequalisers of reflexive graphs. Consider an object \(X\) and a cokernel
  \[
    \xymatrix{K \ar[r]^-{k } & B \ar@{-{ >>}}[r]^-{f} & A \ar[r] & 0}
  \]
  in \(\C\). Then we obtain the right exact sequence
  \[
    \resizebox{\textwidth}{!}
    {\xymatrix@=5em{F(X|K|B) \rtimes F(X|K) \ar[rr]^-{\DCoindArr{S_{1,2}^{X,B}\circ F(1_{X}|k |1_{B}) \and F(1_{X}|k )}} && F(X|B) \ar@{-{ >>}}[r]^-{F(1_{X}|f)} & F(X|A) \ar[r] & 0}}
  \]
  in~\(\X\).\noproof
\end{theorem}

\begin{proof}[Proof of Theorem~\ref{Theorem-RightExact-HVdL}]\label{Proof of Theorem-RightExact-HVdL}
  Applying Theorem~\ref{Theorem-Rightex} to \(1_\X\), we obtain the exactness result for the functor \(1_\X(X|-)=X\cosmash(-)\colon{\X\to \X}\).
\end{proof}

\section{Abelian extensions}\label{Section Abelian Extensions}
We present another application of the right-exactness of cosmash products: a characterisation of the abelian extensions in a semi-abelian category.

There is a subtle difference between the concept of \defn{extension with abelian kernel}---any short exact sequence
\begin{equation}\label{exte}
  \xymatrix{0 \ar[r] & A \ar@{{ |>}->}[r]^-{a} & X \ar@{-{ >>}}[r]^-{p} & G \ar[r] & 0}
\end{equation}
where the kernel \(A\) is abelian---and the notion of \defn{abelian extension}, a regular epimorphism \(p\colon{X\to G}\) which is an abelian object in the (non-pointed) slice category \((\X\downarrow G)\). Since ``abelian object'' here means that \(p\) admits an internal Mal'tsev operation\footnote{In a pointed setting, these internal Mal'tsev operations would become internal abelian groups, but here such a reduction cannot be made. The pointed case---abelian objects in a category of points, which are precisely the Beck modules---is deferred to Section~\ref{Section Tensor of G-actions}.}, this amounts to the condition that the Smith/Pedicchio commutator \([X\times_{G}X,X\times_{G}X]\) is trivial (see, for instance, the analysis made in~\cite{Bourn-Janelidze:Torsors}, and~\cite{Smith,Pedicchio} for the definition of the commutator). We write \(\Ab\Ext(\X)\) for the full subcategory of \(\Ext(\X)\) determined by the abelian extensions in~\(\X\).

For some purposes in homological algebra, one needs extensions with abelian kernel to be abelian extensions. Certainly, any abelian extension has an abelian kernel, since taking the kernel preserves the internal abelian group structure. In general, the converse need not hold: unlike what happens for groups (where the so-called ``Smith is Huq'' condition~\cite{MFVdL} holds), in an arbitrary semi-abelian category an extension with abelian kernel need not be an abelian extension (see~\cite{Borceux-Bourn, Bourn2004, HVdL}). Using right exactness of cosmash products in the form of Theorem~\ref{Theorem-Rightex}, here we analyse the situation from the point of view of internal actions and ternary commutators.

Let us start with some preliminary reminders on internal actions. In a semi-abelian category \(\X\), an \defn{(internal) action} of an object \(X\) on an object \(A\) is an arrow \(\xi \colon A \talf X \to A\) making certain diagrams commute, where \(A \talf X\) is constructed by choosing a coproduct \(A+X\) of \(A\) and \(X\) then taking a kernel \(\kappa_{A,X} \colon A \talf X \to A+X\) of the arrow \(\CoindArr{0 \and 1_X} \colon A+X \to X\)~\cite[Sections~3.2--3.3]{BJK}. (This construction gives us a bifunctor \(\talf \colon \X \times \X \to \X\).) However, in the following, we will consider another point of view of internal actions. The cosmash product \(A \cosmash X\) is canonically included in \(A \talf X\) so we can consider the restriction \(\psi \colon A \cosmash X \to A\) of an action \(\xi \colon {A \talf X \to A}\), called the \defn{core} of \(\xi\). Conversely, an abstract morphism \(\psi \colon A \cosmash X \to A\) is called an \defn{action core} if it extends to an action \(\xi \colon {A \talf X \to A}\), and this happens if and only if \(\psi\) renders three diagrams involving (simple or nested) binary and ternary cosmash products of \(A\) and \(G\) commutative, see~\cite[Theorem 5.9]{Actions}. Whenever this is the case, the core \(\psi\) entirely determines the action \(\xi\); for further details, see~\cite[Section~0.1]{HVdL} and~\cite{Actions}.

\begin{lemma}\label{factoraction}
  Suppose that \(\X\) is semi-abelian, \(\psi\colon A\cosmash X\to A\) determines an action in~\(\X\) and \(p\colon {X \to G}\) a regular epimorphism such that there exists a (necessarily unique) factorisation
  \[
    \xymatrix{A\cosmash X \ar[r]^-{\psi} \ar@{-{ >>}}[d]_-{1_{A}\cosmash p} & A \ar@{=}[d] \\
    A\cosmash G \ar@{->}[r]_-{\varphi} & A}
  \]
  of \(\psi\). Then \(\varphi\) determines an action of \(G\) on \(A\).
\end{lemma}
\begin{proof}
  This can most conveniently be proved using the extension of \(\psi\) to an algebra structure \(\xi\colon A\talf X \to A\), cf.~\cite{Actions} and~\cite{Bourn-Janelidze:Semidirect, Janelidze}, and using the fact that~\(p\) induces a regular epimorphism \(1_{A}\talf p\colon A\talf X \to A\talf G\). This allows to check the algebra conditions for the factorisation \(\zeta\colon A\talf G \to A\) of \(\xi\) by composing with~\(1_{A}\talf p\) and using the obvious commutative diagrams.
\end{proof}

Recall from~\cite{HVdL, Actions} that for a given normal subobject \(A\normal X\), the \defn{conjugation action} of \(X\) on \(A\) is determined by the unique restriction of \(c^{X,X}\coloneq \nabla_X\comp \iota_{X,X}\colon X\cosmash X\to X\) to a morphism \(c^{A,X}\colon A\cosmash X\to A\). It, in turn, induces morphisms
\[
  c^{A,X}_{1,2}\coloneq c^{A,X}\comp S_{1,2}^{A,X}\colon A\cosmash X\cosmash X\to A
\]
and
\[
  c^{A,X}_{2,1}\coloneq c^{A,X}\comp S_{2,1}^{A,X}\colon A\cosmash A\cosmash X\to A\text{.}
\]

\begin{lemma}\label{Lemma-abker}
  Consider the short exact sequence~\eqref{exte} in a semi-abelian category. The conjugation action \(c^{A,X}\) factors through \(1_{A}\cosmash p\) to a morphism \(\psi_{p}\colon {A\cosmash G\to A}\) such that \(\psi_{p}\comp (1_{A}\cosmash p)=c^{A,X}\) if and only if \(A\) is abelian and~\({c_{2,1}^{A,X} =0}\). Moreover, when this happens, the induced morphism \(\psi_{p}\) determines an action of \(G\) on~\(A\).
\end{lemma}
\begin{proof}
  Using the presentation of the morphism \(1_A \cosmash p\) as a cokernel given by Theorem~\ref{Theorem-RightExact-HVdL}, the action \(c^{A,X}\) factors through \(1_{A}\cosmash p\) if and only if
  \[
    c^{A,X}\comp S_{1,2}^{A,X}\comp(1_{A}\cosmash a\cosmash 1_{X})
    \qquad\text{and}\qquad
    c^{A,X}\comp(1_{A}\cosmash a)
  \]
  are trivial. But \(c^{A,X}\comp(1_{A}\cosmash a)=1_{A}\comp c^{A,A}=c^{A,A}\) by naturality of the conjugation action, and
  \[
    {c^{A,X}_{1,2}\comp(1_{A}\cosmash a\cosmash 1_{X}) = c^{A,X}_{2,1}}
  \]
  by~\cite[Lemma~3.14]{HVdL}. Lemma~\ref{factoraction} now says that \(\psi_{p}\) determines an action.
\end{proof}

\begin{theorem}\label{Theorem-Extab}
  For an extension with abelian kernel~\eqref{exte} the following are equivalent:
  \begin{tfae}
    \item \(p\) is an abelian extension;
    \item the Smith/Pedicchio commutator \([X\times_{G}X,X\times_{G}X]\) vanishes;
    \item \([A,A,X]=0\);
    \item \(c_{2,1}^{A,X} =0\);
    \item the conjugation action \(c^{A,X}\) of \(X\) on \(A\) factors through a (necessarily unique) action \(\psi_{p}\) of \(G\) on \(A\);
    \item there exists a morphism \({\psi_{p}\colon A\cosmash G\to A}\) of which \(c^{A,X}\) is the \defn{pullback} \(p^{*}(\psi_{p})\coloneq \psi_p\comp (1_A\cosmash p)\) along \(p\); see Notation~3.9 in~\cite{HVdL}.
  \end{tfae}
\end{theorem}
Note that unlike in~(v), in~(vi) we do not assume a priori that the factorisation \(\psi_p\) determines an action.
\begin{proof}
  (i) \(\Leftrightarrow\) (ii) holds by definition and (ii) \(\Leftrightarrow\) (iii) by~\cite[Theorem~4.4]{HVdL}. We have (iii) \(\Leftrightarrow\) (iv) since the commutator \([A,A,X]\) is the image of the morphism \(c_{2,1}^{A,X}\colon A\cosmash A\cosmash X\to A\). (iv) \(\Leftrightarrow\) (v) \(\Leftrightarrow\) (vi) is Lemma~\ref{Lemma-abker}.
\end{proof}

\begin{corollary}\label{Corollary-AbExt}
  If \(\X\) is semi-abelian, then the inclusion \(\Ab\Ext(\X)\to \Ext(\X)\) has a left adjoint \(\ab\colon \Ext(\X)\to \Ab\Ext(\X)\) which takes an extension \(p\colon{X\to G}\) and maps it to its induced quotient
  \[
    \ab(p)\colon{\tfrac{X}{[A,A]\join[A,A,X]}\to G}\text{.}
  \]
\end{corollary}
\begin{proof}
  Note that that this quotient is a good candidate since we want the kernel to be abelian (\([A,A] = 0\)) and the commutator \([A,A,X]\) to vanish.

  Lemma~\ref{Lemma-Join} and the commutator calculus rules of~\cite[Proposition~2.21]{HVdL} may be used to prove that \([A,A]\join[A,A,X]\) is normal in \(X\): indeed
  \begin{align*}
    [[A,A]\join[A,A,X],X] & = [[A,A],X]\join[[A,A,X],X]\join [[A,A],[A,A,X],X] \\
                          & \leq [A,A,X]\join[A,A,X,X]\join [A,A,A,A,X,X]      \\
                          & \leq [A,A,X] \leq [A,A]\join[A,A,X]
  \end{align*}
  so that normality follows from~\cite[Theorem~6.3]{MM-NC}.

  Next, \(p\) does indeed factor through the given quotient, since by \cite[Theorem~6.3]{MM-NC},
  \[
    [A,A] \join [A,A,X] \subobj [A,A] \join [A,X] = [A,X] \subobj A
  \]
  because \(A\) is normal in \(X\). The quotient \(\ab(p)\) is indeed abelian, essentially because
  \[
    \Bigl[\tfrac{A}{[A,A]\join[A,A,X]},\tfrac{A}{[A,A]\join[A,A,X]},\tfrac{X}{[A,A]\join[A,A,X]}\Bigr] = \tfrac{[A,A,X]}{[A,A]\join[A,A,X]} = 0\text{,}
  \]
  with a similar computation to check that its kernel is abelian. The universal property is now easily checked using functoriality of the construction and the fact that this functor leaves abelian extensions fixed.
\end{proof}

\section{The tensor product of \texorpdfstring{\(G\)}{G}-representations}\label{Section Tensor of G-actions}
Internal actions over a fixed object \(G\) can be viewed as non-abelian \(G\)-modules. We may ask the question, what is their bilinear product. Since ultimately, only the result after abelianisation counts, what we really must understand is how to tensor two \emph{abelian} actions: so, two \emph{Beck modules} over \(G\) or, in other words, two \emph{\(G\)-representations}. Thus, the bilinear product yields a tensor product on Beck modules in any semi-abelian category, generalising the one of representations of groups and Lie algebras, as we shall prove in Example~\ref{Example Lie algebra representations} and Example~\ref{Example group representations}. Computing the bilinear product of two abelian \(G\)-actions crucially involves cosmash products of non-abelian actions, so we need to have a look at those first.

In this section, we view a \(G\)-action \(\xi\) on an object \(A\) as a split short exact sequence
\begin{equation}\label{Split extension induced by action xi}
  \xymatrix{0 \ar[r] & A \ar@{{ |>}->}[r]^-{k_\xi} & A\rtimes_\xi G \ar@{-{ >>}}@<.5ex>[r]^-{p_\xi} & G \ar@{{ >}->}@<.5ex>[l]^-{s_\xi} \ar[r] & 0}
\end{equation}
(via the semi-direct product equivalence~\cite{Bourn-Janelidze:Semidirect}). The category of \emph{all} \(G\)-actions may thus be seen as the semi-abelian category \(\Pt_G(\X)\) of points over \(G\) in \(\X\): split epimorphisms with codomain \(G\), each with a chosen splitting~\cite{Bourn1991,Borceux-Bourn}. Through this equivalence, a point \((p\colon X\to G,\; s\colon G\to X)\), \(p\comp s=1_G\), corresponds to the unique \(G\)-action \(\xi\) on the kernel \(A\) of \(p\) which turns \(X\) into the semi-direct product \(A\rtimes_\xi G\) with \(p_\xi=p\) and \(s_\xi=s\).

Thus the kernel functor \(U\colon\Pt_G(\X)\to \X\) plays the role of the forgetful functor sending an internal \(G\)-action to its underlying object (the object upon which \(G\) acts). Note that the functor \(U\) is an exact functor between semi-abelian categories (i.e., it preserves short exact sequences). It is, in fact, a right adjoint; the associated left adjoint sends an object \(X\) of \(\X\) to the point \((\CoindArr{1_G \and 0},i_1)\colon G+X\point G\), which determines a \(G\)-action on the object \(G\flat X\) called \defn{free \(G\)-action generated by \(X\)}.

\begin{example}[The cosmash product of free actions]\label{Cosmash of free actions}
  Recall that the product of two points \((p,s)\) and \((q,t)\) over \(G\) is obtained as the pullback of \(p\) and \(q\), while their sum is the pushout of \(s\) and \(t\). For instance, the coproduct of the free \(G\)-actions \((\CoindArr{1_G \and 0}, i_1)\colon {G+X\point G}\) and \((\CoindArr{1_G \and 0}, i_1)\colon {G+Y\point G}\) generated by objects \(X\) and \(Y\) is the free \(G\)-action \((\CoindArr{1_G \and 0}, i_1)\colon G+X+Y\point G\) generated by \(X+Y\).

  It now follows immediately from the definitions---see Figure~\ref{Figure Cross Of Flat}
  \begin{figure}
    \hfil\xymatrix@!0@C=7em@R=4em{G\flat(X+Y) \ar@{-{ >>}}[rrrr] \ar@{{ |>}->}[rd] &&&& G\flat X\times G\flat Y \ar@{{ |>}->}[dl]\\
    & G + X + Y \ar@{-{ >>}}[rr]^-{\left\links\begin{smallmatrix}
      i_{G} & i_{X} & 0     \\
      i_{G} & 0     & i_{Y}
    \end{smallmatrix}\right\rechts} \ar@{-{ >>}}@<.5ex>[rd] && (G+X)\times_G(G+Y) \ar@{-{ >>}}@<.5ex>[ld]\\
    && G \ar@{{ >}->}@<.5ex>[lu] \ar@{{ >}->}@<.5ex>[ru] \ar@{{ >}->}@<.5ex>[ld] \ar@{{ >}->}@<.5ex>[rd]^-{=}\\
    & K\rtimes G \ar@{-{ >>}}@<.5ex>[ur] \ar@{-{ >>}}[rr] \ar@{{ >}->}[uu] && G \ar@{{ >}->}[uu] \ar@{-{ >>}}@<.5ex>[ul]^-{=} \\
    K \ar@{{ |>}->}[ur] \ar@{{ |>}->}[uuuu] \ar@{-{ >>}}[rrrr] &&&& 0 \ar@{{ |>}->}[uuuu] \ar@{{ |>}->}[ul]}\hfil
    \caption{Calculating the cosmash product \(K\) of the free \(G\)-actions generated by \(X\) and \(Y\). All commutative squares in this diagram in \(\X\) are pullbacks; the middle square depicts \(K\rtimes G\) as the kernel of \(\left\links\begin{smallmatrix}
        i_{G} & i_{X} & 0     \\
        i_{G} & 0     & i_{Y}
      \end{smallmatrix}\right\rechts\) in \(\Pt_G(\X)\).}\label{Figure Cross Of Flat}
  \end{figure}
  ---that the kernel object~\(K\) in the short exact sequence
  \begin{equation*}\label{Kernel K}
    \xymatrix@=2em{0 \ar[r] & K \ar@{{ |>}->}[rr] && G+X+Y \ar@{-{ >>}}[rr]^-{\left\links\begin{smallmatrix}
      i_{G} & i_{X} & 0     \\
      i_{G} & 0     & i_{Y}
    \end{smallmatrix}\right\rechts} && (G+X)\times_G(G+Y) \ar[r] & 0}
  \end{equation*}
  could be denoted \((G\flat X)\cosmash_{G}(G\flat Y)\), since it is the underlying object of the \(G\)-action obtained as the cosmash product in \(\Pt_{G}(\X)\) of the free \(G\)-actions generated by \(X\) and \(Y\), which we here represent by means of their underlying objects \(G\flat X\) and \(G\flat Y\). On the other hand, since it is also the kernel of the canonical arrow \(G\flat (X+Y)\to G\flat X\times G\flat Y\), it is precisely the second cross effect \(G\flat(X|Y)\) of the endofunctor \(G\flat(-)\colon \X\to \X\), evaluated in \(X\) and \(Y\).
\end{example}

A \(G\)-action is abelian precisely when it is abelian as an extension; i.e., when it is a \defn{Beck module}~\cite{Beck,Barr-Beck} or \defn{\(G\)-representation}---see~\cite{Bourn-Janelidze:Semidirect,Bourn-Janelidze:Torsors} as well as Section~6 of~\cite{HVdL} where this is explained in detail from the viewpoint of the higher Higgins commutator.

In general, it is hard to compute the bilinear product of two \(G\)-actions, since the forgetful functor \(\Pt_G(\X)\to \X\) need not preserve coproducts. The situation simplifies a lot once we assume \(\X\) satisfies some appropriate additional conditions. Here we are interested in the following question:
\begin{quote}
  \emph{Under which conditions can the bilinear product of two \(G\)-actions be computed as the bilinear product of the underlying objects, equipped with a suitable \(G\)-action?}
\end{quote}
Surely, when this happens, the shape of the symmetry isomorphism is inherited from the underlying category \(\X\).

A first simplification occurs when \(\X\) is \emph{algebraically coherent}~\cite{acc}. As already mentioned in Section~\ref{Section Ganea}, this means that the forgetful functor \({\Pt_G(\X)\to \X}\) is coherent (i.e., it does not only preserve finite limits, but also jointly extremal-epimorphic pairs of arrows).

Proposition~6.9 in~\cite{acc} tells us that if \(\X\) is algebraically coherent, then the forgetful functor \(\Pt_G(\X)\to \X\) preserves binary Higgins commutators. It is easy to adapt the proof of this result to see that Higgins commutators of arbitrary length are preserved. Since the kernel functor reflects the zero object (which in the category \(\Pt_G(\X)\) is the point \((1_G,1_G)\)), a \(G\)-action on an object \(A\) is abelian or two-nilpotent whenever so is the underlying object \(A\).

For the forgetful functor \(\Pt_G(\X)\to \X\) to preserve cosmash products (so that, in particular, bilinear products are preserved, since the objects \(N_{X,Y}\) in Proposition~\ref{prop:internal descr of tensor} are constructed in terms of cosmash products only) we may use that it is always continuous (so that binary products and the terminal object are preserved) and exact (=~preserves short exact sequences), and ask that it preserves binary coproducts. This condition is strictly stronger than algebraic coherence. It happens to be equivalent to \defn{local algebraic cartesian closedness \LACC}---the condition that the change-of-base functors in the fibration of points are left adjoints~\cite{Gray2012, Bourn-Gray,GrayPhD}---when \(\X\) is a variety of universal algebras. Examples of semi-abelian \LACC{} categories are relatively scarce, though. The only ones currently known in the literature are: essentially affine categories (which include all additive categories)~\cite{Bourn1991}; internal groups in a cartesian closed category with pullbacks (which include the examples of classical groups, crossed modules and cocommutative Hopf algebras)~\cite{Bourn-Gray,GM-VdL1}; and internal Lie algebras in an additive cocomplete symmetric closed monoidal category (including classical Lie algebras)~\cite{GM-G}. On the other hand, almost any other known semi-abelian category is known \emph{not} to be \LACC\@: see, for instance~\cite{GM-VdL3} where the context of varieties of non-associative algebras over a field is treated in detail. There it is shown that over a field of characteristic zero, the variety of Lie algebras is the only non-abelian example.

\begin{example}[Lie algebra representations]\label{Example Lie algebra representations}
  We view a representation \((A,\xi)\) of an \(R\)-Lie algebra \(X\) as a morphism \(\xi\colon X\to \gl(A)\) where \(\gl(A)\) denotes the Lie algebra of linear endomorphisms of the \(R\)-module \(A\), whose Lie bracket is the commutator of endomorphisms. Here we use that \(\Lie_R\) is an algebraically coherent category (it being locally algebraically cartesian closed), which implies that a representation is the same thing as an action on an abelian object. The fact that \(\Lie_R\) is \LACC{} further predicts that the bilinear product of the representations \((A,\xi)\) and \((B,\zeta)\) is given by the bilinear product of the objects \(A\) and \(B\), equipped with a suitable action. Indeed, for any two such representations \((A,\xi)\) and \((B,\zeta)\), we may take their \defn{Kronecker sum}
  \[
    \xi\tensor \zeta\colon X\to \gl(A\tensor_RB)\colon x\mapsto \xi(x)\tensor 1_B + 1_A\tensor \zeta(x)\text{.}
  \]
  We claim that this provides the needed structure of an \(X\)-representation on the tensor product \(A\tensor_RB\). We thus find a categorical explanation for the use of the Kronecker sum in the tensor product of Lie algebra representations.

  Recall that a general \(X\)-action on a Lie algebra \(C\) is given by a Lie algebra morphism \(\theta\colon {X\to \Der(C)}\) where \(\Der(C)\) is the Lie algebra of derivations on \(C\), whose Lie bracket is the commutator \([D,D']=D\comp D'-D'\comp D\). A \defn{derivation} on \(C\) is an \(R\)-linear map \(D\colon C\to C\) such that \(D([x,y])=[D(x),y]+ [x,D(y)]\).

  We view the bilinear product of the representations \((A,\xi)\) and \((B,\zeta)\) as their cosmash product in the category \(\Nil_2(\Pt_X(\Lie_R))\), so the kernel of the comparison map from their sum \((A,\xi)+_2(B,\zeta)\) to their product \((A,\xi)\times(B,\zeta)\). Keeping in mind that \(\Lie_R\) is algebraically coherent (see the discussion above), we prove that \((A,\xi)+_2(B,\zeta)\) is\footnote{Proposition~\ref{coproduct in Nil_2(Alg-P)} in particular tells us that \([(a,0,0),(0,b,0)] = {(0,0,a\tensor b)}\).} the two-nilpotent Lie algebra \(A+_2B= A\oplus B\oplus (A\tensor_RB)\) equipped with the \(X\)-action
  \[
    \xi+_2 \zeta\colon X\to \Der(A\oplus B\oplus (A\tensor_RB))\colon x\mapsto (\xi(x) \oplus \zeta(x)\oplus (\xi(x)\tensor 1_B + 1_A\tensor \zeta(x)))\text{.}
  \]
  Upon taking the kernel of \((A,\xi)+_2(B,\zeta)\to (A,\xi)\times(B,\zeta)\), this action restricts to the Kronecker sum---which proves our claim.

  Let us verify that it does indeed satisfy the universal property of a coproduct of \((A,\xi)\) and \((B,\zeta)\) in \(\Nil_2(\Pt_X(\Lie_R))\). First of all, it is clear that the coproduct inclusions \(i_A\colon A\to A+_2B\) and \(i_B\colon B\to A+_2B\) are equivariant with respect to the given \(X\)-actions. Now given an \(X\)-action \((C,\theta)\) and equivariant Lie algebra morphisms \(f\colon A\to C\) and \(g\colon B\to C\), we need to check that the universally induced Lie algebra morphism \(\CoindArr{f \and g}\colon  A+_2B\to C\) is equivariant with respect to \(\xi+_2 \zeta\) and~\(\theta\). Indeed, for any  \((a,b,\sum_i(a_i\tensor b_i))\in A+_2B\) and any \(x\in X\), we have
  \begin{align*}
     & \CoindArr{f \and g}\bigl((\xi+_2 \zeta)(x)(a,b,  \sum_i(a_i\tensor b_i))\bigr)                                            \\
     & = \CoindArr{f \and g}\Bigl(\xi(x)(a), \zeta(x)(b),\sum_i\bigl(\xi(x)(a_i)\tensor b_i+a_i\tensor \zeta(x)(b_i)\bigr)\Bigr) \\
     & = f(\xi(x)(a)) + g(\zeta(x)(b)) + \sum_i \bigl([f(\xi(x)(a_i)), g(b_i)]+[f(a_i), g(\zeta(x)(b_i))]\bigr)                  \\
     & = \theta(x)(f(a)) + \theta(x)(g(b)) + \sum_i \bigl([\theta(x)(f(a_i)), g(b_i)]+[f(a_i), \theta(x)(g(b_i))]\bigr)          \\
     & = \theta(x)(f(a)) + \theta(x)(g(b)) + \sum_i \theta(x)\bigl([f(a_i), g(b_i)]\bigr)                                        \\
     & = \theta(x)\bigl(f(a) + g(b) + \sum_i [f(a_i), g(b_i)]\bigr)                                                              \\
     & = \theta(x)\bigl(\CoindArr{f \and g}(a,b,\sum_i(a_i\tensor b_i))\bigr)\text{,}
  \end{align*}
  where the fourth equality follows because \(\theta(x)\) is a derivation on \(C\).
\end{example}

The case of groups is similar:

\begin{example}[Group representations]\label{Example group representations}
  In the case of groups, a \(G\)-representation over \(\Z\) is the same thing as a \(\Z[G]\)-module, which can be described inside the category of groups as a split extension \eqref{Split extension induced by action xi} where the action \(\xi\) of \(G\) on the abelian group \(A\) is codified by means of a homomorphism \(\xi\colon G\to \Aut(A)\). We write \(ga\coloneq \xi(g)(a)\) for any \(g\in G\) and \(a\in A\). Given two representations \((A,\xi)\) and \((B, \zeta)\), their bilinear product is the tensor product \(A\tensor_\Z B\) (since \(\Gp\) is \LACC)\@, equipped with the action determined by \(g(a\tensor b)=ga\tensor gb\) (which is inherited from the action of \(G\) on \(A+B\) in the coproduct of \((A,\xi)\) and \((B, \zeta)\) in the category of \(G\)-actions, given by \(g(a_1\cdot b_1\cdot \cdots \cdot a_n\cdot b_n)=ga_1\cdot gb_1\cdot \cdots \cdot ga_n \cdot gb_n\), as is easily seen by checking the universal property of that coproduct). This is the usual tensor product of group representations.
\end{example}

\section{Further questions and remarks}\label{Section Further Questions}
We end this article with a few questions which remain open for now, and which in our opinion might lead to interesting further investigations.

\subsection{Bilinear products of representations}\label{reps}
In Section~\ref{Section Tensor of G-actions} we explained that if~\(\X\) is a locally algebraically cartesian closed semi-abelian category, then for any object \(G\) of \(\X\), the bilinear product of two \(G\)-representations is the bilinear product of the underlying objects, equipped with an appropriate \(G\)-action. The general situation, for arbitrary semi-abelian categories, appears to be much more complicated. It would seem interesting to obtain an explicit description of how to tensor two representations even in a category which is not \LACC\@, in particular in various varieties of loops and of related algebras. In fact, there exist varieties of loops \(\V\) where our bilinear product definitely differs from the usual tensor product of representations (in the sense of Beck modules in \(\V\), see \cite{Smith07}) since the latter is not of this type, for example when \(\V\) is the variety of Moufang loops; and similarly for the representations of Mal'tsev algebras, see \cite{P-I17}. This inconvenience disappears for certain types of representations considered in the literature (see \cite{Yamaguti63} for Mal'tsev algebras) which are more general than just Beck modules in \(\V\), in particular \cite{P-I17} for Beck modules in suitable categories larger than \(\V\). In such cases it may be interesting to compute and study our bilinear product of the original representations, that is of Beck modules in \(\V\), and to investigate potential applications.

In this direction, J.-M.\ P\'erez-Izquierdo recently communicated to us a formula allowing to compute the bilinear product of Beck modules in a variety of algebras over a given commutative ring with a single anticommutative bilinear  operation subject to an arbitrary set of multilinear relations (thus covering Mal'tsev algebras); and more generally, the second author exhibited a description of the bilinear product of representations in any semi-abelian variety, in terms of the binary and ternary cosmash product of the variety itself. These results will serve as a basis for further investigations, both of theoretical questions and of explicit examples.

Another interesting example are representations of cocommutative Hopf algebras over a field; in fact, the category of the latter is semi-abelian~\cite{GKV,GSV} so that the bilinear product is defined, and its computation is within reach as this category also is locally algebraically cartesian closed as explained above.

\subsection{Tensor products of topological vector spaces}
We defined the bilinear product in the context of a semi-abelian category, but this definition works as well for homological categories~\cite{Borceux-Bourn} with finite sums\footnote{Just keep in mind that certain results such as Theorem~\ref{Theorem-Preservation-Coequalisers} and its consequences might not hold at that level of generality.}. Hence we may ask the question, what it amounts to in the case of topological abelian groups or topological vector spaces. Is it a known tensor product? Note that only the nature of the topology is an issue, since algebraically these bilinear products are the classical tensor products.

\subsection{The variety generated by a quasivariety}
In algebra, any pointed protomodular quasivariety is an example of a homological category. We may wonder what is the relationship between the bilinear product in such a quasivariety and the bilinear product in the variety it generates (by closing under homomorphic images).

\subsection{When is the functor \(X\tensor(-)\) idempotent?}\label{Idempotency}
By Proposition~\ref{Prop pre SMC}, for any object \(X\) of a semi-abelian variety \(\X\), the functor \(X\tensor(-)\colon {\Ab(\X)\to \Ab(\X)}\) is a left adjoint. We may wonder when it is idempotent, so that it defines a reflector to an appropriate subcategory of \(\Ab(\X)\). A study of the properties of this subcategory---for instance, when is it Birkhoff?---may eventually lead to a description of the derived functors of the functor \(X\tensor(-)\) in terms of higher Hopf formulae following~\cite{EGVdL,Everaert-Gran-TT}.

\subsection{When is the bilinear product monoidal?}\label{Monoidality}
In this article, the main property we want a product on a category to satisfy is bilinearity. Associativity and existence of a unit are secondary here, and do indeed not need to hold in general, as we have seen in Subsection~\ref{Associative algebras} (no unit) and Example~\ref{non-associative bilinear product} (no associativity). The canonical symmetry isomorphism may cause further trouble in verifying coherence, since its sign may differ from the canonical twist of the tensor product. However, it would seem interesting to have a precise characterisation of those semi-abelian categories where the bilinear product defines a (closed) symmetric monoidal structure on the abelian core---cf.\ Proposition~\ref{Closed SMC}.

\subsection{Can we have both?}
Is it possible that a bilinear product satisfies idempotency as in~\ref{Idempotency} while being monoidal as in~\ref{Monoidality}? For varieties of universal algebras, this seems close to the situation where the reflector of \(\Ab(\X)\) to a subvariety \(\V\leq \Ab(\X)\) is of the form \(X\tensor(-)\colon \Ab(\X)\to\V\), for some \(X\) in \(\V\) which is a quotient of the unit object.

\subsection{What is the cosmash product in higher nilpotent cores good for?} In analogy with our bilinear product, one may consider the cosmash product of the \(n\)-nilpotent core of a semi-abelian category \(\X\) and ask for its role and properties. For \(n=3\), its values are still abelian, but it is biquadratic (actually of total degree \(3\) in a suitable sense) and factors through the reflection to the two-nilpotent instead of the abelian core in both variables. It thus probably preserves some non-abelian information of its factors, in spite of being itself abelian, similarly to the non-abelian tensor square of two-nilpotent groups, the structure of which was determined in~\cite{GoG}. Furthermore, it preserves coequalisers of reflexive graphs in both variables, and also filtered colimits if \(\X\) is a variety of universal algebras. So, it is a natural construction with good properties; however, we currently do not know whether it is related to the non-abelian tensor product \cite{Brown-Loday,dMVdL19.3}, to homology, or to any other topic relevant in the study of two-nilpotent groups or other algebraic objects; any result or even hint in this direction would thus be interesting and potentially useful!

\section*{Acknowledgements}
We would like to thank Jos\'e Manuel Casas for inciting us to study the Ganea term in a general categorical context. Moreover, the second and third authors gratefully acknowledge the former's longtime friend, Elisabeth Ullerich, for her warm hospitality—complete with delicious food—in September 2013. It was on her balcony, overlooking the breathtaking scenery of the Black Forest mountains, that we began to work on this paper. Last but not least, we wish to thank the referee, whose insightful comments and suggestions helped us to substantially improve the final version of the text.


\end{document}